\documentclass{amsart}
\usepackage{graphicx}
\usepackage{dsfont}
\usepackage{fancyhdr}
\usepackage{helvet}
\usepackage{nomencl} %for the symbol index
\makenomenclature

\usepackage[lined,algonl,boxed]{algorithm2e}

\newtheorem{thm}{Theorem}[section]
\newtheorem{cor}[thm]{Corollary}
\newtheorem{lem}[thm]{Lemma}

\theoremstyle{definition}
\newtheorem{defn}[thm]{Definition}
\theoremstyle{remark}
\newtheorem{rem}[thm]{Remark}

\numberwithin{equation}{section}
\newcommand{\norm}[1]{\left\Vert#1\right\Vert}
\newcommand{\abs}[1]{\left\vert#1\right\vert}
\newcommand{\babs}[1]{\Big \vert#1 \Big \vert}
\newcommand{\set}[1]{\left\{#1\right\}}
\newcommand{\parr}[1]{\left (#1\right )}
\newcommand{\brac}[1]{\left [#1\right ]}
\newcommand{\bigbrac}[1]{\Big [#1 \Big ]}
\newcommand{\ip}[1]{\left \langle #1 \right \rangle }
\newcommand{\Real}{\mathbb R}
\newcommand{\eps}{\varepsilon}
\newcommand{\To}{\longrightarrow}
\newcommand{\too}{\rightarrow}

\newcommand{\A}{\mathcal{A}}
\newcommand{\B}{\mathcal{B}}

\newcommand{\bbar}[1]{\overline{#1}}
\newcommand{\wt}[1]{\widetilde{#1}} %wide tilde
\newcommand{\wh}[1]{\widehat{#1}} %wide hat
\newcommand{\hh}[1]{\breve{#1}}% v sign, for approximated d^R_mu,nu %%{\hat{\vphantom{\rule{1pt}{5.5pt}}\smash{\hat{#1}}}} %double wide hat
\def \M{\mathcal{M}}
\def \N{\mathcal{N}}

\def \PL{PL} %the space of continuous piecewise linear function on mesh
\def \ncPL{ncPL} %the space of non-conforming piecewise linear function on mesh
\def \bfi{\textbf{\footnotesize{i}}} % traditional \i , which is a letter i without a dot

\def \C{\mathbb{C}} %complex plane
\def \eC{\widehat{\mathds{C}}} %complex plane
\def \D{\mathcal{D}} %unit disc
\def \Md{M_{\D}} %the disc Mobius trans
\def \Mec{M_{\eC}} %the full Mobius group
\def \T{\mathcal{T}} %optimal transportation plan

\def \vol{\mbox{\rm{\footnotesize{vol}}}} %surface volume element
\def \Vol{\mbox{\rm{vol}}} %surface volume element

\def \V{V}
\def \T{T}
\def \E{E}
\def \F{F}
\def \mM{\textbf{\textsf{M}}}
\def \mN{\textbf{\textsf{N}}}
\def \mV{\textbf{\textsf{V}}}
\def \mE{\textbf{\textsf{E}}}
\def \mF{\textbf{\textsf{F}}}
\def \mv{\textbf{\textsf{v}}}

\def \me{\textbf{\textsf{e}}}
\def \mf{\textbf{\textsf{f}}}%

\def \R{\mathbb{R}}

\def \Mbs{M\"{o}bius }
\def \d{\mbox{\bf{d}}}
\def \X{\mathcal{X}}
\def \Y{\mathcal{Y}}
\def \P{\mathcal{P}}
\def \MM{\mathbb{M}}
\def \GG{\mathbb{G}}
\def \stu{*u}
\def \CC{\mathcal{C}} %circle collection
\def \L{\mathcal{L}} %density points
\def \Q{\mathcal{Q}} %fixed volume to transform
\begin{document}

\title[Conformal Wasserstein Distance II: Comparing disk and sphere-type surfaces in polynomial time II, computational aspects]
{Conformal Wasserstein Distance:\\ II. Computational Aspects and Extensions}
\author{Y. Lipman, J. Puente, I. Daubechies }
\address{Princeton University}
%\email{}%
%
%\thanks{}%
%\subjclass{}%
%\keywords{}%
%
%%\date{}%
%%\dedicatory{}%
%%\commby{}%
%% ----------------------------------------------------------------
\begin{abstract}
This paper is a companion paper to
\cite{Lipman_Daubechies:2010:polytimesurfcomp}. We provide numerical
procedures and algorithms for
computing the alignment of and
distance between two disk type surfaces. We provide a convergence
analysis of the discrete approximation to the arising
mass-transportation problems. We furthermore generalize the
framework to support sphere-type surfaces, and prove a result
connecting this distance to local geodesic distortion. Lastly, we
provide numerical experiments on several surfaces' datasets and
compare to state of the art method.
\end{abstract}
\maketitle

\printnomenclature

% ----------------------------------------------------------------
\section{introduction and background}
\label{s:intro_and_background}

Alignment of surfaces plays a role in a wide range of scientific
disciplines. It is a standard problem in comparing different scans
of manufactured objects; various algorithms have been proposed for
this purpose in the computer graphics literature. It is  often also
a crucial step in a variety of problems in medicine and biology; in
these cases the surfaces tend to be more complex, and the alignment
problem may be harder. For instance, neuroscientists studying brain
function through functional Magnetic Resonance Imaging (fMRI)
typically observe several people performing identical tasks,
obtaining readings for the corresponding activity in the brain
cortex of each subject. In a first approximation, the cortex can be
viewed as a highly convoluted 2-dimensional surface. Because
different cortices are folded in very different ways, a synthesis of
the observations from different subjects must be based on
appropriate mappings between pairs of brain cortex surfaces, which
reduces to a family of surface alignment problems
\cite{Fischl99,Haxby09}. In another example, paleontologists
studying molar teeth of mammals rely on detailed comparisons of the
geometrical features of the tooth surfaces to distinguish species or
to determine similarities or differences in diet \cite{Jukka07}.

Mathematically, the problem of surface alignment can be described as
follows: given two 2-surfaces $\M$ \nomenclature{$\M$\
($\N$)}{Differentiable surfaces homeomorphic to a disk or a sphere
(will be clear from the context).} and $\N$, find a mapping $f:\M
\rightarrow \N$ that preserves, as best possible, ``important
properties'' of the surfaces. The nature of the ``important
properties'' depends on the problem at hand. In this paper, we
concentrate on preserving the geometry, i.e., we would like the map
$f$ to preserve intrinsic distances, to the extent possible. In
terms of the examples listed above, this is the criterion
traditionally selected in the computer graphics literature; it also
corresponds to the point of view of paleontologists studying tooth
surfaces. To align cortical surfaces, one typically uses the
Talairach method \cite{Lancaster00} (which relies on geometrically
defined landmarks and is thus geometric in nature as well), although
alignment based on functional correspondences has been proposed more
recently \cite{Haxby09}.

%Goal + approach
In \cite{Lipman_Daubechies:2010:polytimesurfcomp} a novel
%alignment
procedure between disk type surfaces
was proposed,
based on uniformization theory
and optimal mass transportation.
%, is proposed.
In a nutshell, the
method maps two surfaces $\M,\N$ to densities $\mu,\nu$ (interpreted
as mass densities) defined on the hyperbolic disk $\D=\set{z \mid
\abs{z}<1}$ \nomenclature{$\D$}{Unit disk $\set{z\mid \abs{z}<1}$},
their canonical uniformization space. (Apart from simplifying the
description of the surface, this also removes any effect of global
translations and rotations on the description of each individual
surface.)  The alignment problem can then be studied in the
framework of Kantorovich mass-transportation \cite{Kantorovich1942}
between these metric densities. Mass-transportation seeks to
minimize the ``average distance'' over which mass needs to be
``moved'' (in the most efficient such moving procedure) to transform
one mass density $\mu$ into another, $\nu$:
\begin{equation}\label{e:basic_Kantorovich_transporation}
    T_c(\mu,\nu) = \mathop{\inf}_{\pi \in \Pi(\mu,\nu)}\int_{\D \times \D}c(z,w)d\pi(z,w),
\end{equation}
where $c(z,w)\geq 0$ is a cost function, and $\Pi(\mu,\nu)$
\nomenclature{$\Pi(\mu,\nu)$}{the collection of probability measures
on $\D \times \D$ with marginals $ \mu$ and $\nu$ (resp.)}
is the
collection of probability measures on $\D \times \D$ with marginals
$ \mu$ and $\nu$ (resp.), that is, for $A\subset \D$, $B\subset \D$,
$\pi(A\times \D) = \mu(A)$ and $\pi(\D \times B) = \nu(B)$.

In our case the uniformizing metric density (or conformal factor)
corresponding to an initial surface is not unique, but is defined
only up to a M\"{o}bius transformation. Because a na\"{\i}ve
application of mass-transportation on the hyperbolic disk would not
possess the requisite invariance under M\"{o}bius transformations,
we generalize the mass-transportation framework, and replace the
cost function $d(z,w)$ traditionally used in defining the ``average
displacement distance'' by a cost that depends on $\mu$ and $\nu$,
$d^R_{\mu,\nu}(z,w)$, measuring the dissimilarity between the two
metric densities on $R$-neighborhoods of $z$ and $w$ (where $R$ is a
parameter that controls the size of the neighborhoods):
\begin{equation}\label{e:generalized_Kantorovich_transportation}
    T^R_d(\mu,\nu) = \mathop{\inf}_{\pi \in \Pi(\mu,\nu)}\int_{\D \times
    \D}d^R_{\mu,\nu}(z,w)d\pi(z,w).
\end{equation}
\nomenclature{$T^R_d(\mu,\nu)$}{Optimal transportation cost between
the densities $\mu,\nu$ with the measure-dependent
$d^R_{\mu,\nu}(z,w)$ cost function.} Introducing neighborhoods also
makes the definition less sensitive to noise in practical
applications. The optimal way of transporting mass in this
generalized framework defines a corresponding optimal way of
aligning the surfaces. This approach also allows us to define a new
distance, $\d^R(\M,\N)$, between surfaces; the average distance over
which mass needs transporting (to transform one metric density into
the other) quantifies the extent to which the two surfaces differ.

%The main contribution of the current paper
This paper contains three contributions that complement and extend
\cite{Lipman_Daubechies:2010:polytimesurfcomp}. The first of these
is to provide an
algorithm for approximating $\d^R(\M,\N)$ and to prove a convergence
result for this algorithm.  In order to state this goal more
precisely, we introduce some technicalities and notations now.

% ----------------------------------------------------------------
\subsection{Uniformization}

Uniformization theory for Riemann surfaces
\cite{Springer57,Farkas92} allows conformally flattening disk type
surfaces onto the unit disk $\phi:\M \rightarrow \D$
\nomenclature{$\phi$}{The uniformization map taking the surface $\M$
conformally to its canonical domain, the unit disk $\D$.} in $\C$, where
$\phi$ is the conformal flattening map, $\D =\{z \ | \ |z|<1\}$ is
the unit disk, and the disk coordinate system is denoted by
$z=x^1+\bfi x^2$. \nomenclature{$z=x^1+\bfi x^2 \ (w=y^1+\bfi
y^2)$}{The coordinates in uniformization space for surface $\M$
($\N$).} The surface's Riemannian metric $g$ \nomenclature{$g$ \
($h$)}{The metric tensor of the surface $\M$ ($\N$).} is then
pushed-forward to a diagonal metric tensor
$$
\wt{g} = \phi_* g = \mu^E(z)\, \delta_{ij}\, dx^i \otimes dx^j,
$$ \nomenclature{$\wt{g} \, (\wt{h})$}{The metric of surface $\M$ ($\N$) pushed forward by the uniformization map to the $\D$, that is
$\wt{g}=\phi_*g$ ($\wt{h}=\phi_*h$).}
where $\mu^E(z)>0$, Einstein
summation convention is used, and the subscript $*$ denotes the
``push-forward'' action; \nomenclature{$\mu^E(z) \ (\nu^E(w))$}{The
conformal factor of surface $\M$ ($\N$) w.r.t. the Euclidean metric
in the unit disk.}
the superscript $E$ stands for Euclidean.
The function $\mu^E$ can also be viewed as the
\emph{density function} of the measure $\Vol_\M$
\nomenclature{$\Vol_\M \ (\Vol_\N)$}{The area measure of surface
$\M$ ($\N$).} induced by the Riemann volume element:
for (measurable) $A \subset \M$,
\begin{equation}\label{e:volume_element}
    \Vol_\M(A) = \int_{\phi(A)} \mu^E(z) \, d\vol_E(z),
\end{equation}
where $d\vol_E(z)=dx^1\wedge dx^2$
is the Euclidean area element \nomenclature{$\Vol_E$}{The standard
Lebesque (Euclidean) measure in $\D$.}. For a second surface $\N$
with Riemannian metric $h$ we will denote its pushed-forward metric
on the uniformization disk $\D$ by $\wt{h} = \phi_* h = \nu^E(w)\,
\delta_{ij}\, dy^i \otimes dy^j,$ where the coordinates in the unit
disk are $w=y^1+\bfi y^2$.

We use the hyperbolic metric on the unit disk
$(1-|z|^2)^{-2}\delta_{ij} dx^i \otimes dx^j$ as a reference metric;
the surface density w.r.t. the hyperbolic metric (conformal scaling)
is
\begin{equation}\label{e:relation_hyperbolic_euclidean_density}
\mu^H(z):=(1-|z|^2)^{2}\,\mu^E(z)\, ,
\end{equation}
where the superscript $H$ stands for hyperbolic.
\nomenclature{$\mu^H \ (\nu^H)$}{The surface conformal factor w.r.t.
the hyperbolic metric in the unit disk. } \nomenclature{$\mu \
(\nu)$}{We use $\mu$ ($\nu$) either to represent the density $\mu^H$
($\nu^H$) or the measure $\Vol_\M$ ($\Vol_\N$). (The meaning will be
clear from the context.)}

We shall often drop this superscript: unless otherwise stated
$\mu=\mu^H$, and $\nu=\nu^H$ in what follows. The density function
$\mu=\mu^H$ satisfies
$$\Vol_\M(A) = \int_{\phi(A)} \mu(z)\, d\vol_H(z)\,,$$
where $d\vol_H(z)=(1-|z|^2)^{-2}\, d\vol_E(z)$. We will use the
notations $\mu,\nu$ also to represent the \emph{measures}
$\Vol_\M,\Vol_\N$ (resp.), where the exact meaning will be clear
from the context. \nomenclature{$\Vol_H$}{Hyperbolic area measure in
$\D$.}

The conformal mappings of $\D$ to itself are the disk-preserving
M\"{o}bius transformations, they constitute the group $\Md$ of
isometries of the hyperbolic disk. An arbitrary element $m\in\Md$ is
characterized by three real parameters: \nomenclature{$\Md$}{The
M\"{o}bius subgroup that preserves the unit disk.}
\begin{equation}\label{e:disk_mobius}
    m(z) = e^{\bfi \theta}\frac{z-a}{1-\bar{a}z}, \ a\in \D, \ \theta \in [0,2\pi).
\end{equation}

The pull-back $m^* \mu(z)$ and the push-forward $m_* \mu(w)$ of the
density $\mu$ by the M\"{o}bius transformation $m$ are given by the
formulas
\begin{equation}\label{e:pullback_of_metric_density_mu_by_mobius}
    m^*\mu(z) = \mu(m(z)),
\end{equation}
and
\begin{equation}\label{e:push_forward_of_metric_density}
    m_* \mu(w) = \mu(m^{-1}(w)),
\end{equation}
respectively. \nomenclature{$m^*\mu(z)\ (m_* \mu(w))$}{The pull-back
(push-forward) of the measure $\mu$ by the map $m$.}

It follows that checking whether or not two surfaces $\M$ and $\N$
are isometric, or searching for isometries between $\M$ and $\N$, is
greatly simplified by considering the conformal mappings from $\M$,
$\N$ to $\D$: once the (hyperbolic) density functions $\mu$ and
$\nu$ are known, it suffices to identify $m \in \Md$ such that
$\nu(m(z))$ equals $\mu(z)$ (or ``nearly'' equals, in a sense to be
made precise).

% ----------------------------------------------------------------
\subsection{Optimal volume transportation for surfaces}

To adapt the optimal transportation framework to the alignment of
surfaces, we use an isometry invariant cost function
$d^R_{\mu,\nu}(z,w)$ that is plugged into the transportation
framework (\ref{e:basic_Kantorovich_transporation}). This special
cost function $d^R_{\mu,\nu}(z,w)$ indicates the extent to which a
neighborhood of the point $z$ in $(\D,\mu)$, the (conformal
representation of the) first surface, is isometric with a
neighborhood of the point $w$ in $(\D,\nu)$, the (conformal
representation of the) second surface. Two definitions are in order:
1) the neighborhoods we will use, and 2) how we
%shall
characterize
the (dis)similarity of two neighborhoods, equipped with different
metrics.

For the neighborhoods, we take the hyperbolic geodesic disks
$\Omega_{z_0,R}$ of radius $R$, where we let $z_0$ range over $\D$,
but keep the radius $R>0$ fixed. \nomenclature{$\Omega_{z_0,R}$}{The
hyperbolic geodesic disk of radius $R$ centered at $z_0\in\D$.} The
following gives an easy procedure to construct these disks. If
$z_0=0$, then the hyperbolic geodesic disks centered at $z_0=0$ are
also ``standard'' (i.e. Euclidean) disks centered at 0:
$\Omega_{0,R} = \{z \,;\, |z|\leq r_R \}$, where
$r_R=\mbox{tanh}(R)$. The hyperbolic disks around other centers are
images of these central disks under M\"{o}bius transformations (=
hyperbolic isometries): setting $m(z)=(z-z_0)(1-z\bar{z_0})^{-1}$,
we have
\begin{equation}\label{e:neighborhood_def}
    \Omega_{z_0,R} = m^{-1}(\Omega_{0,R})\,.
\end{equation}

Next, the (dis)similarity of the pairs $\left(\Omega_{z_0,R}\,,\,
\mu\,\right)$ and $\left(\Omega_{w_0,R}\,,\, \nu\,\right)$ is
defined via pull-back of $\nu$ and using the standard induced norm
(see \cite{Lipman_Daubechies:2010:polytimesurfcomp} for more
details). The final cost function is achieved by taking the infimum
over all M\"{o}bius transformations $m$ such that $m(z)=w$:
\begin{equation}\label{e:d_mu,nu(z,w)_def}
    d^R_{\mu,\nu}(z_0,w_0) :=
\mathop{\inf}_{m \in \Md\,,\,m(z_0)=w_0}\int_{\Omega_{z_0,R}}
\,|\,\mu(z) - (m^*\nu)(z)\,|\, d\vol_H(z),
\end{equation}
where $d\vol_H(z)=(1-|z|^2)^{-2} \,dx^1\wedge dx^2$ is the volume
form for the hyperbolic disk.
\nomenclature{$d^R_{\mu,\nu}(z,w)$}{The cost function on
$\D\times\D$, dependent on measures $\mu,\nu$ and scale $R$. Used in
the surface transportation framework.}

As proved in \cite{Lipman_Daubechies:2010:polytimesurfcomp}
$d^R_{\mu,\nu}(\cdot,\cdot)$ is a metric on $\D$ and as a
consequence
\begin{equation}\label{e:the_distance}
    \d^R(\M,\N)=T^R_d(\mu,\nu)
\end{equation}
(with $T^R_d(\mu,\nu)$ as defined in \ref{e:generalized_Kantorovich_transportation})
defines a semi-metric in the space of disk-type surfaces.
\nomenclature{$\d^R(\M,\N)$}{The distance between disk-type surfaces $\M$ and
$\N$, based on neighborhoods with hyperbolic radius $R$.} To ensure that this is a metric rather than only a
semi-metric, we add an extra assumption, namely that no
(orientation-preserving) self-isometries exist within each of the
compared surfaces. For discussion and more detail we refer the
reader to \cite{Lipman_Daubechies:2010:polytimesurfcomp}.

%---------------------------------------------------------------------
\subsection{Overview}

We can now formulate a precise overview of the
%main results of this paper, in which we provide an
algorithm for approximating
$\d^R(\M,\N)$, and discuss its convergence properties.

In a nutshell, the key steps of the algorithm are: 1) approximate
uniformization for piecewise linear surface representations, 2)
discretize the continuous measures $\mu,\nu$ based on discrete
sample sets $Z=\set{z_i}_{i=1}^n,W=\set{w_j}_{j=1}^p$, obtaining
discrete measures by $\mu_Z,\nu_W$ (resp.), 3) approximate the
measure-dependent cost function $\hh{d}^R_{\mu,\nu}(z,w) \approx
d^R_{\mu,\nu}(z,w)$, and 4) calculate the discrete optimal
transportation cost $T_{\hh{d}}^R(\mu_Z,\nu_W)$ between the discrete
measures $\mu_Z,\nu_W$ based on the approximated cost function
$\hh{d}^R_{\mu,\nu}(z,w)$.

In the heart of our analysis we prove the convergence
$T_{\hh{d}}^R(\mu_Z,\nu_W) \too \d^R(\M,\N)$ as the ``mesh size'' of
the samplings
used in steps 2, 3 and 4,
tend to zero. More precisely, we define the \emph{fill
distance} $\varphi_g(Z)$ for the metric tensor $g$ and the sample
set $Z$ as \nomenclature{$\varphi_g(Z)$}{The fill distance of the
point set $Z=\set{z_i}$ w.r.t. metric tensor $g$.}
\begin{equation}\label{e:def_filling_distance}
    \varphi_g(Z)\,:=\,\sup \set{r>0\ \big| \ z \in \M:B_g(z,r)\cap Z =
\emptyset}~,
\end{equation}
where $B_g(z,r)$ \nomenclature{$B_g(z,r)$}{Geodesic open ball w.r.t.
metric tensor $g$, centered at $z$ and of radius $r$.} is the
geodesic open ball of radius $r$ centered at $z$. That is,
$\varphi_g(Z)$ is the radius of the largest geodesic ball that can
be fitted on the surface $\M$ without including any point of $Z$.
The smaller $\varphi_g(Z)$, the finer the sampling set. We prove the
following theorem:
\begin{thm}\label{thm:final_approx_discrete_transportation}
Let $\mu,\nu$ be Lipschitz continuous probability densities (w.r.t
the hyperbolic measure) defined over $\D$. Let $\M,\N$ be the
disk-type surfaces defined by the metric tensors
$g=\mu(z)\,(1-|z|^2)^{-2} \delta_{ij} dx^i\otimes dx^j, h=\nu(w)\,
(1-|w|^2)^{-2} \delta_{ij} dy^i\otimes dy^j$ (resp.), let $Z,W$ be
discrete sample sets on $\M,\N$ (resp.), $\set{p_k}$ samples in $\D$
used for numerical integration, and let $L$ be the number of
uniformly spaced points on the circle used to define $T_{\hh{d}}^R$
(see below). Then
\begin{eqnarray*}
\Big | \d^R(\M,\N) - T^R_{\hh{d}}(\mu_Z,\nu_W) \Big |
&=& \Big | T^R_{\hh{d}}(\mu,\nu) - T^R_{\hh{d}}(\mu_Z,\nu_W) \Big |\\
&\leq&
\omega_{d^R_{\mu,\nu}}\parr{2\max \parr{\varphi_g(Z),\varphi_h(W)}}
+ C_1\varphi_E  \parr{\set{p_k}} + C_2 \, L^{-1} \,+\, \epsilon_{\mbox{\tiny{\rm unif}}},
\end{eqnarray*}
where
$$\omega_{d^R_{\mu,\nu}}(t)=\sup_{d_\M(z,z') + d_\N(w,w') <
t}\abs{d^R_{\mu,\nu}(z,w)-d^R_{\mu,\nu}(z',w')}.$$
\nomenclature{$\omega_{f}(\cdot)$}{The modulus of continuity of the
function $f$.}
denotes the \emph{modulus of continuity} of the cost
function $d^R_{\mu,\nu}$, and $C_1,C_2$ are constants that depend only
upon $\mu,\nu,R$.
\end{thm}
Here $\varphi_E\parr{\set{p_k}}$ and $L^{-1}$ are two algorithm
parameters that can be made arbitrary small.
The error term $\epsilon_{\mbox{\tiny{\rm unif}}}$
concerns only the approximations made in the discrete uniformizations for each of the two
surfaces, separately; we'll come back to it below -- suffice it to say here
that it is much smaller than the other error terms, in practice.
Finally, it was
proved in \cite{Lipman_Daubechies:2010:polytimesurfcomp} that the
cost function $d^R_{\mu,\nu}(z,w)$ is uniformly continuous on
$\bbar{\D}\times \bbar{\D}$ and therefore
$\omega_{d^R_{\mu,\nu}}\parr{2\max
\parr{\varphi_g(Z),\varphi_h(W)}} \too 0$ as the fill-distances of
the sets $Z,W$ go to zero.

Two comments are in order: first, we believe that the cost function
$d^R_{\mu,\nu}(z,w)$ is Lipschitz rather than just uniformly
continuous. If this is the case, then our analysis guarantees linear
convergence of the algorithm
(leaving aside the $\epsilon_{\mbox{\tiny{\rm unif}}}$ term).
We leave checking the
precise
regularity of
the cost function to future work. Second,
we should discuss in more detail $\epsilon_{\mbox{\tiny{\rm unif}}}$, containing
%we note that our
%convergence analysis does not include
the errors produced from the
discrete uniformization.
One typically starts from
%More explicitly, we start with
a piecewise
%linear
flat approximation of the surfaces $\M,\N$,
%typically
given by
triangle meshes with a very fine mesh size, providing much finer
sampling than the $Z$ or $W$ used in steps 2 through 4.
(For instance, in our numerical
computations, the parameter $L$ and the sample sets $Z,W$ were
picked so that $L^{-1}$, $\varphi_g(Z),\,\varphi_h(W)$ had magnitude
.02,.06,.06 (resp.); the mesh size in the original triangulation of
$\M,\,\N$ is of order .01.) How to
construct
%use these to provide
discrete
approximations of the uniformization
of the surfaces, starting from these approximations to $\M,\, \N$, is a research area in its own right,
and several different methods have been proposed
\cite{Gu03,Polthier05,Springborn:2008:CET:1399504.1360676}; in our
work we adopt the approach of \cite{Polthier05,Lipman:2009:MVF}. The
error
$\epsilon_{\mbox{\tiny{\rm unif}}}$
contributed by this component of
our algorithm
%the approximation strategy
is
governed by the difference between the ``true'' $\mu,\nu$ and the
$\mu^{approx},\nu^{approx}$ stemming from the discrete triangulation
approximation, followed by the discrete uniformization, and is
bounded by (using the triangle inequality proved in Theorem 3.11 in
\cite{Lipman_Daubechies:2010:polytimesurfcomp})
$$
\epsilon_{\mbox{\tiny{unif}}} \leq
T_d^R(\mu,\mu^{approx})+T^R_d(\nu,\nu^{approx}).$$
We expect this difference to be proportional to the triangulation
mesh size, and
%the component of the error it causes
to be negligible
compared to the errors we analyze explicitly in this paper. Since
the convergence analysis of discrete uniformization has not
settled yet into its final form, and given the much smaller size of this
component of the error (both expected and borne out by numerical
experiments), we have opted here to view the discrete uniformization
as a ``black box'', the analysis of which is outside the scope of
this paper, and to neglect this part of the error. We concern
ourselves here with the error made by our algorithm in the
approximation to $T_d^R(\mu^{approx},\nu^{approx})$, namely
with
$$\Big|T_d^R(\mu^{approx},\nu^{approx})-T_d^R(\mu_Z,\nu_W)\Big|\,.$$

We now turn to the other contributions made by this paper.
In an earlier version of the paper,
Theorem \ref{thm:final_approx_discrete_transportation}
was the main result. Interesting questions and challenges by the reviewers led us to investigate extensions and further mathematical properties of our construction; the results are formulated as two further contributions.

The first of these is a generalization of the framework above to other
types of surfaces. We show how a similar distance $\d^A(\M,\N)$
\nomenclature{$\d^A(\M,\N)$}{The distance between shpere-type (or
disk-type) surfaces $\M$ and $\N$, based on neighborhoods with area
$A$} can be defined for genus-zero, or sphere-type surfaces. This
involves some new ideas, since the absence of a distance function on
the sphere $S^2$ that would be invariant under all conformal maps
from $S^2$ to itself, implies that the definitions of the
neighborhoods $\Omega_{z_0,R}$ cannot simply be copied from the case
for disk-type surfaces.

The final contribution of this paper concerns possible connections
between the distance $\d^R(\M,\N)$ and the notion of geodesic
distortion. Although there is certainly much more to be said upon
this topic than we do here, we do present a first result, showing
that if the distance $\d^R(\M,\N)$ between two disk-like surfaces is
small, then the two surfaces are \emph{locally} near-isometric. More
precisely, we prove the following:
\begin{thm}\label{thm:stability_4.1}
%If $\d^R(\M,\N)$ is small then the two surfaces are locally
%isometric:
Let $\M$ and $\N$ be differentiable disk-like surfaces. If
$\d^R(\M,\N)$ is sufficiently small, then we can cover $\M$ (minus
an arbitrarily small boundary layer) with patches $\Omega_{z_0,R}$
and define mappings $f^{z_0}:\M\too\N$ (M\"{o}bius transformations)
such that for all $z_1,z_2\in \M$ (not very close to the boundary)
with geodesic distance $d_g(z_1,z_2)\leq r(R)$, $r(R)>0$, there
exists a patch $\Omega_{z_0,R}$ such that
$z_1,z_2\in\Omega_{z_0,R}$, and
%$z,z'\in\Omega_{z_0,R} \cap \M$, $f^{z_0}\parr{\Omega_{z_0,R}
%\cap \M} \subset \Omega_{w_0,R} \cap \N$ and:
\begin{equation}\label{e:lower_bound}
    \parr{1 - C_2 \d^R(\M,\N)^{1/3}} d_g(z_1,z_2)\leq d_h(f^{z_0}(z_1),f^{z_0}(z_2)) \leq \parr{1+C_1 \d^R(\M,\N)^{1/3}}
    d_g(z_1,z_2),
\end{equation}
where $d_g(z_1,z_2),d_h(w_1,w_2)$ are the geodesic distances of
$z_1,z_2\in\M$ and $w_1,w_2\in\N$, respectively, and $C_1,C_2>0$ are
constants independent of the choice of $z_1,z_2$.
\end{thm}
\nomenclature{$d_g(z_1,z_2)$}{The geodesic distance between $z_1$
and $z_2$ based on the metric $g$.}

%In a later part of the paper we indicate how the above ideas can be
%generalized to other types of surfaces. In particular we show how
%the distance can be defined for genus zero, or sphere-type,
%surfaces. We also present a theoretical result directly connecting
%this distance to the notion of local geodesic distortion. Lastly, in
%order to validate the usefulness of our method, we report results of
%the method applied to various benchmark data sets and comparison to
%a state-of-the-art method.

%---------------------------------------------------------------------
\subsection{Related work}
%related works
The approach taken in this paper is related to the computer graphics
constructions in \cite{Lipman:2009:MVF}, which rely on the
representation of isometries between topologically equivalent
simply-connected surfaces by M\"{o}bius transformations between
their uniformization spaces, and which exploit that 1) the
M\"{o}bius group has small dimensionality (e.g. 3 for disk-type
surfaces and 6 for sphere-type) and 2) changing the metric in one
piece of a surface has little influence on the uniformization of
distant parts. These two observations lead, in
\cite{Lipman:2009:MVF}, to fast and effective algorithms to identify
near-isometries between differently deformed versions of a surface.
In our present context, these same observations lead to a simple
algorithm for surface alignment, reducing it to a linear programming
problem.

Other distances between surfaces have been used recently for several
applications \cite{memoli07}. A prominent mathematical approach to
define distances between surfaces considers the surfaces as special
cases of \emph{metric spaces}, and uses then the Gromov-Hausdorff
(GH) distance between metric spaces \cite{Gromov06}. The GH distance
between metric spaces $X$ and $Y$ is defined through examining all
the isometric embedding of $X$ and $Y$ into (other) metric spaces;
although this distance possesses many attractive mathematical
properties, it is inherently hard computationally
\cite{memoli05,BBK06}. For instance, computing the GH distance is
equivalent to a non-convex quadratic programming problem; solving
this directly for correspondences is equivalent to integer quadratic
assignment, and is thus NP-hard \cite{Cela98}. In addition, the
non-convexity implies that the solution found in practice may be a
local instead of a global minimum, and is therefore not guaranteed
to give the correct answer for the GH distance. The distance between
surfaces as we define in
\cite{Lipman_Daubechies:2010:polytimesurfcomp} does not have these
shortcomings because the computation of the distance between
surfaces using this approach can be recast as a linear program, and
can therefore be implemented using efficient polynomial algorithms
that are moreover guaranteed to converge to the correct solution.

In \cite{memoli07}, the GH distance of \cite{memoli05} is
generalized by introducing a quadratic mass transportation scheme to
be applied to metric spaces equipped with a measure (mm spaces); the
computation of this Gromov-Wasserstein distance for mm spaces is
somewhat easier and more stable to implement than the original GH
distance \cite{memoli07}. A crucial aspect in which our work differs
from \cite{memoli07} is that, in contrast to the (continuous)
quadratic programming method proposed in \cite{memoli07} to compute
the Gromov-Wasserstein distance between mm spaces, our conformal
approach leads to a convex (even linear) problem, solvable via a
linear programming method.

It is worth mentioning that optimal mass transportation has been
used in the engineering literature
as well,
to define interesting metrics
between images; in this context the metric is often called the
Wasserstein distance. The seminal work for this image analysis
approach is the paper by Rubner et al.~\cite{Rubner2000-TEM}, in
which images are viewed as discrete measures, and the distance is
called appropriately the ``Earth Mover's Distance''.

Another related method is presented in the papers of Zeng et al.
\cite{Gu2008_a,Gu2008_b}, which also use the uniformization space to
match surfaces. Our work differs from that of Zeng et al. in that
they use prescribed feature points (defined either by the user or by
extra texture information) to calculate an interpolating harmonic
map between the uniformization spaces, and then define the final
correspondence as a composition of the uniformization maps and this
harmonic interpolant. This procedure is highly dependent on the
prescribed feature points, provided as extra data or obtained from
non-geometric information. In contrast, our work does not use any
prescribed feature points or external data, and makes use of only
the geometry of the surface; in particular we utilize the conformal
structure itself to define deviation from (local) isometry.

%arrangement of article
\subsection{Organization}
%Our paper is organized as follows:
Section
\ref{s:the_discrete_case_implementation} presents the main steps for
the discretization of the continuous case and provides algorithmic
aspects for the alignment procedure. Section
\ref{s:generalization_to_sphere_type} generalizes the method to
sphere-type surfaces. Section \ref{s:stability} provides a
theoretical result connecting our distance directly to local
geodesic distortion. Section \ref{s:examples} presents experimental
validation of the algorithms and concludes;
in particular, we report results of
the method applied to various benchmark data sets and
provide a
comparison to
a state-of-the-art method.
This paper also contains
four appendices: A) contains few approximation results used by our
algorithm, B) contains background on the discrete conformal mapping
we use, C) contains proofs of some properties of the linear program
solution, and D) contains the approximation analysis of the discrete
optimal transport cost to its continuous counterpart.

%in Section \ref{s:prelim} we
%briefly recall some facts about uniformization and optimal mass
%transportation that we shall use, at the same time introducing our
%notation. Section \ref{s:optimal_vol_trans_for_surfaces} contains
%the main results of this paper, constructing the distance metric
%between disk-type surfaces, in several steps. Section
%\ref{s:the_discrete_case_implementation} discusses various issues
%that concern the numerical implementation of the framework we
%propose; Section \ref{s:examples} illustrates our results with a few
%examples.

%% ----------------------------------------------------------------
%\section{Overview}
%\label{s:overview}
%\input{overview}
%
%%% ----------------------------------------------------------------
%\section{Optimal volume transportation for surfaces}
%\label{s:optimal_vol_trans_for_surfaces}
%\input{optimal_start_yl}
%\input{optimal2_alt}

% ----------------------------------------------------------------
\section{Algorithm for comparing disk-type surfaces and analysis}
\label{s:the_discrete_case_implementation}
Transforming the theoretical framework discussed above into an
algorithm requires several steps of approximation. Our general plan
is to recast the transportation
eq.~(\ref{e:generalized_Kantorovich_transportation}) as a linear
programming problem between discrete measures. The steps of our
algorithm are as follows:\\
%\begin{tabular}{cc}
  \textit{Preprocess:}  approximating the surfaces' uniformization, \\
  \textit{Step 1:}   discretizing the resulting continuous measures, \\
  \textit{Step 2:}   approximating the cost function $d^R_{\mu,\nu}(\cdot,\cdot)$, \\
  \textit{Step 3:}   solving a linear programming problem to achieve the final
approximation of the distance, and the optimal transportation plan
(correspondences). \\
  \textit{Step 4 (optional):}  extract a consistent set of
  correspondences.\\
%\end{tabular}

In the following we describe in detail each of these steps; we also
provide a convergence analysis for steps 1-3, but not for the
{\em preprocess step}. (As explained in the introduction the convergence of the
approximated uniformization is not the focus of this paper, and we
consider it as a "black box".) For the sake of completeness, and as
a guide to readers who would like to implement the algorithm,
we nevertheless include a description of this
part in Appendix~\ref{a:appendix B}.

%-------------------------------------------------------------------
\subsection{\emph{Step 1}: Discretizing continuous measures}
\label{ss:discretizing_discrete_measures}

\nomenclature{$Z \ (W)$}{The discrete sets of points used to
discretize the measures $\mu$ ($\nu$).}

In this subsection we
indicate how to construct the discrete
measures $\mu_Z,\nu_W$ \nomenclature{$\mu_Z$\ ($\nu_W$)}{The
discrete measures that approximate $\mu$ ($\nu$) defined by the
discrete set of points $Z$ ($W$).} used in further steps.

Given the measure $\mu=\Vol_\M$ on $\D$, we discretize it by first
distributing $n$ points $Z=\set{z_i}_{i=1}^n$ ``uniformly w.r.t.
$\mu$''. Details on the particular algorithm we used for sampling
are provided in Appendix~\ref{a:appendix B} (we use the same
technique for the sets $Z,W$ described there). For $i=1,\ldots,n$,
we define the sets $\set{\Xi_i}_{i=1}^n$ to be the Voronoi cells
corresponding to $z_i \in Z$ defined by the metric of $\M$; this
gives a partition of $\D$ into disjoint sets, $\D=\cup_{i=1}^n
\Xi_i$. For a more detailed definition of Voronoi cells as-well as
properties of the discrete measures see
Appendix~\ref{a:approximating_optimal_transport}. Next, define the
discrete measure $\mu_Z$ as a superposition of delta measures
localized in the points of $Z$, with weights given by the areas of
$\Xi_i$, i.e.
\begin{equation}\label{e:mu_Z}
    \mu_Z \,=\,\sum_{i=1}^n\,\mu_i \delta_{z_i},
\end{equation}
with $\mu_i:=\Vol_\M(\Xi_i)=\int_{\Xi_i}\mu(z)d\vol_H$. Similarly we
denote by $W=\{w_j\}_{j=1}^p$, $\set{\Upsilon_j}_{j=1}^p$, $\nu_W$,
and $\nu_j:=\Vol_\N(\Upsilon_j)$ the corresponding quantities for
the measure
$\nu=\Vol_\N$.\nomenclature{$\Xi_i$\,($\Upsilon_j$)}{Voronoi cells
on $\D$ based on the point samples $Z$ ($W$) and the metric of the
surface $\M$ ($\N$).} \nomenclature{$\delta_{z_i}$}{Dirac measure
concentrated at point $z_i$.}

We shall always assume that the surfaces $\M$ and $\N$ have the same
area, which, for convenience, we can take to be 1. It then follows
that the discrete measures $\mu_Z$ and $\nu_W$ have equal total mass
(regardless of whether $n=p$ or not). The approximation algorithm
will compute optimal transport for the discrete measures $\mu_Z$ and
$\nu_W$; the corresponding discrete approximation to the distance
between $\M$ and $\N$ is then given by $T^R_d(\mu_Z,\nu_W)$.

The convergence analysis we present will be in terms of the
\emph{fill distance} $\varphi_g(Z),\varphi_h(W)$ defined in the
introduction. Note that our analysis will work with any point sample
sets as long as their fill distances converge to zero.

% ----------------------------------------------------------------
\subsection{\emph{Step 2}: approximating the cost function $d^R_{\mu,\nu}$.}
In order to approximate $T_d^R(\mu_Z,\nu_W)$ we need to approximate
the cost function $d^R_{\mu,\nu}(z,w)$ between pairs of points
$(z_i,w_j)\in Z\times W$.

Applying (\ref{e:d_mu,nu(z,w)_def}) to the points $z_i$, $w_j$ we
have:
\begin{equation}\label{e:d_mu,nu(z_i,w_j)}
    d^R_{\mu,\nu}(z_i,w_j) = \min_{m(z_i)=w_j}\int_{\Omega_{z_i, R}}\Big|\,\mu(z)-\nu(m(z))\,\Big |\,d\vol_H.
\end{equation}

To obtain $d^R_{\mu,\nu}(z_i,w_j)$ we will thus need to approximate
integrals over hyperbolic disks of radius $R$, which is done via a
separate approximation procedure, set up once and for all in a
preprocessing step at the start of the algorithm.

By using a \Mbs transformation $\widetilde{m}$ such that
$\widetilde{m}(0)=z_0$, and the identity
\[
\int_{\Omega_{z_0,R}}\Big|\,\mu(z) - \nu(m (z))\,\Big| \,d\vol_H(z)
= \int_{\Omega_{0,R}}\Big|\,\mu(\widetilde{m}(u)) - \nu(m \circ
\widetilde{m} (u))\,\Big| \,d\vol_H(u)~,
\]
we can reduce the integrals over the hyperbolic disks
$\Omega_{z_i,R}$ to integrals over a hyperbolic disk $\Omega_{0,R}$
centered around zero.

%In order to (approximately) compute integrals over $\Omega_0
%=\Omega_{0,R}=\{z |\ |z|\leq r_R\}$, we first pick a positive
%integer $K$ and distribute centers $p_k,\,k=1,...,K $ in $\Omega_0$.
%We then decompose $\Omega_0$ into Voronoi cells $\Delta_k$
%corresponding to the $p_k$ (we use Euclidean metric although we note
%that using Hyperbolic metric would be somewhat more effective here),
%obtaining $\Omega_0 = \cup_{k=1}^K \Delta_k$; see Figure
%\ref{fig:integration_pnts} (note that these Voronoi cells are
%completely independent of those used in
%\ref{ss:discretizing_discrete_measures}.)
%\begin{figure}
%  % Requires \usepackage{graphicx}
%  \includegraphics[width=0.7\columnwidth]{figures/int_rules_100_300.png}\\
%  \caption{The integration centers and their corresponding Voronoi cells used
%for calculating the integration weights for the discrete quadrature.
%Left: $100$ centers; Right: $300$.}\label{fig:integration_pnts}
%\end{figure}
%To approximate the integral of a continuous function $f$ over
%$\Omega_{0,R}$ we then use the rectangle-type rule
%\[
%\int_{\Omega_0}\, f(z)\, d\vol_H(z) \approx \sum_k \,
%\brac{\int_{\Delta_k}d\vol_H(z)}\,f(p_k) = \sum_k \,\alpha_k f(p_k)
%\]
%where $\alpha_k = \int_{\Delta_k}d\vol_H(z)$.

To approximate the integral of a continuous function $f$ over
$\Omega_{0,R}$ we then use a rectangle-type quadrature
\[
\int_{\Omega_0}\, f(z)\, d\vol_H(z) \approx \sum_k \,\alpha_k f(p_k)
\]
where $p_k\in \Omega_{0,R},\alpha_k \in \R, k=1..K$, are the centers
and coefficients (resp.) of the quadrature. The coefficients
$\alpha_k$ are defined as the hyperbolic area of the Euclidean
Voronoi cells $\Delta_k$ corresponding to the centers $p_k$.

We thus have the following approximation:
\begin{eqnarray}
d^R_{\mu,\nu}(z_i,w_j)&=&
\min_{m(z_i)=w_j}\int_{\Omega_{z_i,R}}\Big|\,\mu(z) - \nu(m
(z))\,\Big|
\,d\vol_H(z)\nonumber\\
&=& \min_{m(z_i)=w_j}\int_{\Omega_{0,R}}\Big|\,\mu(\widetilde{m}_i
(z)) -
\nu(m(\widetilde{m}_i (z)))\,\Big|\,d\vol_H(z)\nonumber\\
&\approx&   \min_{m(z_i)=w_j}  \sum_k \,\alpha_k
\,\Big|\,\mu(\widetilde{m}_i (p_k)) - \nu(m(\widetilde{m}_i
(p_k)))\,\Big|~,\label{e:discrete_approx__d_mu_nu(z_i,w_j)}
\end{eqnarray}
where the \Mbs transformations $\widetilde{m}_i$, mapping 0 to
$z_i$, are selected as soon as the $z_i$ themselves have been
picked, and remain the same throughout the remainder of the
algorithm.

Let us denote this approximation by $$\wh{d}^R_{\mu,\nu}(z_i,w_j) =
\min_{m(z_i)=w_j}\sum_k \,\alpha_k \,\left|\,\mu(\widetilde{m}_i
(p_k)) - \nu(m(\widetilde{m}_i (p_k)))\,\right|.$$

It can be shown that picking a set of centers $\set{p_k}$ with
Euclidean fill-distance $\varphi_E(\set{p_k}) = h>0$ (that is, we
use the Euclidean metric to define the fill distance of the set
$\set{p_k}$) leads to an $O(h)$ approximation; in
Appendix~\ref{a:appendix A} we prove:
\begin{thm}\label{thm:convergence_of_numerical_quadrature}
For Lipschitz continuous $\mu,\nu$,
\[
\abs{d^R_{\mu,\nu}(z_i,w_j)- \wh{d}^R_{\mu,\nu}(z_i,w_j)}\leq
C\,\varphi_E\left( \set{p_k} \right)~,
\]
where the constant $C$ depends only on $\mu,\nu,R$.
\end{thm}

In practice, the minimization over $M_{\D,z_i,w_j}$ (the set of all
\Mbs transformations that map $z_i$ to $w_k$) in the computation of
$\wh{d}^R_{\mu,\nu}$ is discretized as well: instead of minimizing
over all $M_{\D,z_i,w_j}$, we minimize over only the \Mbs
transformations
$\left(m_{z_i,w_j,2\pi\ell/L}\right)_{\ell=0,1,..,L-1}$, defined by
\begin{equation}\label{e:def_m_zi_wi_2pi l_div_L}
    m_{z_i,w_j,2\pi\ell/L} = \wt{m}_j \circ \mathfrak{R}_\ell \circ \wt{m}_i^{-1},
\end{equation}
with $\wt{m}_i$ as defined above, $\mathfrak{R}_\ell(z)=e^{\bfi
2\pi\ell/L}z$ , $L$ a parameter that reflects how many points we use
to discretize $[0,2\pi)$, and $\wt{m}_j \in \Md$ an arbitrary but
fixed \Mbs map that takes $0$ to $w_j$.

%where we use the following characterization of the disk-M\"{o}bius
%group:
%\begin{lem}\label{lem:a_and_tet_formula_in_mobius_interpolation}
%For any $z_0,w_0 \in \D$, the set $M_{D,z_0,w_0}$ constitutes a
%$1$-parameter family of disk M\"{o}bius transformations,
%parametrized continuously over $S^1$ (the unit circle). More
%precisely, every $m \in M_{D,z_0,w_0}$ is of the form
%
%\begin{equation}\label{e:a_of_mobius}
%% \nonumber to remove numbering (before each equation)
%m(z)= \tau\,\frac{z-a}{1-\overline{a}z}~,~~\mbox{ {\rm{with} }}~~ a
%= a(z_0,w_0,\sigma) :=\frac{z_0-w_0
%\,\overline{\sigma}}{1-\overline{z_0}\,w_0\,\overline{\sigma}}
%~~~\mbox{{\rm and }}~~ \tau = \tau(z_0,w_0,\sigma) := \sigma
%\frac{1- \overline{z_0} \,w_0 \,\overline{\sigma}}
%{1-z_0\,\overline{w_0}\, \sigma},
%\end{equation}
%where $\sigma \in S_1:=\{z \in \C\,;\,|z|=1\}$ can be chosen freely.
%\end{lem}

Taking this into account as well, we have thus
\begin{eqnarray}
d^R_{\mu,\nu}(z_i,w_j)\approx \hh{d}^R_{\mu,\nu}(z_i,w_j) :=
\min_{\ell=1..L} \sum_k \,\alpha_k \,\Big|\,\mu(\widetilde{m}_i
(p_k)) - \nu(m_{z_i,w_j,2\pi \ell/L}(\widetilde{m}_i
(p_k)))\,\Big|~;\label{e:discrete_approx__d_mu_nu(z_i,w_j)_bis}
\end{eqnarray}
as we prove in Appendix~\ref{a:appendix A} the error made in
approximation (\ref{e:discrete_approx__d_mu_nu(z_i,w_j)_bis}) is

\begin{thm}\label{thm:final_approx_to_d^R_mu,nu}
For Lipschitz continuous $\mu,\nu$,
\[
\abs{\,d^R_{\mu,\nu}(z_i,w_j)- \hh{d}^R_{\mu,\nu}(z_i,w_j)} \leq
C_1\,\varphi_E\left( \set{p_k} \right) + C_2 L^{-1}~,
\]
where the constants $C_1,C_2$ depends only on $\mu,\nu,R$.
\end{thm}

% ----------------------------------------------------------------
\subsection{\emph{Step 3}: Solving a linear program.}
\label{ss:step_3_linear_programming}

We now have in place all the ingredients to formulate the final
linear programming problem, the solution of which approximates the
distance $\d^R(\M,\N)$. The final step is to solve a discrete
optimal transportation problem between the discrete measures $\mu_Z$
and $\nu_W$ with the approximated cost function
$\hh{d}^R_{\mu,\nu}(z_i,w_j)$:
\begin{equation}\label{e:discrete_kantorovich}
  \sum_{i,j}\hh{d}^R_{ij}\pi_{ij} \rightarrow \min
\end{equation}
\begin{equation}\label{e:discrete_kantorovich_CONSTRAINTS}
\begin{array}{l}
  \left \{
  \begin{array}{l}
    \sum_i \pi_{ij} = \nu_j \\
    \sum_j \pi_{ij} = \mu_i \\
    \pi_{ij} \geq 0
  \end{array}
  \right . ,
  %\sum_i \pi_{ij} = \nu_j & \\
%  \sum_j \pi_{ij} = \mu_i &\\
%  \pi_{ij} \geq 0.  &
\end{array}
\end{equation}
where $\mu_i=\mu(\Xi_i)$ and $\nu_j=\nu(\Upsilon_j)$, and
$\hh{d}^R_{ij} = \hh{d}^R_{\mu,\nu}(z_i,w_j)$.

The optimal transportation plan $\pi^*$ then furnishes our final
approximation: $T_{\hh{d}}(\mu_Z,\nu_W) = \sum_{ij}
\hh{d}^R_{ij}\pi^*_{ij}$. The approximation result will be expressed
in terms of the modulus of continuity of our cost function:
$\omega_{d^R_{\mu,\nu}}$. Our result will use the following
regularity theorem of mass transportation, proved in Appendix
\ref{a:approximating_optimal_transport},
\begin{thm}\label{t:convergence_optimal_cost}
Suppose $c:\X\times\Y \To \Real_+$ is a continuous function, with
$\X,\Y$ compact complete separable metric spaces, $S$ and $T$ are
sample sets in $\X,\Y$ (resp.), $\mu,\nu$ are probability measures
on $\X,\Y$.
\\(A) if $c$ is uniformly
continuous then $$\T_c(\mu_S, \nu_T) \too \T_c(\mu,\nu), \ \ \ as \
h \too 0,$$\\ (B) if $c$ is Lipschitz continuous with a constant
$\lambda$, then
$$\abs{\T_c(\mu,\nu) - \T_c(\mu_S,\nu_T) } < 2\lambda h,$$
where, $h=\max \set{\varphi_\X(S),\varphi_\Y(T)}$, and $\mu_S,\nu_T$
are as defined similarly to (\ref{e:mu_Z}). (See
Appendix~\ref{a:approximating_optimal_transport} for precise
definition.)
\end{thm}

Our main approximation result is as follows:

\begin{thm}\label{thm:final_approx_discrete_transportation}
Let $\mu,\nu$ be Lipschitz continuous probability densities (w.r.t.
the hyperbolic measure) defined over $\D$. Let $\M,\N$ be the
disk-type surfaces defined by the metric tensors
$g=\mu(z)\,(1-|z|^2)^{-2} \delta_{ij} dx^i\otimes dx^j, h=\nu(w)\,
(1-|w|^2)^{-2} \delta_{ij} dy^i\otimes dy^j$ (resp.). Let $\pi^*$ be
the minimizer of the linear program defined by
\textrm{(\ref{e:discrete_kantorovich})-(\ref{e:discrete_kantorovich_CONSTRAINTS})},
then
$$\Big | \d^R(\M,\N) - T_{\hh{d}}(\mu_Z,\nu_W)  \Big | \leq
\omega_{d^R_{\mu,\nu}}\parr{2\max \parr{\varphi_g(Z),\varphi_h(W)}}
+ C_1\varphi_E  \parr{\set{p_k}} + C_2 \, L^{-1},$$ where
$\omega_{d^R_{\mu,\nu}}$ denotes the modulus of continuity of the
function $d^R_{\mu,\nu}$, $C_1,C_2$ are constants dependent only
upon $\mu,\nu,R$.
\end{thm}
\begin{proof}
First,
\begin{align*}
\Big | \d^R(\M,\N) - T_{\hh{d}}(\mu_Z,\nu_W) \Big | & \leq \Big |
\d^R(\M,\N) - T_d(\mu_Z,\nu_W) \Big | + \Big | T_{d}(\mu_Z,\nu_W) -
T_{\hh{d}}(\mu_Z,\nu_W) \Big | \\ & = I + II.
\end{align*}
where $T_d(\mu_Z,\nu_W)$ is the optimal transport cost between the
discrete measures $\mu_z,\nu_W$ using the exact cost function
$d^R_{\mu,\nu}(z_i,w_j)$.

In Appendix \ref{a:approximating_optimal_transport} (Theorem
\ref{thm:main_theorem}) we prove that
$$ I \leq \omega_{d^R_{\mu,\nu}}(2
\max\set{\varphi_g(Z),\varphi_h(W)});$$ the result is proved in the
more general context of compact complete separable metric spaces; we
believe this result may be useful, independently of the remainder of
this paper, to approximate optimal transport cost in more general
contexts (see Appendix \ref{a:approximating_optimal_transport} for
more details).

To bound $II$, denote by $\pi'_{ij}$ the optimal plan in
$T_{\hh{d}}(\mu_Z,\nu_W)$, then,
\begin{align*}
T_{d}(\mu_Z,\nu_W) - T_{\hh{d}}(\mu_Z,\nu_W) & = \inf_\pi
\sum_{i,j}\pi_{i,j} d^R_{i,j} -  \sum_{i,j}\pi'_{i,j} \hh{d}^R_{i,j}
\\ & \leq \sum_{i,j} \pi'_{i,j} \parr{d^R_{i,j} - \hh{d}^R_{i,j}} \\
& \leq C_1 \, \varphi_E(\set{p_k}) + C_2 \, L^{-1},
\end{align*}
where in the last inequality we used Theorem
\ref{thm:final_approx_to_d^R_mu,nu}. The symmetric inequality can be
achieved similarly. This completes the proof.
\end{proof}

It is proved in \cite{Lipman_Daubechies:2010:polytimesurfcomp} that
$d^R_{\mu,\nu}(\cdot,\cdot)$ is uniformly continuous on
$\bbar{\D}\times \bbar{\D}$. Therefore, the above theorem implies
convergence of our discrete approximation. More specifically, our
approximation will converge like the modulus of continuity of
$d^R_{\mu,\nu}(\cdot,\cdot)$; remember that for uniformly continuous
functions $f$, the modulus of continuity satisfies $\lim_{r\too
0}\omega_f(r) = 0$. As mentioned in the introduction, we believe
that $d^R_{\mu,\nu}(\cdot,\cdot)$ is actually Lipschitz continuous,
in that case the above theorem actually implies linear convergence
rate. We leave the question of higher regularity of the cost
function $d^R_{\mu,\nu}$ to future work.

In the remaining part of this subsection we discuss some variations
and properties of the linear program formulation
eq.(\ref{e:discrete_kantorovich})-(\ref{e:discrete_kantorovich_CONSTRAINTS}).
In practice, surfaces are often only partially isometric.
Furthermore, the sampled points may also fail to have a good
one-to-one and onto correspondence (i.e. there typically are some
points in both $Z$ and $W$ that do not correspond well to any point
in the other set). In these cases it is desirable to allow the
algorithm to consider transportation plans $\pi$ with marginals
\emph{smaller or equal} to $\mu$ and $\nu$. Intuitively this means
that we allow that only some fraction of the mass is transported and
that the remainder can be ``thrown away''. This leads to the
following formulation:
\begin{equation}\label{e:discrete_PARTIAL_kantorovich}
\sum_{i,j}d_{ij}\pi_{ij} \rightarrow \min
\end{equation}
\begin{equation}\label{e:discrete_PARTIAL_kantorovich_CONSTRAINTS}
\begin{array}{l}
  \left \{ \begin{array}{c}
             \sum_i \pi_{ij} \leq \nu_j \\
             \sum_j \pi_{ij} \leq \mu_i \\
             \sum_{i,j} \pi_{ij} = \Q  \\
             \pi_{ij} \geq 0
           \end{array}
   \right .
\end{array}
\end{equation}
where $0 < \Q \leq 1$ is a parameter set by the user that indicates
how much mass \emph{must} be transported, in total.

Since these equations and constraints are all linear, we have the
following theorem:
\begin{thm}
The equations {\rm
(\ref{e:discrete_kantorovich})-(\ref{e:discrete_kantorovich_CONSTRAINTS})}
and
{\rm
(\ref{e:discrete_PARTIAL_kantorovich})-(\ref{e:discrete_PARTIAL_kantorovich_CONSTRAINTS})} admit a global
minimizer that can be computed in polynomial time, using standard
linear-programming techniques.
\end{thm}

When correspondences between surfaces are sought, i.e. when one
surface is viewed as being transformed into the other, one is
interested in restricting $\pi$ to the class of permutation matrices
instead of allowing all bistochastic matrices. (This means that each
entry $\pi_{ij}$ is either 0 or 1.)  In this case the number of
centers $z_i$ must equal that of $w_j$, i.e. $n=N=p$, and it is best
to pick the centers so that $\mu_i=\frac{1}{N}=\nu_j$, for all $i,\
j$. It turns out that these restrictions are sufficient to {\em
guarantee} (without restricting the choice of $\pi$ in any way) that
the minimizing $\pi$ is a permutation:

\begin{thm}
If $n=N=p$ and $\mu_i=\frac{1}{N}=\nu_j$, then
\begin{enumerate}
\item
There exists a global minimizer of
{\rm(\ref{e:discrete_kantorovich})} that is a permutation matrix.
\item
If furthermore $\Q = \frac{M}{N}$, where $M< N$ is an integer, then
there exists a  global minimizer of {\rm
(\ref{e:discrete_PARTIAL_kantorovich})} $\pi$ such that $\pi_{ij}
\in \{0,1\}$ for each $i,\,j$.
\end{enumerate}
\label{t:relaxation}
\end{thm}
\begin{rem}
In the second case, where $\pi_{ij} \in \{0,1\}$ for each $i,\,j$,
and $\sum_{i,j=1}^N \pi_{ij}=M$, $\pi$ can still be viewed as a
permutation of $M$ objects, ``filled up with zeros''.  That is, if
the zero rows and columns of $\pi$ (which must exist, by the pigeon
hole principle) are removed, then the remaining $M \times M$ matrix
is a permutation.
\end{rem}
\begin{proof}
We first note that in both cases, we can simply renormalize each
$\mu_i$ and $\nu_j$ by $N$, leading to the rescaled systems
\begin{equation}
  \left \{ \begin{array}{c}
             \sum_i \pi_{ij} = 1 \\
             \sum_j \pi_{ij} = 1 \\
             \pi_{ij} \geq 0
           \end{array}
   \right . \mbox{\hspace{1 in}}
  \left \{\begin{array}{c}
         \sum_i \pi_{ij} \leq 1 \\
             \sum_j \pi_{ij} \leq 1 \\
             \sum_{i,j} \pi_{ij} = M  \\
             \pi_{ij} \geq 0
\end{array}
   \right .
\label{e:discrete_kantorovich_CONSTRAINTS_disk}
\end{equation}
To prove the first part, we note that the left system in
(\ref{e:discrete_kantorovich_CONSTRAINTS_disk}) defines a convex
polytope in the vector space of matrices that is exactly the
Birkhoff polytope of bistochastic matrices. By the Birkhoff-Von
Neumann Theorem \cite{Lovasz86} every  bistochastic matrix is a
convex combination of the permutation matrices, i.e. each $\pi$
satisfying the left system in
(\ref{e:discrete_kantorovich_CONSTRAINTS_disk}) must be of the form
$\sum_k c_k\tau^k$, where the $\tau^k$ are the $N!$ permutation
matrices for $N$ objects, and $\sum_k c_k = 1$, with  $c_k \geq 0$.
The minimizing $\pi$ in this polytope for the linear functional
(\ref{e:discrete_kantorovich}) must thus be of this form as well. It
follows that at least one $\tau^k$ must also minimize
(\ref{e:discrete_kantorovich}), since otherwise we would obtain the
contradiction
\begin{equation}\label{e:linear_extrermas_at_vertices}
    \sum_{ij}d_{ij}\pi_{ij} = \sum_k c_k \Big (\sum_{ij}d_{ij} \tau^k_{ij} \Big ) \geq \min_k \Big \{ \sum_{ij}d_{ij} \tau^k_{ij} \Big \} > \sum_{i,j}\,d_{ij}\,\pi_{ij} ~.
\end{equation}

The second part can be proved along similar steps: the right system
in (\ref{e:discrete_kantorovich_CONSTRAINTS_disk}) defines a convex
polytope in the vector space of matrices; it follows that every
matrix that satisfies the system of constraints is a convex
combination of the extremal points of this polytope. It suffices to
prove that these extreme points are exactly those matrices that
satisfy the constraints and have entries that are either 0 or 1
(this is the analog of the Birkhoff-von Neumann theorem for this
case; we prove this generalization in a lemma in Appendix C); the
same argument as above then shows that there must be at least one
extremal point where the linear functional
(\ref{e:discrete_kantorovich}) attains its minimum.
\end{proof}

When we seek correspondences between two surfaces, there is thus no
need to {\em impose} the (very nonlinear) constraint on $\pi$ that
it be a permutation matrix; one can simply use a (standard) linear
program and Theorem \ref{t:relaxation} then guarantees that the
minimizer for the ``relaxed'' problem {\rm
(\ref{e:discrete_kantorovich})-(\ref{e:discrete_kantorovich_CONSTRAINTS})}
or {\rm (\ref{e:discrete_PARTIAL_kantorovich})-
(\ref{e:discrete_PARTIAL_kantorovich_CONSTRAINTS})} is of the
desired type if $n=N=p$ and $\mu_i=\frac{1}{N}=\nu_j$.

\subsection{Consistency}
In our schemes to compute the surface transportation distance, for
example by solving (\ref{e:discrete_PARTIAL_kantorovich}), we have
so far not included any constraints on the regularity of the
resulting optimal transportation plan $\pi^*$. When computing the
distance between a surface and a reasonable deformation of the same
surface, one does indeed find, in practice, that the minimizing
$\pi^*$ is fairly smooth, because neighboring points have similar
neighborhoods. There is no guarantee, however, that this has to
happen. Moreover, we will be interested in comparing surfaces that
are far from (almost) isometric, given by noisy datasets. Under such
circumstances, the minimizing $\pi^*$ may well ``jump around''. In
this subsection we propose a regularization procedure to avoid such
behavior.

Computing how two surfaces best correspond makes use of the values
of the ``distances in similarity'' $d^R_{\mu,\nu}(z_i,w_j)$ between
pairs of points that ``start'' on one surface and ``end'' on the
other; computing these values relies on finding a minimizing \Mbs
transformation for the functional (\ref{e:d_mu,nu(z,w)_def}). We can
keep track of these minimizing \Mbs transformations $m_{ij}$ for the
pairs of points $(z_i,w_j)$ proposed for optimal correspondence by
the optimal transport algorithm described above. Correspondence
pairs $(i,j)$ that truly participate in some close-to-isometry map
will typically have \Mbs transformations $m_{ij}$ that are very
similar. This suggests a method of filtering out possibly mismatched
pairs, by retaining only the set of correspondences $(i,j)$ that
cluster together within the M\"{o}bius group.

There exist many ways to find clusters. In our applications, we
gauge how far each \Mbs transformation $m_{ij}$ is from the others
by computing a type of $\ell_1$ variance:
\begin{equation}\label{e:variance_function}
    E_V(i,j) = \sum_{(k,\ell)}\norm{m_{ij} - m_{k\ell}},
\end{equation}
where the norm is the Frobenius norm (also called the
Hilbert-Schmidt norm) of the $2\times 2$ complex matrices
representing the M\"{o}bius transformations, after normalizing them
to have determinant one. We then use $E_V(i,j)$ as a consistency
measure of the corresponding pair $(i,j)$.

\section{Generalization to sphere-type surfaces}
\label{s:generalization_to_sphere_type}
So far we have restricted ourselves to disk-type surfaces, which is
somewhat limiting in practice. It is fairly straightforward to
generalize the ideas presented in
\cite{Lipman_Daubechies:2010:polytimesurfcomp} to other types of
surfaces; in this part of the paper we show how this can be done. We
choose to concentrate on the common case of sphere-type surfaces,
that is, genus zero surfaces. We will start by making the necessary
theoretical changes, and then we will present the numerical
algorithm; an example will be given in Section \ref{s:examples},
alongside examples for disk-type surfaces.

\subsection{Generalization of the distance function.}
The Uniformization theory for sphere-type surfaces ensures a
conformal one-to-one and onto mapping of the surface to the 2-sphere
or equivalently to the extended complex plane $\eC=\C \cup
\set{\infty}$, the Stone-\v{C}eck compactification of $\C$.
\nomenclature{$\eC$}{The extended complex plane, $\eC = \C \cup
\set{\infty}$.} The group of automorphisms of the extended plane are
the M\"{o}bius transformations $m:\eC \too \eC$, given by:
\begin{equation}\label{e:general_mobius_trans}
    m(z)=\frac{az+b}{cz+d},
\end{equation}
where $a,b,c,d \in \C$ and $ad-bc \ne 0$. In other words, any
bijective conformal mapping taking the extended plane to itself is a
M\"{o}bius transformation, and vice-versa any M\"{o}bius
transformation is a bijective conformal map of the extended plane.

The key to successful generalization to this case is choosing the
neighborhoods $\Omega_{z_0,R}$ for a point $z_0$ in a
M\"{o}bius-invariant way. In contrast to the situation on the disk,
where the hyperbolic distance is invariant under the group $\Md$,
the extended complex plane $\eC$ does not posses a distance
invariant under the full M\"{o}bius group $\Mec$. This can be
understood by noting that there is no non-constant continuous
two-argument function $f(z,w)$ such that $f(m(z),m(w))=f(z,w)$ for
all M\"{o}bius transformations $m\in \Mec$ and all $z,w\in \eC$.
Therefore, the neighborhoods must be constricted in a different way.

We tackle this problem by starting with the most basic invariant of
M\"{o}bius geometry, namely (generalized) circles. The neighborhood
of a point $z_0$ will then be defined as the interior of a
particular circle in the extended plane.

Let us first define the collection of \emph{circles} in $\eC$ plane
with prescribed \emph{orientation} by $\CC$. The role of the
orientation attached to circles will become clear momentarily.

\begin{defn}\label{def:mobius_circle}
\textbf{ }
\begin{enumerate}
\item
A \emph{circle} $c \in \CC$ is defined as the set of complex numbers
satisfying an equation of the type
$$\mathfrak{A}z\bar{z} + \mathfrak{B}z + \bbar{\mathfrak{B}z}+\mathfrak{D}=0,$$
where $\mathfrak{A},\mathfrak{D} \in \R$ and $\mathfrak{B}\in \C$.
Note that we define $\infty \in c$ if $\infty$ is an accumulation
point of $c$ (in the extended complex plane topology).
\item
The \emph{orientation} of a circle $c\in \CC$ is defined by labeling
the ``inside'' and ``outside'' connected parts of $\eC \setminus c$.
\end{enumerate}
\end{defn}
\nomenclature{$\CC$}{The collection of circles endowed with
orientation in $\eC$.}

Note that this definition of a ``circle'' includes straight lines
that can be thought of as circles through infinity (which is a
legitimate point in the extended complex plane), and also ``empty
circles'' that are empty sets (e.g. if
$\mathfrak{A}=\mathfrak{D}=1$, $\mathfrak{B}=1$).

The candidate neighborhoods, needed to generalize the definition of
$d^R_{\mu,\nu}(\cdot,\cdot)$, will be defined by selecting
particular oriented circles in $\CC$; since M\"{o}bius
transformations already take circles to circles, we need to ensure
only that our selection criterium is invariant under M\"{o}bius
transformations as well. We shall of course pick the neighborhood of
$z_0 \in \eC$ from the collection of oriented circles $\CC$ that
contain $z_0$. In addition, the choice should be: 1)
isometry-invariant (i.e., if $\M,\N$ are isometric sphere-type
surfaces, and $z_0, w_0 \in \D$ are corresponding points under the
isometry on their uniformizations on $\eC$, then the M\"{o}bius
transformation $m$ that corresponds to that isometry should map the
neighborhood $\Omega_{z_0,A}$ around $z_0$ to the neighborhood
$\Omega_{w_0,A}$ around $w_0$), 2) robust to noise, and 3)
characterized by a ``size'' parameter similar to $R$ in the
disk-type surfaces.

The key idea is to single out from the collection of circles $\CC$ a
single or discrete number of circles using the surface's metric
tensor. We will outline one possible construction as an example but
other construction are certainly possible.

To characterize ``size'' we shall use area: we consider circles such
that their area (w.r.t the surface's metric) is of prescribed
magnitude $A$, that is  $\set{c\ \mid \ \Vol_\M(\bbar{c})=A}$, where
$\bbar{c}$ denotes the union of the interior of the circle $c$
(defined by its orientation) and the set $c$. Since we assume our
surfaces have unit area, $A\in (0,1]$.\nomenclature{$\bbar{c}$}{The
union of the interior of the set $c$ and the set $c$ (used for
circles).}

The neighborhood $\Omega_{z_0,A}$ is then defined by
\begin{align}  \nonumber
 \Omega_{z_0,A}  =  & \bbar{\mathop{\mathrm{argmin}}_{c\in\CC_{z_0,A}}
 \mathrm{length}_\M(c)},\\ \label{e:def_nigh_in_sphere}
 & \mathrm{where}\\ \nonumber
 & C_{z_0,A} = \set{c\in \CC \ \mid \ \Vol_\M(\bbar{c}) = A \ \mathrm{and} \
 z_0 \in \bbar{c}},\ \mathrm{and} \nonumber
\end{align}
where $\mathrm{length}_{\M}(c)$ denotes the length of the curve $c$
based on the metric of surface $\M$.
\nomenclature{$\mathrm{length}_\M (c)$}{The length of a curve $c$
w.r.t the metric of surface $\M$.}

A few remarks are in order. Let us consider the collection of
circles $c\in\CC$ such that $\Vol_\M(\bbar{c})=A$. We can use the
Riemann sphere model which can be thought of as the standard $S^2
\subset \R^3$. Each circle $c\in\CC$ is either s standard circle on
$S^2$ or a point or an empty set. For every point $p \in S^2$ there
is a one dimensional family of circles $c_{p,t}$ defined by
$c_{p,t}=\set{q \in S^2 \ \mid \ \ip{q,p}=1-t}$, where $t\in(0,2)$:
the interior is taken to be the part that contains the point $p$, so
that $\bbar{c}_{p,t} = \set{q \in S^2 \ \mid \ \ip{q,p} \geq 1-t}$.
Obviously, $v:t\mapsto \Vol_\M(c_{p,t})$ is a monotone function, and
$\mathop{\lim}_{t\too 0}v(t) = 0$, $\mathop{\lim}_{t\too 1}v(t) =
1$. Lastly, $v(t)$ is a continuous function and therefore there
exists a unique value $t_A$ (depending on $p$) such that $v(t_A)=A$.
This means that for every point $p\in S^2$ we can find a unique
circle from the concentric family $\set{c_{p,t}}_{p\in S^2,
t\in(0,2)}$ that has area $A$. Since every non-empty and non-point
circle in $\CC$ can be identified as $c_{p,t}$ for some point $p\in
S^2$ and some $t\in(0,2)$, the collection $c_{p,t_A(p)}$, $p\in S^2$
is a parametrization of the collection $\set{c\in\CC \ \mid \
\Vol_\M(\bbar{c})=A}$. Now, the extra restriction $z_0\in \bbar{c}$
defines a subset of $S^2$ in the sense that we consider only $p\in
S^2$ such that $z_0\in \bbar{c}_{p,t_A}$. Since $t_A$ is continuous
as a function of $p$ this restriction defines a compact subset of
$S^2$, which in turn implies that the minimum in
eq.(\ref{e:def_nigh_in_sphere}) is achieved. Lastly, on this two
dimensional manifold we consider the function $c\mapsto
\mathrm{length}_\M(c)$ which is a smooth function and in the generic
case has a unique global minimum. We will henceforth assume that
eq.(\ref{e:def_nigh_in_sphere}) has a unique minimizer.

Since the neighborhoods $\Omega_{z_0,A}$ are chosen from the circle
collection, using only intrinsic properties they are invariant to
M\"{o}bius changes of coordinates of the metric; in other words if
two isometric surfaces $\M,\N$ are compared at a pair of isometric
points $z_0,w_0$ ($w_0$ is the image of $z_0$ under the isometry)
then the isometry mapping $\M$ to $\N$ is a M\"{o}bius
transformation taking not only $z_0$ to $w_0$ but also
$\Omega_{z_0,A}$ to $\Omega_{w_0,A}$.

Once the neighborhoods are set, the definition of
$d^A_{\mu,\nu}(z_0,w_0)$ is straightforward: denote by
$M^A_{z_0,w_0}$ the collection of M\"{o}bius transformations $m$
that take the interior of the circle $\Omega_{z_0,A}$ to the
interior of $\Omega_{w_0,A}$, and for which $m(z_0)=w_0$. This is
again a one parameter subgroup parameterized over the unit circle
(angle) $\theta \in [0,2\pi)$, as in the disk-type surface case we
then define

\begin{equation}\label{e:d_mu,nu(z,w)_def_sphere_type}
    d^A_{\mu,\nu}(z_0,w_0) :=
\mathop{\inf}_{m \in M^A_{z_0,w_0}}\int_{\Omega_{z_0,A}} \,\Big |\,1
- \frac{ \nu((m(z)) \abs{m'(z)}^2}{\mu(z)}\,\Big|\, d\vol_M(z).
\end{equation}
This is indeed the analog to eq.~(\ref{e:d_mu,nu(z,w)_def}) for
disk-type surfaces since both integrals can be written in the
invariant form
$$\int_{\Omega_{z_0,A}} \norm{\wt{g}(z) -
\parr{m^*\wt{h}}(z)}_{\wt{g}} d\vol _\M (z),$$ where $\wt{g}$ (resp. $\wt{h}$) is the
push-forward metric of $\M$ (resp. $\N$) on $\eC$, and the norm
$\norm{\cdot}_{\wt{g}}$ is the standard one, induced by $\wt{g}$:
for a tensor $t_{ij}dx^i\otimes dx^j$, we denote $\wt{g}=\wt{g}_{ij}
dx^i \otimes dx^j$, then $\norm{t}^2_{\wt{g}} =
t_{ij}t_{k\ell}\wt{g}^{ik}\wt{g}^{j\ell}$.

This area-based definition of the neighborhoods could also be used
in the disk-type case; this would yield a unified definition for
both cases. The difference between the new definition (above) and
the old definition (hyperbolic geodesic disk) in the case of
disk-type surfaces is that the old definition provides smaller
neighborhoods near the boundary for disk type surfaces, while the
new definition will maintain constant area neighborhoods even
arbitrarily close to the surface's boundary; depending on the
application and data properties, one or the other selection may be
preferable.

\subsection{Numerical details.}

The algorithm for the sphere-type case is basically the same as for
the disk-type; That is, we first sample $N=n=p$ equally distributed
points (as described in Section
\ref{s:the_discrete_case_implementation}) $Z$,$W$ on the surfaces
$\M$,$\N$ (respectively). Second, the cost function
$d^A_{\mu,\nu}(z_i,w_j)$ is computed between every pair of sample
points $(z_i,w_j) \in Z\times W$, and finally, a discrete
mass-transportation problem is solved between the discrete measures
$\mu_Z$ and $\nu_W$ to output the distance and the correspondences.
A few adjustments need to be made to this algorithm for the
sphere-type case: 1) precomputing (approximating) the neighborhoods
$\Omega_{z_i,A}$, $z_i\in Z$, and $\Omega_{w_j,A}$, $w_j\in W$,
which involves more computation than for the disk-type surfaces
case, 2) representing the conformal density on the extended complex
plane rather than the unit disk, and approximating the local
distance $d^R_{\mu,\nu}(z_i,w_j)$, and 3) calculating the optimal
transport between the discrete densities. Next we describe these
adjustments in more detail.

\par\textbf{Computing the neighborhoods $\Omega_{z_i,A}$.}
We describe the construction of neighborhoods $\Omega_{z_i,A}$ for
every sample point $z_i\in Z$ in $\M$. The construction in $\N$ is
identical.  We want to find the neighborhood $\Omega_{z_i,A}$
(assumed unique) based on the definition
eq.(\ref{e:def_nigh_in_sphere}). That is, $\Omega_{z_i,A}$ is the
interior of a conformal circle, has surface area $A$, and has
minimal circumference compared to all other such circles. A circle
(on $S^2$ or equivalently on $\eC$) is defined by a triplet of
points $z_j,z_k,z_\ell$ in the usual way. Adding the orientation,
each triplet provides us with two choices of conformal circle
neighborhoods. In our implementation we considered all $2{N \choose
3}$ circles generated by the sample set $Z$. For each such circle
(endowed with orientation) we estimate the surface area inscribed in
it. If it is $\epsilon-$close to the prescribed amount $A$ we
estimate its circumference on the surface. We define
$\Omega_{z_i,A}$ the one with smallest circumference that contains
$z_i$.

\par\textbf{Approximating $d^R_{\mu,\nu}(z_i,w_j)$.}
The second issue that arises when generalizing to sphere-type
surfaces is the representation of the conformal density
$\mu(z),\nu(w)$. One option is to use a spherical interpolation
scheme and repeat the steps as described in Appendix~\ref{a:appendix
B}. However, one can also pick a different path that is very simple
and offers an alternative to the smooth TPS approximation described
in Appendix~\ref{a:appendix B}. The idea is to represent the
conformal density by keeping track of a set of equally spread points
$\wt{Q}=\set{\wt{q}_\ell}_{\ell=1}^\L \subset \M$ on the surface
(similarly to $Z$), where each point represent a surface patch
(Voronoi cell) of size $\frac{1}{\L}$. In our implementation for
sphere-type surfaces, we use this latter choice. We usually take a
set of size $\L\approx 1000$. The discrete density is then
represented on the extended plane as $Q=\set{q_\ell}_{\ell=1}^\L =
\Phi(\wt{Q}) \subset \eC$. It can be shown that given a domain
$\Omega \subset \eC$ we have the approximation:
\begin{equation}\label{e:points_density_tto_approximate_volume}
    \abs{\frac{1}{\L} \sum_{q_i\in \Omega}1 - \int_\Omega d\vol_\M } \leq C \varphi
    (\set{\wt{q}_\ell}).
\end{equation}
To justify this approximation result let us denote by $V_i$ the
Voronoi cells of the set $\wt{Q}$ based on the metric of surface
$\M$. Lemma \ref{lem:voronoi_in_balls} (proved in Appendix
\ref{a:approximating_optimal_transport}) then implies that
$\frac{1}{\L}\sum_{q_i\in \Omega}1 = \Vol_\M(\cup V_i) \leq
\Vol_\M\parr{\mathop{\cup}_{p\in\M}
B_g(p,2\varphi_g(\wt{Q})+\epsilon)}$, for arbitrary $\epsilon>0$,
and it is not hard to see that for domains $\Omega$ with regular
boundary curves $\abs{\Vol_\M\parr{\mathop{\cup}_{p\in\M}
B_g(p,2\varphi_g(\wt{Q})+\epsilon)} - \Vol_\M(\Omega)} \leq C
\varphi(\set{\wt{q}_\ell})$.

Note that the above arguments do not use the uniform property of
$\wt{Q}$, and will work also when the sampling is not uniform, as
long as the fill distance $\varphi_g(\wt{Q})\too 0$.

Figure \ref{f:sphere_density_and_neighborhood} shows the density
points spread on a cat model and on its uniformization sphere, as
well as the neighborhood (in red) of a single point on the cat's
front leg. Denote by $P=\set{p_\ell}_{\ell=1}^\L \subset \eC$ the
density points for surface $N$. To approximate
$d_{\mu,\nu}^A(z_0,w_0)$ as defined in
eq.(\ref{e:d_mu,nu(z,w)_def_sphere_type}) we follow the following
steps: first, we map $Q \cap \Omega_{z_0,R}$ to the unit disk $\D$
via a M\"{o}bius transformation that is defined by taking $z_0$ to
the origin and $\Omega_{z_0,A}$ to the unit disk $\D$. Denote the
resulting unit disk points by $Q_{z_0}$. Similarly we map $P\cap
\Omega_{w_0,A}$ to the unit disk (taking $w_0$ to the origin). We
denote the resulting set by $P_{w_0}$. Second, for each $\theta \in
\set{0,\frac{2\pi}{L},2\frac{2\pi}{L},...,(L-1)\frac{2\pi}{L}}$ we
rotate $Q_{z_0}$ by $\theta$ around the origin, $e^{\bfi \theta}
Q_{z_0}$, and compare to the second density $P_{w_0}$. The way we
compare the two densities is justified by
eq.~(\ref{e:points_density_tto_approximate_volume}), that is we
subdivide the unit disk (actually the entire square $[-1,1]^2$) into
bins and count, for each of the two densities $e^{\bfi \theta}
Q_{z_0}, P_{w_0}$ the number of points in each bin. Then, we can sum
the absolute value of the difference to achieve our approximation to
$d_{\mu,\nu}^A(z_0,w_0)$. To smoothen the approximation, it is
useful to convolve the bins' structure with some kernel (this is
analog to the smoothing splines used in the disk case). In our
experiments (presented in Section \ref{s:examples}) we used a
$30\times30$ bin structure and convolved with the kernel
$\frac{1}{9}(1,1,1)\otimes(1,1,1)$.

\par\textbf{Solving the linear programming for spheres.}
Once we have defined $Z=\set{z_i}_{i=1}^n, W=\set{w_j}_{j=1}^p$,
$\mu_Z,\nu_W$ and $\hh{d}^A_{\mu,\nu}(z_i,w_j)$ we can go ahead and
calculate $T_{\hh{d}}(\mu_Z,\nu_W)$ as explained in Section
\ref{ss:step_3_linear_programming}. Since our analysis in Appendix
\ref{a:approximating_optimal_transport} is for general compact
separable and complete metric spaces it will hold also for the
sphere case. Hence, once the approximation error of
$\hh{d}^A(z_i,w_j) \approx d^A_{\mu,\nu}(z_i,w_j)$ is set (as
outlined above), Theorem
\ref{thm:final_approx_discrete_transportation} can be applied to
yield the convergence result.

\begin{figure}[t]
\centering \setlength{\tabcolsep}{0.4cm}
\includegraphics[width=1\columnwidth]{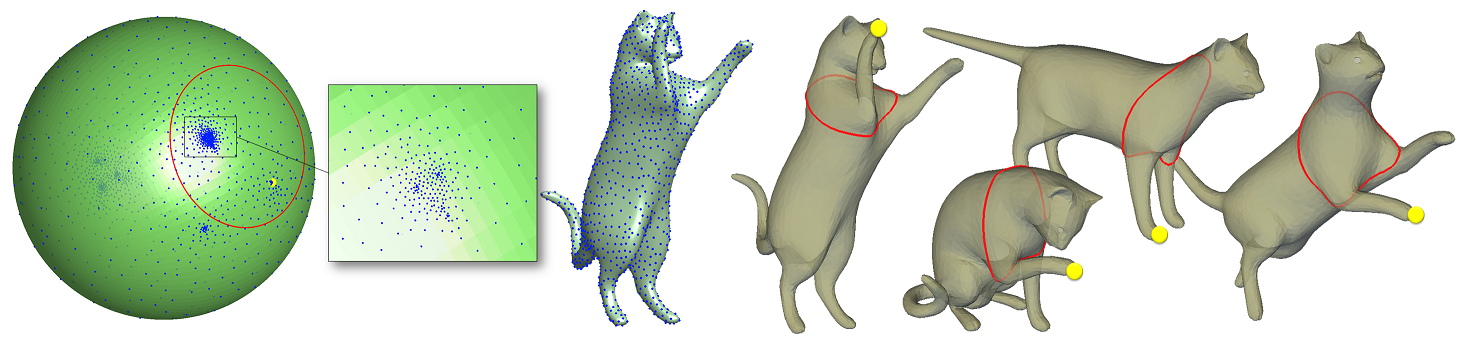}
\caption{Sampling of the sphere-type cat surface (third from the
left), and the sampling shown on the uniformization sphere (left).
Note the zoom-in inset of the cat's head. We also show the
neighborhood (marked with red on the sphere) for $A=0.3$ of a point
(marked with yellow ball on the sphere) on the cat's front leg (the
model is taken from the Non-rigid World data-set
\protect{\cite{BBK06b}}). On the right we show the neighborhood
corresponding to the yellow point on the original cat surface model
(fourth from the right) and on few other surfaces from the same
class (the neighborhoods were computed independently in each
surface). Note the invariance of this neighborhood under
nearly-isometric deformations.}
\label{f:sphere_density_and_neighborhood}

\end{figure}

% ----------------------------------------------------------------
\section{Stability}
\label{s:stability}

In this section we prove Theorem \ref{thm:stability_4.1}; this is a
first result connecting the new distance with local geodesic
distortion. We will prove it for the disk case, but similar
arguments can be used for the sphere case. 

\begin{proof}
(of Theorem \ref{thm:stability_4.1}) We will use three different
metrics on $\D$: the metrics of $\M,\N$ represented by the tensors
$\wt{g},\wt{h}$ (resp.), and the hyperbolic metric
$d_H(\cdot,\cdot)$.

The main idea of the proof is to use that a small value for $\d^R(\M,\N)$
means that there exists a $\pi \in \Pi(\mu,\nu)$ with respect to which
the integral  (\ref{e:generalized_Kantorovich_transportation}) of the cost
function $d^R_{\mu,\nu}(z,w)$ is small as well; by the
definition (\ref{e:d_mu,nu(z,w)_def}) of this cost function there must therefore
be many corresponding neighborhoods $\Omega_{z_0,R}$ and $\Omega_{w_0,R}$ in
$\M$ and $\N$ respectively that are very similar; we shall use these similarities
to build local isometries.

Denote $\d^R(\M,\N)=\eps$. Fix $K\subset \D$ to be a hyperbolic disk
centered at the origin with an arbitrarily large (but finite)
radius, $K=\Omega_{0,L}$. (Note: we could equally well have picked
$K$ to be an arbitrary set that is compact in the hyperbolic metric;
this particular choice alleviates notations.) Because
$\Omega_{0,L+R}$ is a compact subset of $\D$, there exists a
constant $C=C(L)< \infty$ such that $d_H(z,z') \leq C d_g(z,z')$ for
all $z,\,z'\in \Omega_{0,L+R}$. Similarly, there exist positive
constants $C'=C'(L),\,C_1=C_1(L)<\infty$ such that $\mu(z)\geq
\frac{1}{C'}$ for all $z\in \Omega_{0,L+R}$ and $\Vol_\M(\Omega_{z,R/2})
\geq \frac{2}{ C_1}$ for all $z \in \Omega_{0,L} = K$
.

Now set $r < R/2C$. We will prove the desired
bounds for arbitrary points $z_1,z_2\in K$ such that $d_g(z_1,z_2)<r$. Let us pick such an
arbitrary pair, which we shall keep fixed for the moment. We immediately note that
$d_H(z_1,z_2)< R/2$.

Let now $\gamma_{z_1,z_2}$ be the minimal-length-geodesic curve
connecting $z_1$ and $z_2$ (in terms of the metric corresponding to
surface $\M$); by taking $r>0$ sufficiently small we can ensure that
this geodesic is unique. Since $d_g(z_1,\xi)<r$ and thus
$d_H(z_1,\xi)<R/2$ for all $\xi \in \gamma_{z_1,z_2}$, it follows
that $\gamma_{z_1,z_2} \subset \Omega_{z_1,R/2}$. Morever, by a
simple application of the triangle inequality, we have
$\gamma_{z_1,z_2} \subset \Omega_{z_0,R}$ for all $z_0$ such that
$d_H(z_1,z_0)<R/2)$, i.e. for all $z_0 \in \Omega_{z_1,R/2}$. It
follows that
$$
\Omega_{z_1,R/2}\subset B := \set{z_0 \ \mid  \ \gamma_{z_1,z_2} \subset
\Omega_{z_0,R} }\,.
$$
On the other hand,
$\Vol_\M(\Omega_{z_1,R/2})\geq$ $\inf_{z\in K}\Vol_\M(\Omega_{z,R/2})$
$= \min_{z\in K}\Vol_\M(\Omega_{z,R/2})>\frac{2}{ C_1}$. This
implies that the volume of $B$ on the
surface $\M$ can be bounded from below by
$$
\Vol_\M(B)\geq
\Vol_\M(\Omega_{z_1,R/2}) \geq \frac{2}{ C_1}\,.
$$

We can use this lower bound to show that there must be points $z_0$ in $B$,
and corresponding points $w_0$ in $\D$,
for which $d^R_{\mu,\nu}(z_0,w_0)$ is small. Indeed, let
be $\pi^*$ an optimal transportation plan realizing the
minimal transportation cost, then:
\begin{align*}
\eps & = \int_{\D\times \D}d^R_{\mu,\nu}(z,w)d\pi^*(z,w)
 \geq \int_{B\times \D}d^R_{\mu,\nu}(z,w) d\pi^*(z,w) \\
& \geq \inf_{(z,w)\in B\times \D}\brac{d^R_{\mu,\nu}(z,w)}
\int_{B\times \D}d\pi^*(z,w) =\inf_{(z,w)\in B\times
\D}\brac{d^R_{\mu,\nu}(z,w)} \vol_\M(B).
\end{align*}

There thus exists some point $(z_0,w_0)\in B\times \D$ such that
$$d^R_{\mu,\nu}(z_0,w_0) \leq \frac{2\,\eps}{\vol_\M(B)} \leq C_1 \,\eps.$$

Next, we note that $d^R_{\mu,\nu}$ can be written as (see
\cite{Lipman_Daubechies:2010:polytimesurfcomp}, and
eq.(\ref{e:d_mu,nu(z,w)_def}))
$$
d^R_{\mu,\nu}(z_0,w_0) = \inf_{m\in \Md, m(z_0)=w_0}
\int_{\Omega_{z_0,R}}\norm{\wt{g}-m^*\wt{h}}_{\wt{g}(z)}d\vol_\M(z)
= \int_{\Omega_{z_0,R}} \abs{\mu(z)-\nu(m(z))} \, d\vol_H(z),
$$
where $\|\cdot\|_{\wt{g}(z)}$ is the norm in the relevant tensor space
as defined in Section \ref{s:generalization_to_sphere_type} (and in
\cite{Lipman_Daubechies:2010:polytimesurfcomp}). The $m\in \Md$ that
satisfy $m(z_0)=w_0$ constitute a one-parameter compact family, so
that the infimum is achieved; let us call $f^{z_0}$ this minimizing
M\"{o}bius transformation from $\Omega_{z_0,R}$ to $\Omega_{w_0,R}$.
We have thus
$$
\int_{\Omega_{z_0,R}}\norm{\wt{g}-\left(f^{z_0}\right)^*\wt{h}}_{\wt{g}(z)}d\vol_\M(z)
= \int_{\Omega_{z_0,R}} \abs{\mu(z)-\nu(f^{z_0}(z))} \, d\vol_H(z)
\leq C_1 \,\epsilon\,.
$$
Now
$$
s(z) = \abs{\mu(z)- \nu(f^{z_0}(z))}
$$
is Lipschitz as a function of the argument $z$, since $\mu,\nu$ are
Lipschitz on $\Omega_{0,L+2R}\subset \D$, 
and $f^{z_0}$ is analytic; moreover, by observing
that $\Omega_{z_0,R}\subset \Omega_{0,L+2R}$, we can
bound the Lipschitz constant for $s$ independently of the particular
choices made so far, i.e., for all $z,\,z' \in \Omega_{z_0,R}$,
$$
|s(z)-s(z')|\le \kappa \, |z-z'|\,.
$$
Take now any $u$ in $\Omega_{z_0,R}$, and set $S=s(u)$. Then
$$
C_1 \, \epsilon \ge \int_{\Omega_{z_0,R}} \max(0, S- \kappa
|u-z|)\,d\vol_H(z) \geq C'' \, S^3\,,
$$
where $C''>0$ can be picked independently of the location of $u$ within 
$\Omega_{z_0,R}$, and uniformly for $z_0 \in \Omega_{0,L+R}$. 
It follows that $S \le C_2\, \epsilon^{1/3}$. This shows that
\begin{equation}\label{e:first_bound_on_norm_mu-f*nu}
    \max_{z\in \Omega_{z_0,R}}\norm{\wt{g}-m^*\wt{h}}_{\wt{g}(z)} \leq  C_2
    \,\eps^{1/3}\,,
\end{equation}
for some constant $C_2$ that doesn't depend on $\epsilon$ or $z_1,\,z_2$.

Given arbitrary $z_1$, $z_2$ in $K$, we have thus found $\Omega_{z_0,R}$
that contains the full geodesic $\gamma_{z_1,z_2}$ and a M\"{o}bius map
$f^{z_0}$ from $\Omega_{z_0,R}$ to a corresponding $\Omega_{w_0,R}$ that,
within to a small error controlled by the small quantity $\d^R(\M,\N)$,
maps the local geometry in $\M$ to that in $\N$.

To alleviate notations in what follows, we drop the superscript
$z_0$ on $f^{z_0}$. Our plan is to use the minimizing geodesic path
$\gamma_{z_1,z_2}(t):[0,1]\too \M$ between the points
$\gamma_{z_1,z_2}(0)=z_1$, and $\gamma_{z_1,z_2}(1)=z_2$ to compute
bounds on $d_h(f(z_1),f(z_2))$. Using the differential $[Df]$ of the
map $f$, which is a linear map between the tangent spaces $T_{z}\M$
(with the metric $\wt{g}$) and $T_w\N$ (with the metric $\wt{h}$),
we have 
\begin{align}\label{e:prf_lower_bnd} \nonumber
d_h(f(z_1),f(z_2))  & \leq \int_0^1 \norm{\frac{d}{dt}
f(\gamma_{z_1,z_2}(t))}_h dt \\ \nonumber
& \leq \int_0^1 \norm{[Df]} \norm{\dot{\gamma}_{z_1,z_2} (t)}_g dt \\
& \leq \max_{z\in\gamma_{z_1,z_2}} \norm{[Df(z)]} d_g(z_1,z_2) \,,
\end{align}
which shows that an upper bound on $\max_{z\in\gamma_{z_1,z_2}}
\norm{[Df(z)]}$ will give us one of the desired inequalities.

The second inequality is achieved using that
\begin{align}\label{e:prf_upper_bnd} \nonumber
d_g(f^{-1}(f(z_1)),f^{-1}(f(z_2)))  & = \int_0^1 \norm{\frac{d}{dt}
f^{-1}(\gamma_{f(z_1),f(z_2)}(t))}_g dt \\ \nonumber
& \leq \int_0^1 \norm{[Df^{-1}]} \norm{\dot{\gamma}_{f(z_1),f(z_2)} (t)}_h dt \\
& \leq \max_{w\in\gamma_{f(z_1),f(z_2)}} \norm{[Df(w)]^{-1}}
d_h(f(z_1),f(z_2)) .
\end{align}
To use this, we thus need to upper bound
$\max_{w\in\gamma_{f(z_1),f(z_2)}} \norm{[Df(w)]^{-1}}$.

In the remainder of this proof, we show how bounds on
$\max_{z\in\gamma_{z_1,z_2}} \norm{[Df(z)]}$ and
$\max_{w\in\gamma_{f(z_1),f(z_2)}} \norm{[Df(w)]^{-1}}$ can be
derived from (\ref{e:first_bound_on_norm_mu-f*nu}).

First, we take an orthonormal basis $E=\set{e_1,e_2} \subset
T_{z}\M$. That is,
\begin{equation}\label{e:ortho_basis_e_i}
    e_k^i e_\ell^j g_{ij} = e_k^i e_\ell^j \wt{\mu}(z) \delta_{i,j} =
\delta_{k,\ell}.
\end{equation}
Similarly we take orthonormal basis $B = \set{b_1,b_2} \subset
T_{w}\N$.

We will denote the matrix $[Df] = [Df_{z}]$ representing the
differential $Df_z$ of $f$ at the point $z$, in the bases $E,B$. The
norm of $Df_z$ is the induced norm
$$\norm{Df_z} = \max_{\xi \in T_{z}\M, \xi \ne 0}\frac{\norm{Df(\xi)}_{\wt{h}(f(z))}}{\norm{\xi}_{\wt{g}(z)}}.$$

Writing the tensor $\wt{g}$ in the basis $E$ we get the Euclidean
form $\wt{g}(z) = de_1\otimes de_1 + de_2\otimes de_2$, and the
tensor $\wt{h}$ in the basis $B$ will have the same form $\wt{h}(w)
= db_1\otimes db_1 + db_1\otimes db_2$. The pull-back $f^*\wt{h}(z)$
will have the form $f^*\wt{h}(z) =
\parr{[Df_z]^t[Df_z]}_{ij} de_i\otimes de_j$. Therefore in the basis $E$
we have $$\norm{\wt{g}-f^* \wt{h}}_{\wt{g}(z)} =
\norm{Id-[Df]^t[Df]}_F,$$ where $Id$ is the $2\times 2$ identity
matrix, and $\norm{\cdot}_F$ denotes the Frobenius norm.

Writing the singular value decomposition of $[Df]$ we have
$$[Df] = Q \, \mbox{diag}
\parr{\sigma_1(z),\sigma_2(z)}R^T,$$ where $Q,R$ are orthogonal and
$\sigma_1\leq\sigma_2$ are the respective singular values. Then we
have
\begin{equation}\label{e:second_bound_on_norm_mu-f*nu}
    \norm{\wt{\mu}-f^*\wt{\nu}}_{g(z)}^2 = \norm{Id - [Df]^T [Df]}^2_F =
    \norm{R\parr{Id -  \left(
                              \begin{array}{cc}
                                (\sigma_1(z))^2 & 0 \\
                                0 & (\sigma_2(z))^2 \\
                              \end{array}
                            \right)
    }R^T}^2_F=(1-\sigma_1(z)^2)^2 + (1-\sigma_2(z)^2)^2,
\end{equation}

Using
(\ref{e:first_bound_on_norm_mu-f*nu}),(\ref{e:second_bound_on_norm_mu-f*nu})
we get
$$\abs{1-\sigma_2(z)^2}\leq C_2 \eps^{1/3}.$$
From this last bound on $\sigma_2(z)$ for $z\in \Omega_{z_0,R}$ there
exists a constant $C_3>0$ such that
$$0 \leq  \sigma_2(z) \leq 1+ C_3 \eps^{1/3}.$$
Since $\norm{[Df]}= \sigma_2$, this is the desired estimate for the first inequality.

For the second inequality, we need to bound $\norm{[Df]^{-1}}= \sigma_1^{-1}$. The
computation (\ref{e:second_bound_on_norm_mu-f*nu}) shows that
$$\abs{1-\sigma_1(z)^2}\leq C_2 \eps^{1/3}\,,$$
from which we obtain
$$0 \leq  \frac{1}{\sigma_1(z)}\leq 1+ C_4 \eps^{1/3}\,,$$
which concludes our argument.

\end{proof}

% ----------------------------------------------------------------
\section{Experimental validation and comments}
\label{s:examples}
In this section we perform experimental validation of our
algorithms. We have tested and experimented with our algorithms on four
different data-sets:
\begin{enumerate}
\item
\textsl{Non-rigid World data-set \cite{BBK06b}.} This data-set was
distributed by Bronstein,Bronstein and Kimmel and was specifically
constructed for evaluating shape comparison algorithms in the
scenario of non-rigid shapes; it contains meshes of different
objects (cats, dogs, wolves, humans, etc...) in different poses. We
compare our results on this data-set to the Gromov-Hausdorff
algorithm suggested in \cite{BBK06,BBK06b}.

\item
\textsl{SHREC 2007 Watertight Benchmark \cite{Giorgi07}} This
data-set contains meshes of several objects within a given \emph{semantic}
class for several different classes, such as chairs, 4-legged animals, humans, etc...
It is more challenging for
isometric-invariant matching algorithms since most of the
objects are far from isometric to the objects in the same
semantic class, for example the 4-legged animals class contains a
giraffe and a dog.

\item
\textsl{Synthetic.} We constructed this data-set to test the effect
of the ``size'' parameter $R$ on the distance behavior.

\item
\textsl{Primate molar teeth.} This data-set originates
from a real biological problem/application; it consists of molar teeth
surface for different primates. It was communicated to us by
biologists who compare these shapes for
characterization and classification of mammals.

\end{enumerate}

\begin{rem}
For all data-sets, we scaled the meshes to have unit area, because our goal is to compare
surfaces solely based on shape, regardless of size.
\end{rem}

\textbf{Non-rigid World data-set and comparison to
Gromov-Hausdorff-type distance.}
In the first experiment we ran our sphere-type algorithm to determine
conformal Wasserstein distances for all pairs in the Non-rigid World
data-set, distributed by Bronstein, Bronstein and Kimmel \cite{BBK06,BBK06b,BBK2007non_rigid_book};
we compare the results to those obtained using the code for the (symmetrized)
partial embedding Gromov-Hausdorff (speGH) distance distributed by the same authors.
The speGH distance has been used with great success in surface comparison
\cite{BBK06,BBK06b, BBK2007non_rigid_book} and can handle situations beyond the scope
of our, more limited algorithm, it can e.g., compare surfaces of different genus.
We therefore consider it as a state-of-the-art algorithm. In order to compare our
Conformal Wasserstein distance with speGH for applications of interest to us, we use both
algorithms on Non-rigid World data-set where all the surfaces are scaled to have unit area
and where 100 sample points are chosen on each surface.

The Non-rigid World data-set contains meshes of different poses of the following
articulated objects: a centaur, a cat, a dog, a horse, a human female, and two human males
(``Michael'' and ``David'').
%%
%%%
%In the first experiment we ran our
%sphere-type algorithm between all pairs of the Non-rigid World
%data-set, distributed by Bronstein, Bronstein and Kimmel
%\cite{BBK06,BBK06b}. This data-set contain meshes of articulated
%objects of the following categories: centaur, cat, dog, horse,
%wolves, human female, and human male. We have scaled each of the
%meshes to have total unit area. Note that we have done so in all our
%experiments since our distance is designed to be size invariant.
%Furthermore, we have down-sampled and smoothed the meshes. Lastly we
%have unified ``Michael'' and ``David'' to be of the same ``Human
%male'' category for two reasons: 1) to be consistent with the rest
%of the categories and, 2) since they are almost indistinguishable
%after the sub-sampling and scaling.
%
%We have also ran the algorithm for approximating the (symmetric)
%partial embedding Gromov-Hausdorff distance
%\cite{BBK06,BBK06b,BBK2007non_rigid_book}, that is
%$\max \set{d_{PE}(\M,\N),d_{PE}(\N,\M)}$, distributed by the same above
%mentioned authors; we refer to this distance as speGH. We have also tried using
%$d_{PE}(\M,\N)+d_{PE}(\N,\M)$ and $\min \set{d_{PE}(\M,\N),d_{PE}(\N,\M)}$ however
%these choices produced slightly worse results. We set the number of samples points to
%be $100$ (the same as we used in our algorithm).

The two resulting dissimilarity matrices are shown in Figure
\ref{f:GHvsOURS_dissim_mat_and_knn_mat} (a,d). The dissimilarity
matrices are both normalized by translating the minimal
value to zero and scaling the maximal value to one. The color scheme
is Matlab's ``Jet'' where dark blue represents low (close to zero)
value, and dark red indicates high (close to one) values. The
Sphere-type algorithm used $A=0.3$, and a  $30\times 30$ bins discretization with the
convolution kernel $(1,1,1)\otimes(1,1,1)$ (all the sphere-type
examples use these bin settings) to obtain the discretizations
of the conformal density. The timing for running
one comparison $\d^A(\M,\N)$ was 90 seconds on 2.2GHz AMD Opteron
processor. Figure
\ref{f:GHvsOURS_dissim_mat_and_knn_mat} (b-c) and (e-f) shows two nearest neighbors
classifications tests where a white square inside a dark-red area means success and white
square on black area is a failure.

The structure of the dissimilarity
matrix is illustrated in the two plots in Figure \ref{f:knn_tests}.
Figure \ref{f:knn_tests}
(a) shows the classification rates as a function of the number $K$ of
nearest neighbors, where for each fixed $K$ we calculated the
classification rate as follows. For each object we counted how many from
its $K$-nearest neighbors are of the same class. We summed all these numbers
and divided by the total number of possible correct classifications.
The blue curve shows the analysis of
the dissimilarity matrix output by our distance algorithm and
the red curve by the speGH distance code. In (b) we show the ROC curve
where for every $K=1,2,..,10$ we plot the True Positive Rate (TPR), that is
the number of true positives divided by the number of positives, as a function of the
False Positive Rate (FPR), that is the number of false positives divided by the number of negatives.

\begin{figure}[t]
\centering \setlength{\tabcolsep}{0.4cm}
\begin{tabular}{ccc}
\includegraphics[width=0.25 \columnwidth]{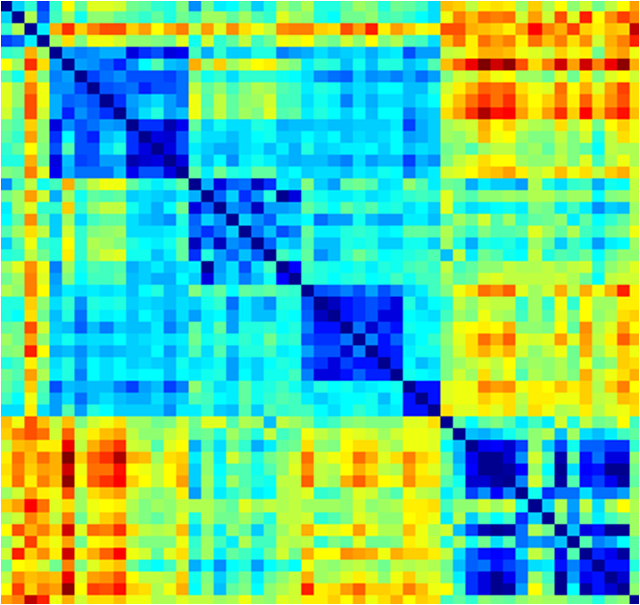} &
\includegraphics[width=0.25 \columnwidth]{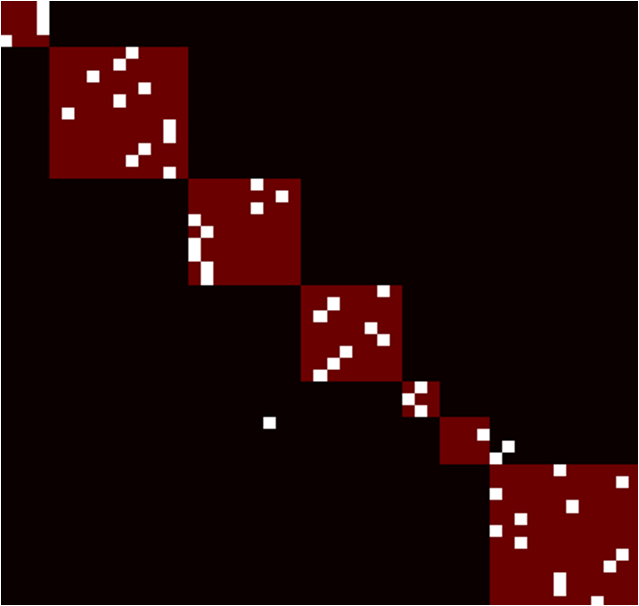} &
\includegraphics[width=0.25 \columnwidth]{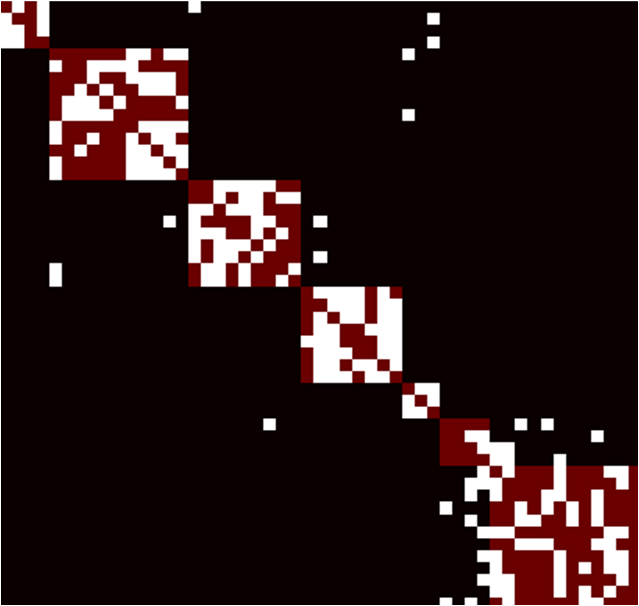} \\
(a) This paper's dissimilarity matrix & (b) $1^{st}$ nearest neighbor  & (c) $5$ nearest neighbors \\
     &  $94\%$ classification rate  & $87\%$ classification rate\\
\includegraphics[width=0.25 \columnwidth]{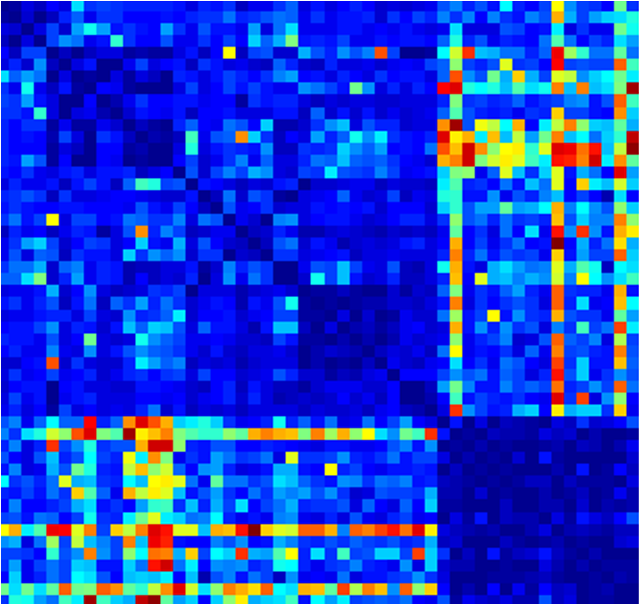} &
\includegraphics[width=0.25 \columnwidth]{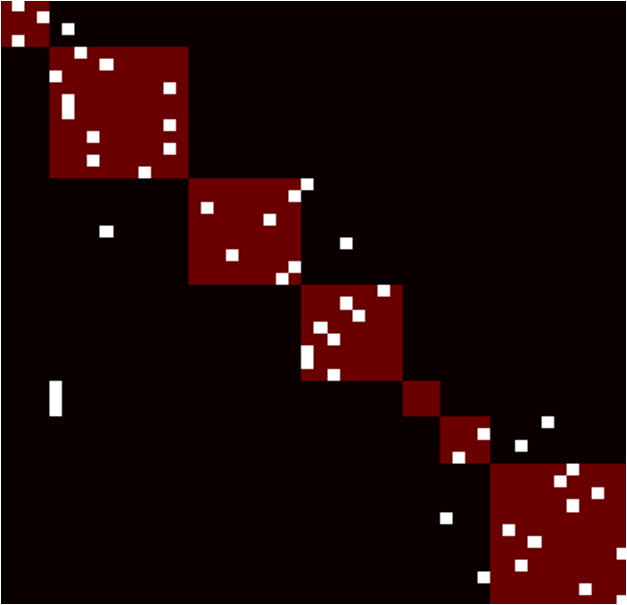} &
\includegraphics[width=0.25 \columnwidth]{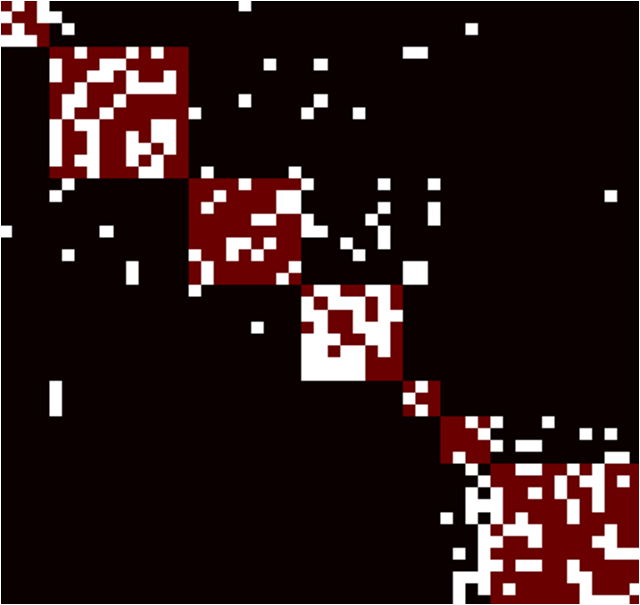} \\
(d) speGH dissimilarity matrix & (e) $1^{st}$ nearest neighbor & (f) $5$ nearest neighbors \\
     & $78\%$ classification rate&  $68\%$ classification rate\\
\end{tabular}
\caption{The dissimilarity matrices for the unit area scaled
Non-rigid World data-set \protect{\cite{BBK06,BBK06b}} calculated
with the conformal Wasserstein (CW) distance given by
our algorithm (a) and the symmetric partial-embedding
Gromov-Hausdorff (speGH) distance \protect{\cite{BBK06,BBK06b}} (d). In the
second column (b,e) we show the ground truth classification matrix
(dark-red) and the first nearest neighbor (white) according to the CW (in b)
and to the speGH (in e) distance for each row. The
third column (c,f) shows the five nearest neighbors (when there
are fewer than five in some category we simply limit ourselves to the number
in that category). Note that white squares should be in dark red
regions to indicate correct classification. Note that with only 100 sampling points,
``Michael'' and ``David'' are not distinguishable, so if a pose of Michael is
among the nearest neighbors of a pose of David, we still count it as a correct classification.}
\label{f:GHvsOURS_dissim_mat_and_knn_mat}
\end{figure}

\begin{figure}[t]
\centering \setlength{\tabcolsep}{0.4cm}
\begin{tabular}{cc}
\includegraphics[width=0.45 \columnwidth]{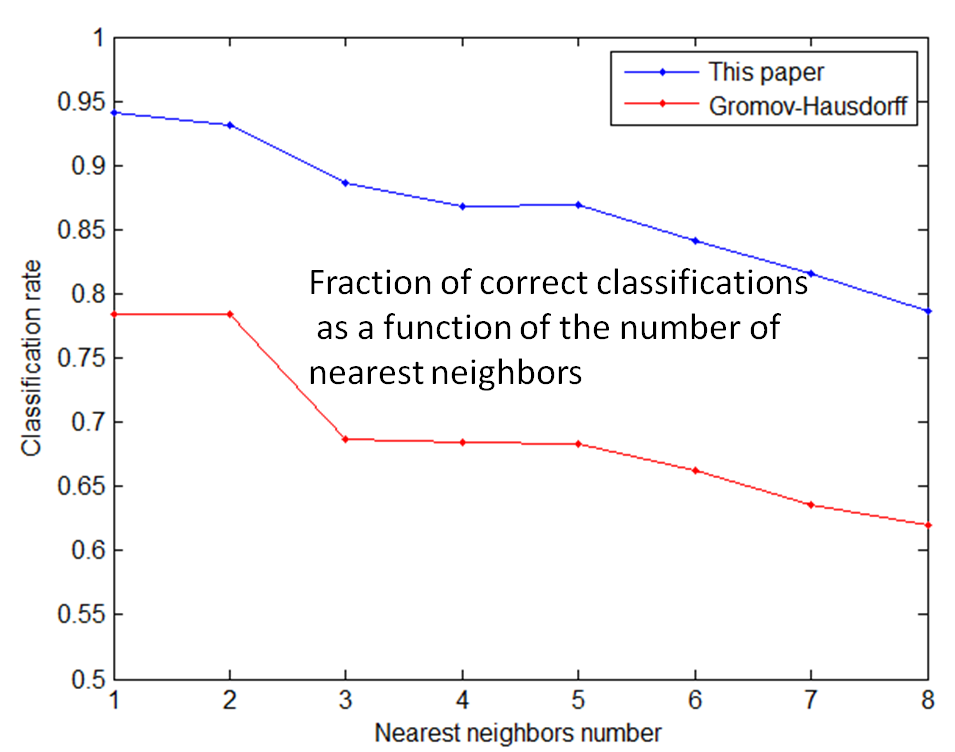} &
\includegraphics[width=0.45 \columnwidth]{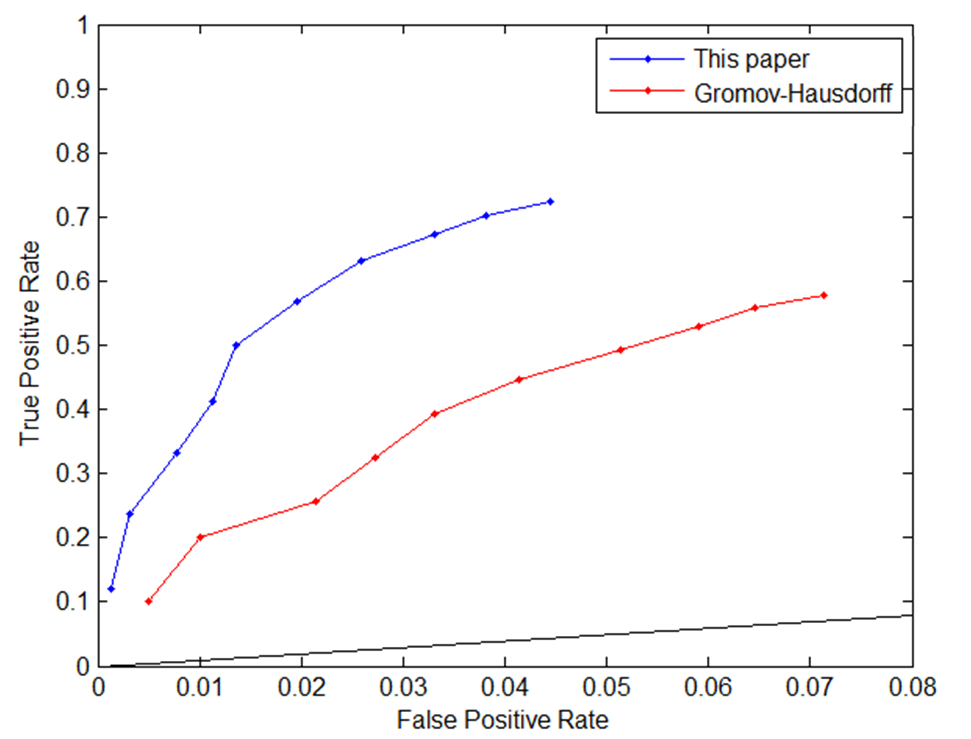}
\\
(a) Correct Classification Rate  & (b) ROC curve (TPR/FPR) \\
\end{tabular}
\caption{Correct Classification Rate and ROC curves for the dissimilarity matrices
produces by our method (blue) and Gromov-Hausdorff-type metric (red). See text for
details.} \label{f:knn_tests}
\end{figure}

\textbf{SHREC 2007 Watertight Benchmark \cite{Giorgi07}.} Our next
experiment deals with a data-set with larger in-class variations;
SHREC 2007 contains 20 categories of models with 20 meshes for each category
 (400 meshes in
total) the categories are, e.g., chairs, 4-legged animals, humans, planes, tables,
etc...

For our experiment, we restricted ourselves to all the meshes of 8 categories that
contained only surfaces of genus zero (since our current algorithm does not
support surfaces of higher genus) and that seemed intrinsically similar; these
categories were: humans, 4-legged animals, ants, hands,
airplanes, teddy-bears, pliers, Armadillos. We ran our
sphere-type algorithm to compute the distance between all pairs. We
tested the ``size'' parameters $A=0.3,0.4,0.5$. The bin
was the same as for the previous data-set. However,
to achieve faster running times (we had about 25,000 comparisons...) we took
only $50$ sample
points. The running time for one pair of objects was around $15$ seconds.

Figure \ref{f:shrek_different_R} shows the dissimilarity matrices (top row)
using the three different values if $A$: $0.3,0.4,0.5$, and the dissimilarity matrix
resulting from combining them:
$$T_d(\mu,\nu) = T_d^{0.3}(\mu,\nu)\cdot T_d^{0.4}(\mu,\nu) \cdot T_d^{0.5}(\mu,\nu).$$
Note that $T_d(\mu,\nu)$ is also a metric and suggests a way to remove the influence
of the size parameter if desired.
The combined distance produced the best classification results as seen in the bottom row,
where for each row the white square shows the nearest neighbor to that object. The dark red areas
represent the different categories. The combined
distance reached the very high classification rate of $95\%$ on this challenging data-set.
Figure \ref{f:knn_tests_shrec} shows for one object in each category its four closest neighbors.
Note the non-rigid nature of some of the objects (e.g., humans, hands), and the substantial deviation
from perfect isometry within class (e.g., 4-legged animals). Figure \ref{f:knn_tests_failure}
demonstrates a partial failure case where although the first two nearest neighbors to the Giraffe
 are within the 4-legged animals category, the third nearest neighbor is the one-armed armadillo,
 which belong to a different category (remember
 that the algorithm is size invariant).

\begin{figure}[h]
\centering \setlength{\tabcolsep}{0.4cm}
\begin{tabular}{cccc}
\includegraphics[width=0.2 \columnwidth]{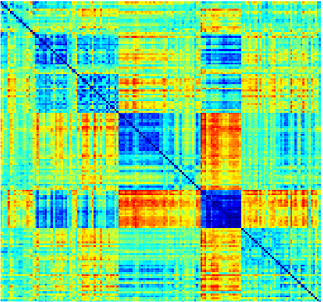} &
\includegraphics[width=0.2 \columnwidth]{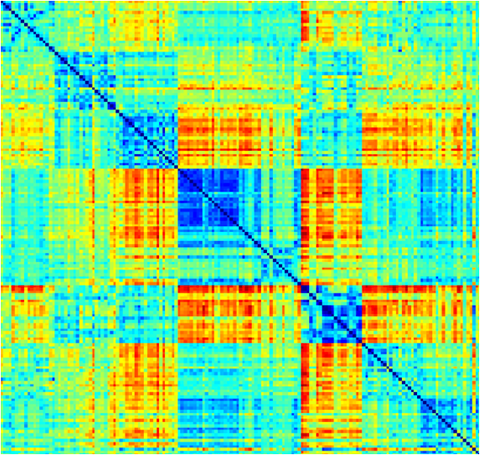}&
\includegraphics[width=0.2 \columnwidth]{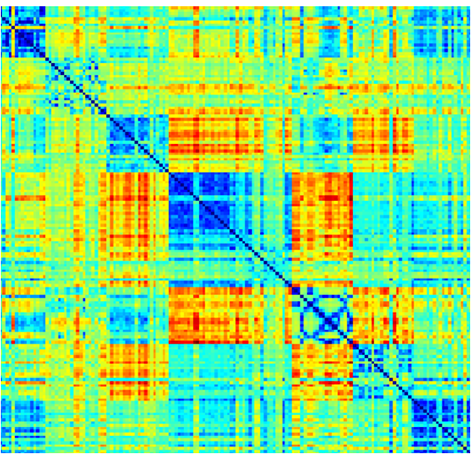}&
\includegraphics[width=0.2 \columnwidth]{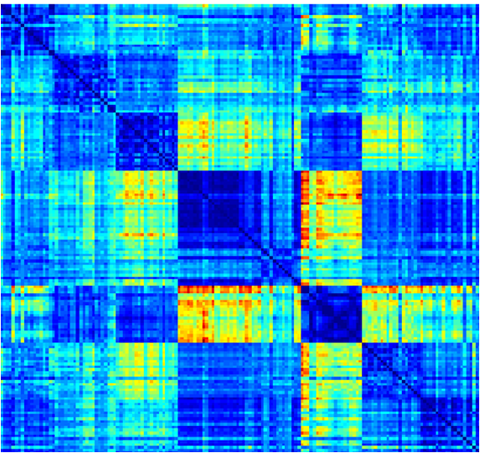}\\
 $R=0.3$ & $R=0.4$  &  $R=0.5$ & combined $R=0.3,0.4,0.5$\\
\includegraphics[width=0.2 \columnwidth]{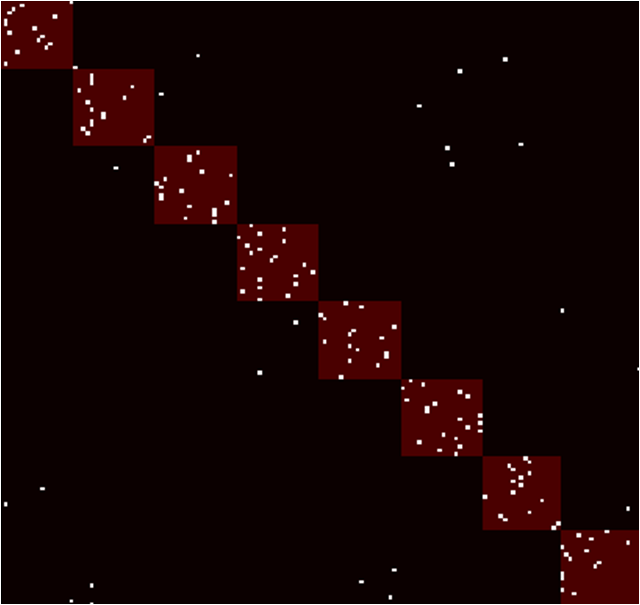} &
\includegraphics[width=0.2 \columnwidth]{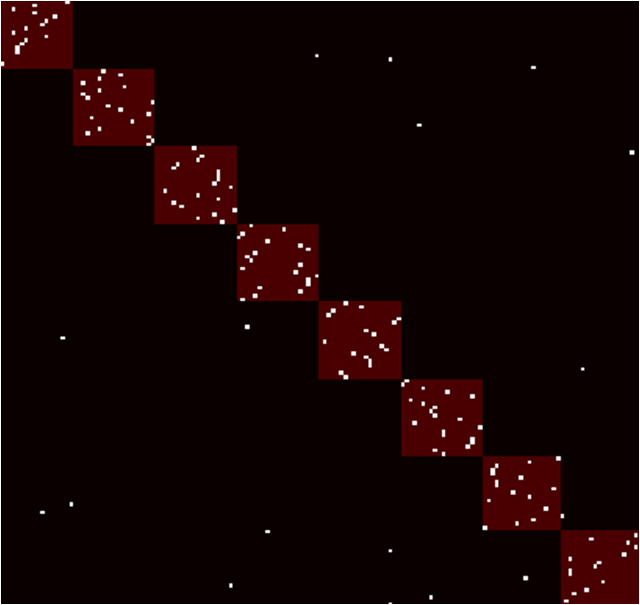}&
\includegraphics[width=0.2 \columnwidth]{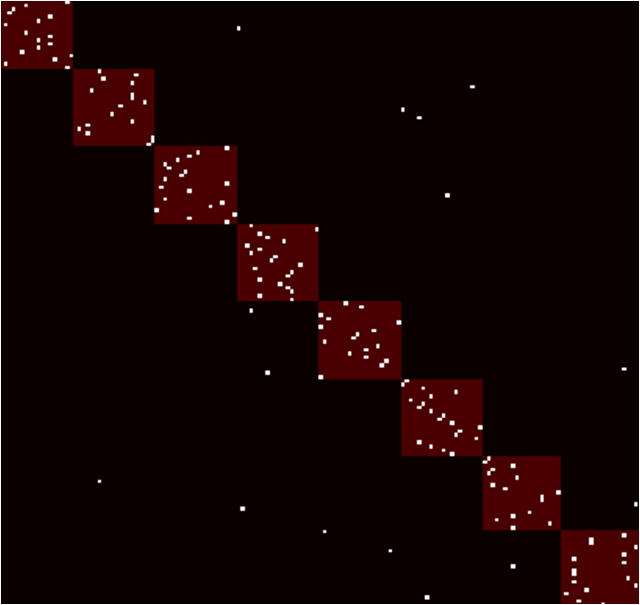}&
\includegraphics[width=0.2 \columnwidth]{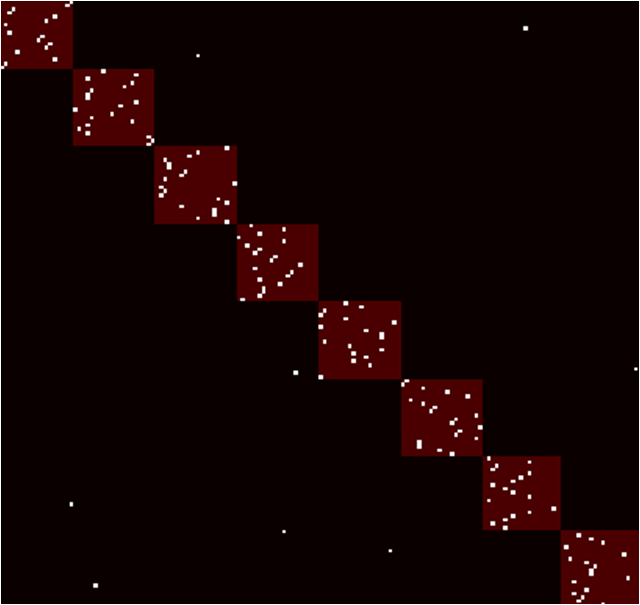}\\
 classification rate: $85\%$ & $89\%$  &  $90\%$ & $95\%$ \\

\end{tabular}
\caption{Dissimilarity matrices (top row) for the SHREC 2007 watertight Benchmark \protect{\cite{Giorgi07}}
with different size parameter: $R=0.3,0.4,0.5$ and their combination (see text for details).
Bottom row shows the first nearest neighbor classification test where white squares denote
the nearest neighbor of that row's object and the dark red area represent correct category classification. } \label{f:shrek_different_R}
\end{figure}

\begin{figure}[h]
\centering \setlength{\tabcolsep}{0.4cm}
\includegraphics[width=0.90\columnwidth]{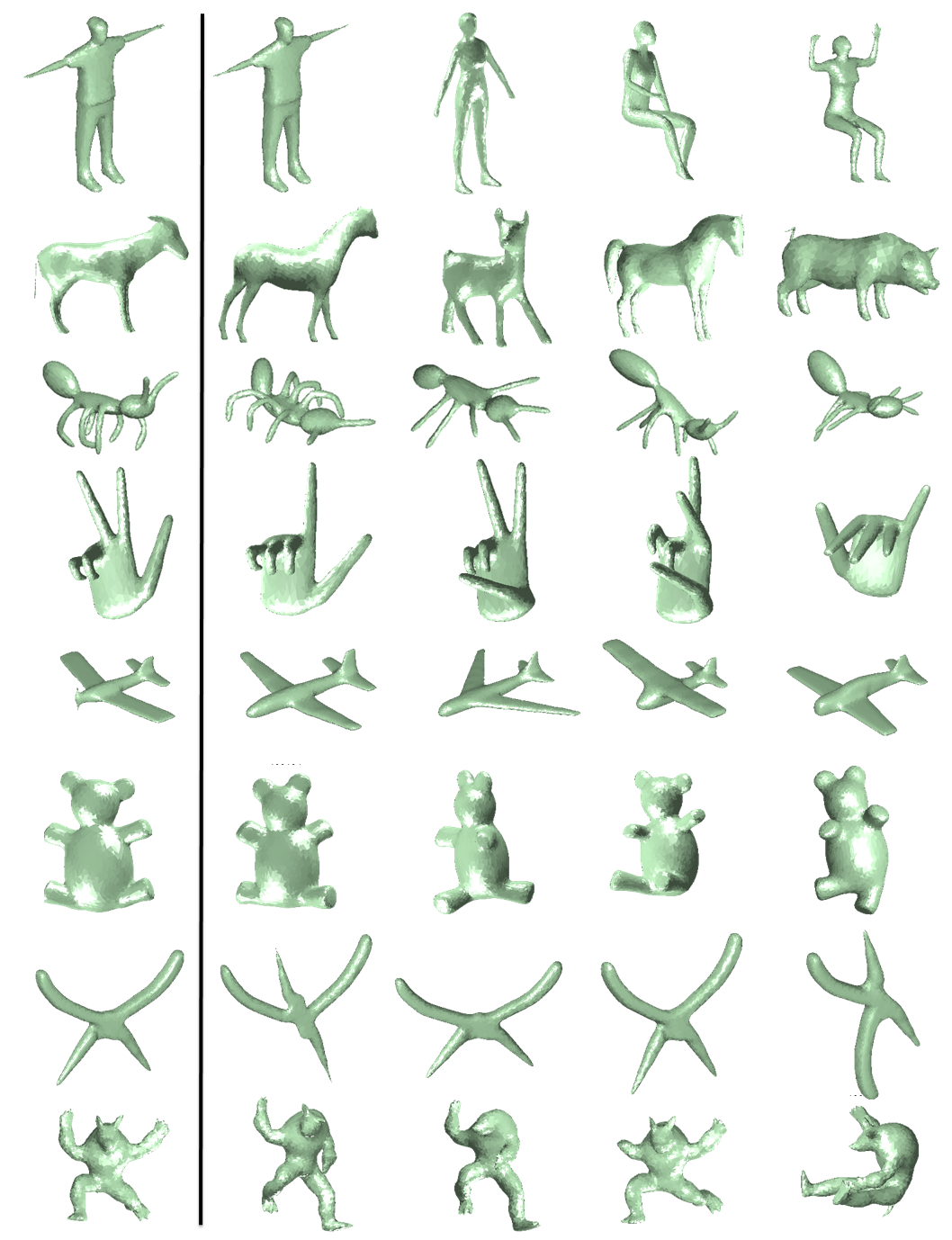}
\caption{SHREC 2007 watertight Benchmark \protect{\cite{Giorgi07}}:
we show the four closest neighbors to each of the objects on the
left side (we show one example from each category). }
\label{f:knn_tests_shrec}
\end{figure}

\begin{figure}[h]
\centering \setlength{\tabcolsep}{0.4cm}
\includegraphics[width=\columnwidth]{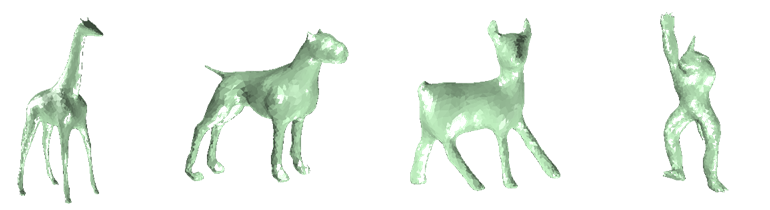}
\caption{SHREC 2007 watertight Benchmark \protect{\cite{Giorgi07}}:
a partially failure case where the giraffe has correct two nearest
neighbors, however its $3^{rd}$ nearest neighbor is a one armed
armadillo. Remember that our matching is scale invariant.}
\label{f:knn_tests_failure}
\end{figure}

\textbf{Synthetic data-set.} This experiment was designed to test the
influence of the size parameter $R$ on the behavior of the distance.
The surfaces we compared are shown in the top row of Figure \ref{f:bumps},
they each have three small bumps, in different positions. At first sight,
one might think that for small $R$, the distance $\d^R(\M,\N)$ based on comparing
neighborhoods of ``size'' $R$, would have trouble distinguishing these objects
from each other. However, one should keep in
mind that the uniformization process is a global one: changing the metric in one
region of the surface would effect the uniformization of other regions
(but influence would decay like appropriate Green's function). Figure \ref{f:bumps}
plot the distance of disk-type model $A$ to the four others, for different $R$ values.
We also show hyperbolic neighborhoods
corresponding to the three different size parameters (color coded the same way as the graphs).
We scaled the distances to have maximum of one (since smaller $R$ results naturally in smaller
distances). Note that even the smallest size value $R=0.25$ already distinguishes between
the different models. Further note, that larger size parameters such as $R=0.75$ results in
slightly more intuitive linear distance behavior. Overall, the size parameter
$R$ does affect the distance, but not in a very significant way.

\begin{figure}[h]
\centering \setlength{\tabcolsep}{0.4cm}
\includegraphics[width=\columnwidth]{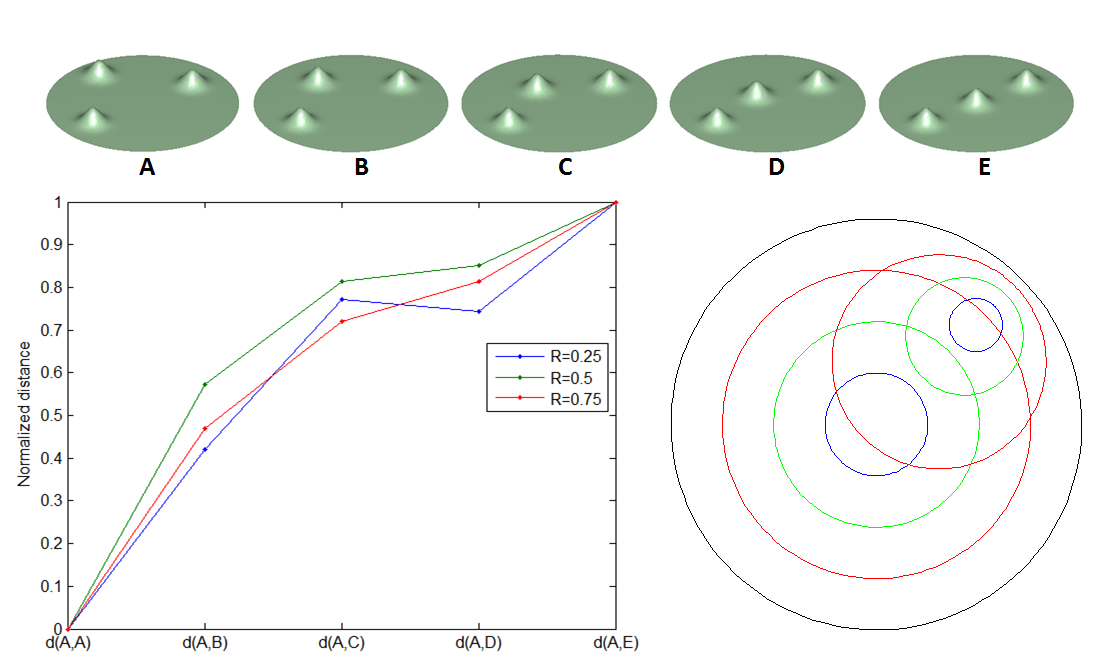}
\caption{Testing influence of the size parameter $R$ to sperate
identical models with small features. We compare the distances of
the disk model marked with $A$ to all other models $A-E$ using three
different size parameters $R=0.25,0.5,0.75$. We also show examples
of hyperbolic disk of these radii as used by our disk-type algorithm
(the are color coded like the graph lines).} \label{f:bumps}
\end{figure}

\begin{figure}[h]
%\begin{figure}[h]
\centering
%\begin{table}
\begin{tabular}{rcccc}
\includegraphics[width=0.2\columnwidth]{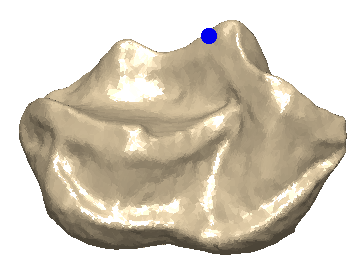} &
\includegraphics[width=0.2\columnwidth]{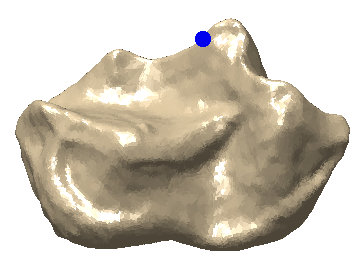} &
\includegraphics[width=0.2\columnwidth]{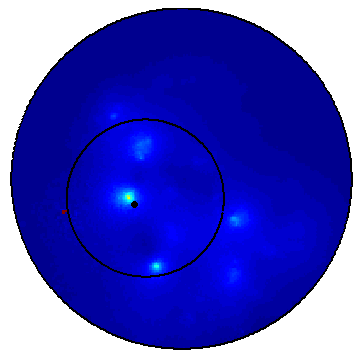} &
\includegraphics[width=0.2\columnwidth]{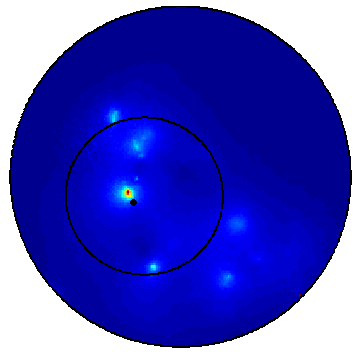} &
\includegraphics[width=0.2\columnwidth]{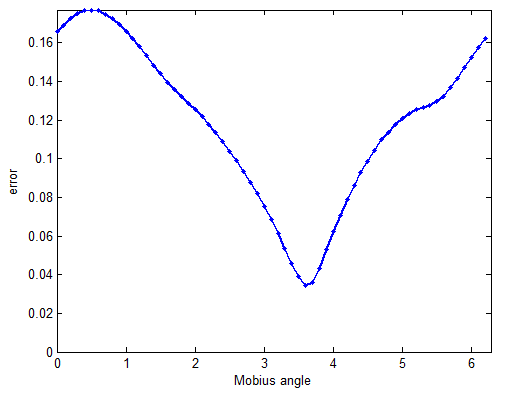} \\
(a) &&&& Good pair (a)\\
\includegraphics[width=0.2\columnwidth]{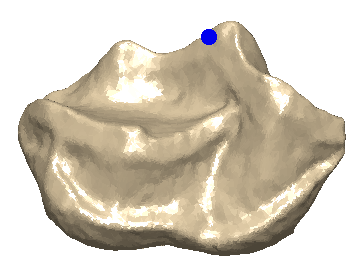} &
\includegraphics[width=0.2\columnwidth]{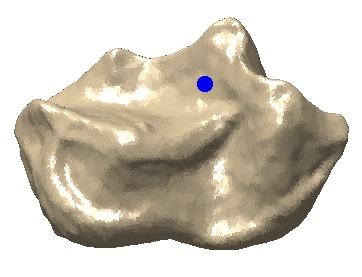} &
\includegraphics[width=0.2\columnwidth]{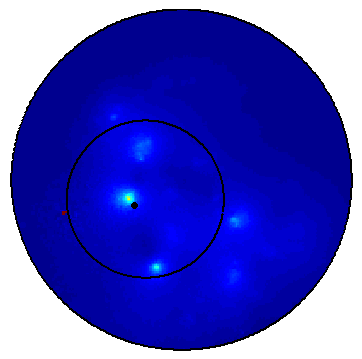} &
\includegraphics[width=0.2\columnwidth]{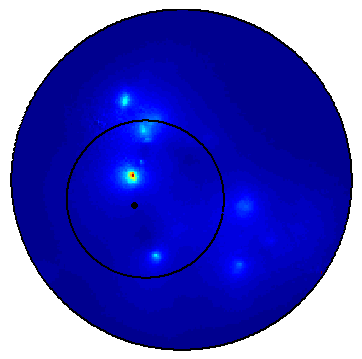} &
\includegraphics[width=0.2\columnwidth]{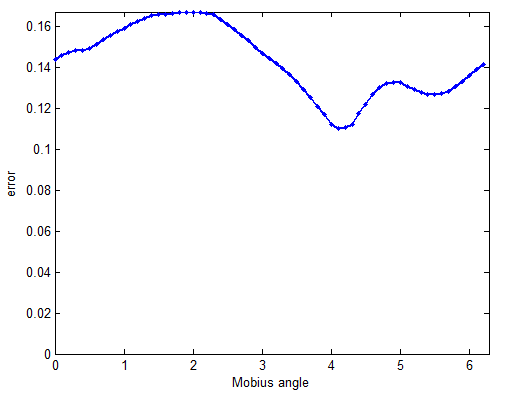} \\
(b) & & &  & Erroneous pair (b)
\end{tabular}
\caption{Calculation of the local distance
$d^R_{\mu,\nu}(\cdot,\cdot)$ between pairs of points on two
different surfaces (each row shows a different pair of points; the
two surfaces are the same in the top and bottom rows). The first row
shows a ``good'' pair of points together with the alignment of the
conformal densities $\mu,m^*\nu$ based on the best M\"{o}bius
transformation $m$ minimizing $d^R_{\mu,\nu}$. The plot of this
latter integral as a function of $m$ (parameterized by $\sigma \in
[0,2\pi)$, see (\ref{e:disk_mobius})) is shown in the right-most
column. The second row shows a ``bad'' correspondence which indeed
leads to a higher local distance
$d^R_{\mu,\nu}$.}\label{fig:good_bad_pair_correspondence}

%\end{table}
\end{figure}

\textbf{Primate molar teeth.} Finally, we present a few experimental
results related to a biological application; in a case study of the
use of our approach to the characterization of mammals by the
surfaces of their molars, we compare high resolution scans of the
masticating surfaces of molars of several lemurs, which are small
primates living in Madagascar. Traditionally, biologists
specializing in this area carefully determine landmarks on the tooth
surfaces, and measure characteristic distances and angles involving
these landmarks. A first stage of comparing different tooth surfaces
is to identify correspondences between landmarks.  Figure
\ref{fig:good_bad_pair_correspondence} illustrates how
$d^R_{\mu,\nu}(z,w)$ (disk-type) can be used to find corresponding
pairs of points on two surfaces by showing both a ``good'' and  a
``bad'' corresponding pair. The left two columns of the figure show
the pair of points in each case; the two middle columns show the
best fit after applying the minimizing M\"{o}bius on the
corresponding disk representations; the rightmost column plots $
\int_{\Omega_{z_0,R}} \,|\,\mu(z) -
(m_{z_0,w_0,\sigma}^*\nu)(z)\,|\, d\vol_H(z)$, the value of the
``error'', as a function of parameter $\sigma$, parameterizing the
\Mbs transformations that map a given point $z_0$ to another given
point $w_0$ (they are parameterized over $S^1$, see Lemma 3.5 in
\cite{Lipman_Daubechies:2010:polytimesurfcomp} ). The ``best''
corresponding point $w_0$ for a given $z_0$ is the one that produces
the lowest minimal value for the error, i.e. the lowest
$d^R_{\mu,\nu}(z_0,w_0)$.

Figure \ref{fig:120_corrs} show the top 120 most consistent
corresponding pairs (in groups of 20) for two molars belonging to
lemurs of different species. Corresponding pairs are indicated by
highlighted points of the same color. These correspondences have
surprised the biologists from whom we obtained the data sets; their
experimental measuring work, which incorporates finely balanced
judgment calls, had defied earlier automatization attempts.

Once the differences and similarities between molars from different animals
have been quantified, they can be used (as part of an approach) to
classify the different individuals. Figure
\ref{fig:distance_graph_embedded} illustrates a preliminary result
that illustrates the possibility of such classifications based
on the distance operator between surfaces introduced in this paper.
The figure illustrates the pairwise distance matrix for eight molars,
coming from individuals in four different species (indicated by color).
The clustering was based on only the distances between the molar surfaces;
it clearly agrees with the clustering by species, as communicated to us
by the biologists from whom we obtained the data sets.

One final comment
regarding the computational complexity of our method. There are two main
parts: the preparation of the distance matrix $d_{ij}$ and the
linear programming optimization. For the linear programming
part we used a Matlab interior point implementation with $N^2$
unknowns, where $N$ is the number of points spread on the
surfaces. In our experiments, the optimization typically terminated
after $15-20$ iterations for $N=150-200$ points, which took about 2-3
seconds. The computation of the similarity distance $d_{ij}$
took longer, and was the bottleneck in our experiments. We separate
the disk-type and the sphere-type algorithms.

For the sphere-type algorithm if we use $N=\L$ sample points ($Q$)
on each surface (see Section \ref{s:generalization_to_sphere_type})
then for each pair we compare the difference (using fixed size bin
structure)
 of the
discrete conformal densities for fixed number of M\"{o}bius
transformations. This results in $O(N^3)$ algorithm for computing
the distance matrix $d_{ij}$. In our experiments the total distance
computation time (including linear programming optimization) was
around $15$ seconds for $N=\L=50$ (in the SHREC 2007 data-set), to
$90$ second per comparison for $N=\L=100$ (in the Non-rigid world
data-set). In the sphere-type examples we have used 2.2GHz AMD
Opteron processor. The sphere-type algorithm was coded completely in
Matlab and was not optimized.

For the disk-type algorithm, if we spread $N$ points on each
surface, and use them all to interpolate the conformal factors
$\Gamma_\mu, \Gamma_\nu$, if we use $P$ points in the integration
rule, and take $L$ points in the M\"{o}bius discretization (see
Section \ref{s:the_discrete_case_implementation} for details) then
each approximation of  $d^R_{\mu,\nu}(z_i,w_j)$ by
(\ref{e:discrete_approx__d_mu_nu(z_i,w_j)_bis}) requires $O(L \cdot
P \cdot N)$ calculations, as each evaluation of
$\Gamma_\mu,\Gamma_\nu$ (Thin-Plate Spline approximation we use to
interpolate the conformal densities, described in
Appendix~\ref{a:appendix B}) takes $O(N)$ and we need $L\cdot P$ of
those. Since we have $O(N^2)$ distances to compute, the computation
complexity for calculating the similarity distance matrix $d_{ij}$
is $O(L\cdot P \cdot N^3)$. This step was coded in C++ (and
therefore the time difference to the sphere-type case) and took
$3.5$ seconds for $N=50$, $51$ seconds for $N=100$, under $5$
minutes for $N=150$ and two hours for $N=300$ (in these examples we
took $P\approx N$). However, also in this case we have not
 optimized the algorithm and we believe these times can be
reduced significantly. The disk-type algorithm ran on Intel Xeon (X5650) 2.67GHz processor.

%To conclude, we defined and introduced an algorithm for measuring
%isometry-invariant distance between disk-type surfaces, and finding
%point correspondences automatically. In term of future work, we
%would like to expand this new type of distances to wider class of
%surfaces. We would also like to further analyze the algorithm and
%its convergence properties. Finally, we would like to use the
%discrete correspondences to construct a continuous map between the
%surfaces.

\begin{figure}[h]
%\begin{figure}[h]
\centering
%\begin{table}
\begin{tabular}{cccc}
\includegraphics[width=0.3\columnwidth]{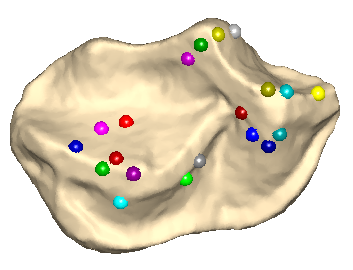} &
\includegraphics[width=0.3\columnwidth]{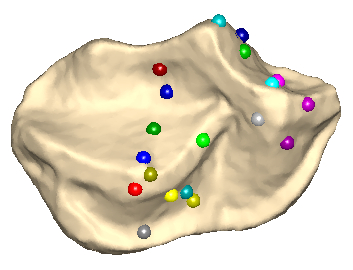} &
\includegraphics[width=0.3\columnwidth]{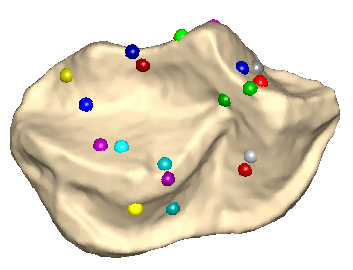} \\
\includegraphics[width=0.3\columnwidth]{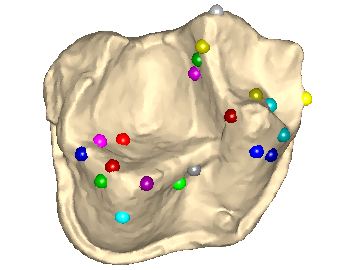} &
\includegraphics[width=0.3\columnwidth]{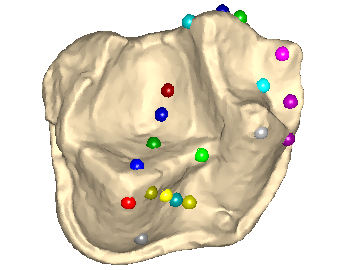} &
\includegraphics[width=0.3\columnwidth]{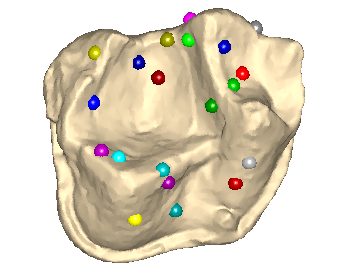} \\
\hline
\includegraphics[width=0.3\columnwidth]{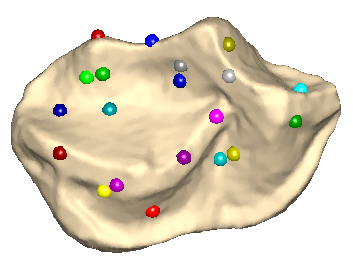} &
\includegraphics[width=0.3\columnwidth]{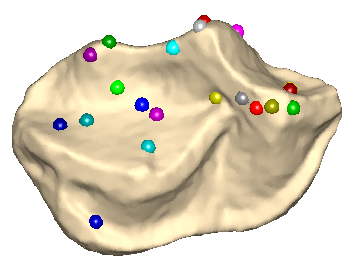} &
\includegraphics[width=0.3\columnwidth]{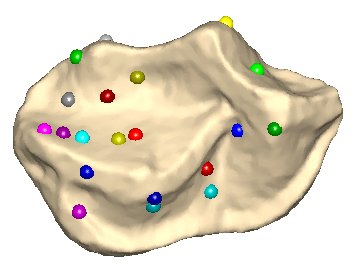} \\
\includegraphics[width=0.3\columnwidth]{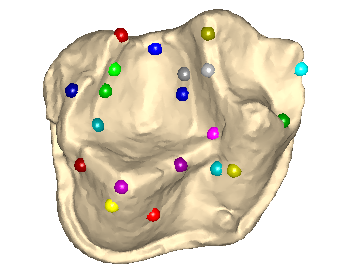} &
\includegraphics[width=0.3\columnwidth]{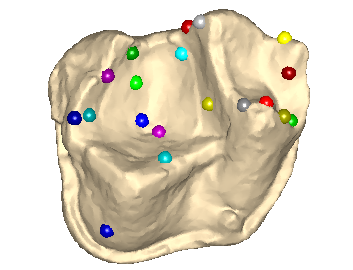} &
\includegraphics[width=0.3\columnwidth]{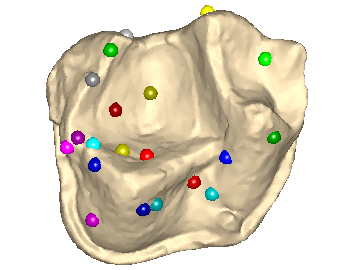} \\
\end{tabular}
\caption{The top 120 most consistent corresponding pairs between two
molar teeth models.} \label{fig:120_corrs}
%\end{table}
\end{figure}

\begin{figure}[h]
%\begin{figure}[h]
\centering
%\begin{table}
\includegraphics[width=0.9\columnwidth]{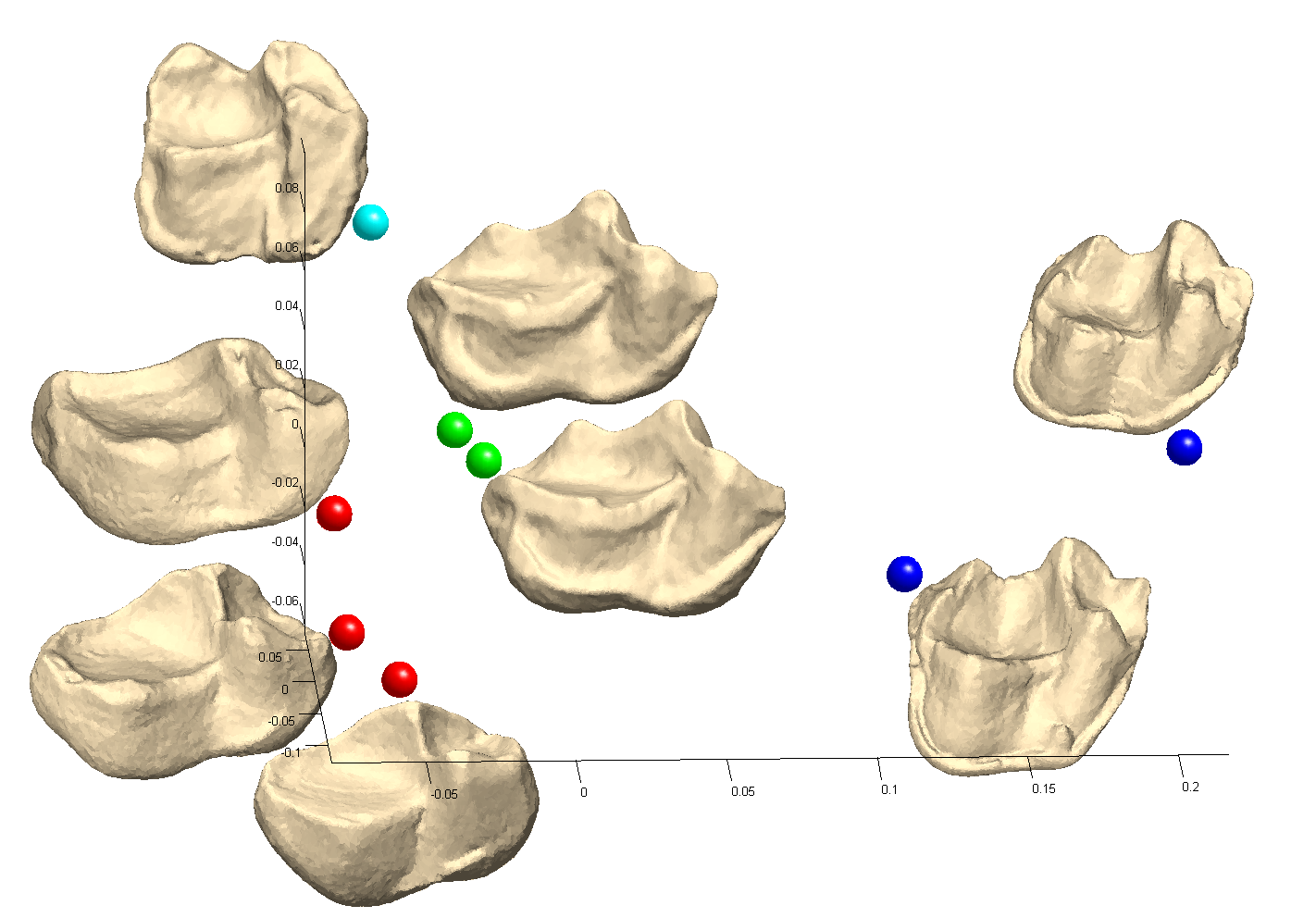}
\caption{Embedding of the distance graph of eight teeth models using
multi-dimensional scaling. Different colors represent different
lemur species. The graph suggests that the geometry of the teeth
might suffice to classify species.}
\label{fig:distance_graph_embedded}
%\end{table}
\end{figure}

%% -----------------------------------------
%\section{Conclusions}

%% -----------------------------------------
\section{Acknowledgments}
The authors would like to thank C\'{e}dric Villani and Thomas
Funkhouser for valuable discussions. We are grateful to Jukka
Jernvall, Stephen King, and Doug Boyer for providing us with the
tooth data sets, and for many interesting comments. We would like to
thank the anonymous reviewers that challenged us to improve our
manuscript with excellent comments and suggestions. ID gratefully
acknowledges (partial) support for this work by NSF grant
DMS-0914892, and by an AFOSR Complex Networks grant; YL thanks the
Rothschild foundation for postdoctoral fellowship support.

\bibliographystyle{amsplain}
\bibliography{kantorovich_for_surfaces_II}

%----------------------------
\appendix
\renewcommand{\thesection}{ \Alph{section}}
\section{}
\label{a:appendix A}
\renewcommand{\thesection}{\Alph{section}}
%\section*{Appendix A}
We prove Theorem \ref{thm:convergence_of_numerical_quadrature}. We
start with a simple lemma showing that all M\"{o}bius
transformations restricted to $\Omega_{0,T}$, $T < \infty$, are
Lipschitz with a universal constant, for which we provide an upper
bound.
\begin{lem}\label{lem:mobius_is_lipschitz}
A M\"{o}bius transformation $m \in \Md$ restricted to
$\Omega_{0,T}$, $T<\infty$ is Lipschitz continuous with Lipschitz
constant $C_m \leq \frac{1-|a|^2}{(1-r_T|a|)^2}$, where $a =
m^{-1}(0)$ and $r_T=tanh(T)$.
\end{lem}
\begin{proof}
Denote $m(z)=e^{\bfi \theta}\frac{z-a}{1-z\bbar{a}}$. Then, for
$z,w\in \Omega_{0,T}$ we have
\begin{align*}
\babs{m(z)-m(w)} & \leq \babs{e^{\bfi \theta}\frac{z-a}{1-z\bbar{a}}
- e^{\bfi \theta}\frac{w-a}{1-w\bbar{a}} } \leq
\babs{\frac{(z-a)(1-w\bbar{a})-(w-a)(1-z\bbar{a})}{(1-z\bbar{a})(1-w\bbar{a})}}\\
& \leq \babs{\frac{(z-w)(1-\abs{a}^2)}{(1-z\bbar{a})(1-w\bbar{a})}}
\leq |z-w|\frac{1-\abs{a}^2}{(1-r_T\abs{a})^2}.
\end{align*}
\end{proof}

Next we prove: \\
{\bf Theorem \ref{thm:convergence_of_numerical_quadrature}} {\em For
Lipschitz continuous $\mu,\nu$,
\[
\left|\,d^R_{\mu,\nu}(z_i,w_j)- \min_{m(z_i)=w_j}\sum_k \,\alpha_k
\,\left|\,\mu(\widetilde{m}_i (p_k)) - \nu(m(\widetilde{m}_i
(p_k)))\,\right|\, \right| \leq C\,\varphi_E\left( \set{p_k}
\right)~,
\]
where the constant $C$ depends only on $\mu,\nu,R$. }
\begin{proof}
First, denote $f(z) = \babs{\mu(\wt{m}_i (z)) - \nu(m(\wt{m}_i
(z)))}$. Then,
\begin{eqnarray}\label{e:approx_quad_eq1}
\babs{\int_{\Omega_{0,R}}f(z)d\vol_H(z)-\min_{m(z_i)=w_j}\sum_k
\alpha_k
f(p_k)} & \leq \sum_k \int_{\Delta_k} \babs{f(z)-f(p_k)}d\vol_H(z)  \\
\nonumber & \leq \omega^{\Omega_{0,R}}_{f}\parr{\varphi_E
\parr{\set{p_k}}} \int_{\Omega_{0,R}}d\vol_H,
\end{eqnarray}
where $\Delta_k$ are the intersections of $\Omega_{0,R}$ with the Euclidean Voronoi cells defined by the
centers $p_k$, and the modulus of
continuity $\omega^{\Omega_{0,R}}_{f}\parr{h} = \sup_{|z-w|<h;
z,w\in \Omega_{0,R}}\abs{f(z)-f(w)}$ is used. Note that
\begin{equation}\label{e:approx_quad_eq2}
    \omega^{\Omega_{0,R}}_{f} \leq \omega^{\Omega_{0,R}}_{\mu \circ \wt{m}_i} +
\omega^{\Omega_{0,R}}_{\nu \circ m \circ \wt{m}_i }.
\end{equation}
Denote the Lipschitz constants of $\mu,\nu$ by $C_\mu,C_\nu$,
respectively. From Lemma \ref{lem:mobius_is_lipschitz} we see that,
for $z,w \in \Omega_{0,R}$,
\begin{align*}
\babs{\mu(\wt{m}_i(z))-\mu(\wt{m}_i(w))} & \leq C_\mu
\babs{\wt{m}_i(z) - \wt{m}_i(w)}  \leq C_\mu
\frac{1-|a|^2}{(1-r_R|a|)^2} \abs{z-w}  \leq C_\mu
\frac{1}{(1-r_R)^2} \abs{z-w},
\end{align*}
which is independent of $\wt{m}_i$. Similarly,
\begin{align*}
\babs{\nu(m(\wt{m}_i(z)))-\nu(m(\wt{m}_i(w)))} & \leq C_\nu
\babs{m(\wt{m}_i(z)) - m(\wt{m}_i(w))} \leq C_\nu
\frac{1}{(1-r_R)^2} \abs{z-w},
\end{align*}
which is independent of $m,\wt{m}_i$. Combining these with eq.
(\ref{e:approx_quad_eq1}-\ref{e:approx_quad_eq2}) we get
$$\babs{\int_{\Omega_{0,R}}f(z)d\vol_H(z)-\min_{m(z_i)=w_j}\sum_k \alpha_k
f(p_k)} \leq
\parr{C_\mu+C_\nu}\frac{\int_{\Omega_{0,R}}d\vol_H}{(1-r_R)^2}\varphi_E \parr{\set{p_k}},$$
which finishes the proof.
\end{proof}

Finally, we prove: 

{\bf Theorem \ref{thm:final_approx_to_d^R_mu,nu}}
{\em For Lipschitz continuous $\mu,\nu$,
\[
\abs{\,d^R_{\mu,\nu}(z_i,w_j)- \hh{d}^R_{\mu,\nu}(z_i,w_j)} \leq
C_1\,\varphi_E\left( \set{p_k} \right) + C_2 \,L^{-1}~,
\]
where the constants $C_1,C_2$ depend only on $\mu,\nu,R$. }
\begin{proof}
In view of Theorem \ref{thm:convergence_of_numerical_quadrature} it
is sufficient to prove that
$$\abs{\wh{d}^R_{\mu,\nu}(z_i,w_j)-\hh{d}^R_{\mu,\nu}(z_i,w_j)}\leq C \, L^{-1},$$
for an appropriate constant $C$ depending only upon $\mu,\nu,R$. Denote
by $m_{i,j}$ the minimizer of $\wh{d}^R_{\mu,\nu}(z_i,w_j)$. Then
\begin{align*}
\abs{\wh{d}^r_{\mu,\nu}(z_i,w_j) - \hh{d}^r_{\mu,\nu}(z_i,w_j)} & =
\mathop{\min}_{\ell=1..L}\abs{ \sum_k \alpha_k \bigbrac{
\abs{\mu(\wt{m}_i(p_k)) - \nu(m_{i,j}(\wt{m}_i(p_k))) } -
\abs{\mu(\wt{m}_i(p_k)) - \nu(m_{z_i,w_j,2\pi \ell/
L}(\wt{m}_i(p_k))) }} }\\
& \leq \mathop{\min}_{\ell=1..L} \sum_k \alpha_k \Big |
\nu(m_{i,j}(\wt{m}_i(p_k))) - \nu(m_{z_i,w_j,2\pi \ell/
L}(\wt{m}_i(p_k)))\Big | \\
& \leq C_\nu \mathop{\min}_{\ell=1..L} \sum_k \alpha_k \Big |
m_{i,j}(\wt{m}_i(p_k)) - m_{z_i,w_j,2\pi \ell/ L}(\wt{m}_i(p_k))\Big
| \\
& \leq \frac{C_\nu}{(1-r_R)^2} \mathop{\min}_{\ell=1..L} \sum_k
\alpha_k \Big | \wt{m}_j^{-1}(m_{i,j}(\wt{m}_i(p_k))) -
\wt{m}_j^{-1}(m_{z_i,w_j,2\pi \ell/ L}(\wt{m}_i(p_k)))\Big |
\end{align*}
where, as in previous theorem, we denote by $C_\nu$ the Lipschitz
constant of $\nu$ in $\D$, and in the last inequality we have used
Lemma \ref{lem:mobius_is_lipschitz} while taking $\wt{m}_j^{-1}$ as
defined in Section \ref{s:the_discrete_case_implementation}. From
eq.(\ref{e:def_m_zi_wi_2pi l_div_L}) we have that
\begin{align*}
\abs{\wh{d}^r_{\mu,\nu}(z_i,w_j) - \hh{d}^r_{\mu,\nu}(z_i,w_j)} &
\leq \frac{C_\nu}{(1-r_R)^2} \mathop{\min}_{\ell=1..L} \sum_k
\alpha_k \Big | \wt{m}_j^{-1}(m_{i,j}(\wt{m}_i(p_k))) - e^{\bfi 2\pi
\ell/ L} p_k\Big |.
\end{align*}
Now note that $\wt{m}_j^{-1}\circ m_{i,j} \circ \wt{m}_i \in \Md$ also
fixes the origin; it follows that  $\wt{m}_j^{-1}\circ m_{i,j} \circ
\wt{m}_i(z) = e^{\bfi \theta} z$ for some
$\theta\in[0,2\pi)$. We therefore have
\begin{align*}
\abs{\wh{d}^r_{\mu,\nu}(z_i,w_j) - \hh{d}^r_{\mu,\nu}(z_i,w_j)} &
\leq \frac{C_\nu}{(1-r_R)^2} \mathop{\min}_{\ell=1..L} \sum_k
\alpha_k \Big |e^{\bfi \theta}  p_k - e^{\bfi 2\pi \ell/ L} p_k\Big
| \\
& \leq \frac{r_R \, C_\nu}{(1-r_R)^2} \mathop{\min}_{\ell=1..L}
\sum_k \alpha_k  \Big |e^{\bfi \theta}  - e^{\bfi 2\pi \ell/ L} \Big
| \\
& \leq \frac{r_R \, C_\nu}{(1-r_R)^2} \parr{\sum_k \alpha_k}
\mathop{\min}_{\ell=1..L}
 \Big |e^{\bfi \theta}  - e^{\bfi 2\pi \ell/ L} \Big
| \\ & \leq \frac{r_R \, C_\nu \, 2\pi}{L (1-r_R)^2}
\int_{\Omega_{0,R}}d\vol_H.
\end{align*}

\end{proof}

\renewcommand{\thesection}{ \Alph{section}}
\section{}
\label{a:appendix B}
\renewcommand{\thesection}{\Alph{section}}
%\section*{Appendix B}
In this appendix we review a few basic notions such as the
representation of (approximations to) surfaces by faceted, piecewise
flat approximations, called {\em meshes}, and discrete conformal
mappings; the conventions we describe here are the same as adopted
in \cite{Lipman:2009:MVF}.

We denote
a triangular mesh by the triple $M = (\V, \E, \F)$, where
$\V=\{v_i\}_{i=1}^m \subset \Real^3$ is the set of vertices,
$\E=\{e_{i,j}\}$ the set of edges, and $\F=\{f_{i,j,k}\}$  the set
of faces (oriented $i\rightarrow j \rightarrow k$).
\nomenclature{$M=(V,E,F) \ (N)$}{Mesh approximating surface $\M$
($N$) with vertices $V$, edges $E$, and faces $F$.} When dealing
with a second surface, we shall denote its mesh by $N$. In this appendix, we assume
our mesh is homeomorphic to a disk.

Next, we introduce ``conformal mappings'' of a mesh to the unit
disk. Natural candidates for discrete conformal mappings are not
immediately obvious. In particular, it is not possible
to use a continuous
piecewise-affine map the restriction of which to each triangle would be
a (positively oriented) similarity transformation: continuity would force 
the similarity transformations of any two adjacent
triangles to coincide, meaning such
a map would be globally a similarity. A different approach uses the
notion of discrete harmonic and discrete conjugate harmonic
functions due to Pinkall and Polthier \cite{Pinkall93,Polthier05} to
define a discrete conformal mapping on the mid-edge mesh. The
mid-edge mesh $\mM = (\mV, \mE, \mF)$ of a given mesh $M=(\V, \E,
\F)$ is defined as follows. \nomenclature{$\mM = (\mV, \mE, \mF) \
(\mN)$}{The mid-edge mesh.}  For the vertices $\mv_r \in \mV$, we
pick the mid-points of the edges of the mesh $M$; we call these the
mid-edge points of $M$. There is thus a $\mv_r \in \mV$
corresponding to each edge $e_{i,j} \in \E$. If $\mv_s$ and $\mv_r$
are the mid-points of edges in $\E$ that share a vertex in $M$, then
there is an edge $\me_{s,r} \in \mE$ that connects them. It follows
that for each face $f_{i,j,k} \in \F$ we can define a corresponding
face $\mf_{r,s,t} \in \mF$, the vertices of which are the mid-edge
points of (the edges of) $f_{i,j,k}$; this face has the same
orientation as $f_{i,j,k}$. Note that the mid-edge mesh is not a
manifold mesh, as illustrated by the mid-edge mesh in Figure
\ref{f:discrete_type_1}, shown together with its ``parent'' mesh: in
$\mM$ each edge ``belongs'' to only one face $\mF$, as opposed to a
manifold mesh, in which most edges (the edges on the boundary are
exceptions) function as a hinge between two faces. This ``lace''
structure makes a mid-edge mesh more flexible: it turns out that it
is possible to define a piecewise linear map that makes each face in
$\mF$ undergo a pure scaling (i.e. all its edges are shrunk or
extended by the same factor) and that simultaneously flattens the
whole mid-edge mesh (we provide more details on this flattening
below). By extending this back to the original mesh, we thus obtain
a map from each triangular face to a similar triangle in the plane;
these individual similarities can be ``knitted together'' through
the mid-edge points, which continue to coincide (unlike most of the
vertices of the original triangles).

We have thus relaxed the problem, and we define a map via a
similarity on each triangle, with continuity for the complete map at
only \emph{one} point of each edge, namely the mid point. This
procedure was also used in \cite{Lipman:2009:MVF}; for additional
implementation details we refer the interested reader (or
programmer) to that paper, which includes a pseudo-code.

\begin{figure}[ht]
\centering \setlength{\tabcolsep}{0.4cm}
\begin{tabular}{@{\hspace{0.0cm}}c@{\hspace{0.2cm}}c@{\hspace{0.0cm}}}
%c@{\hspace{0.0cm}}}
\includegraphics[width=0.4\columnwidth]{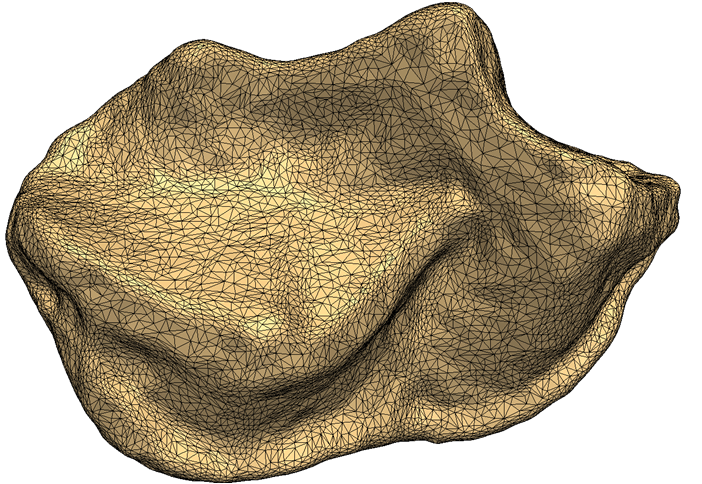}
&
\includegraphics[width=0.4\columnwidth]{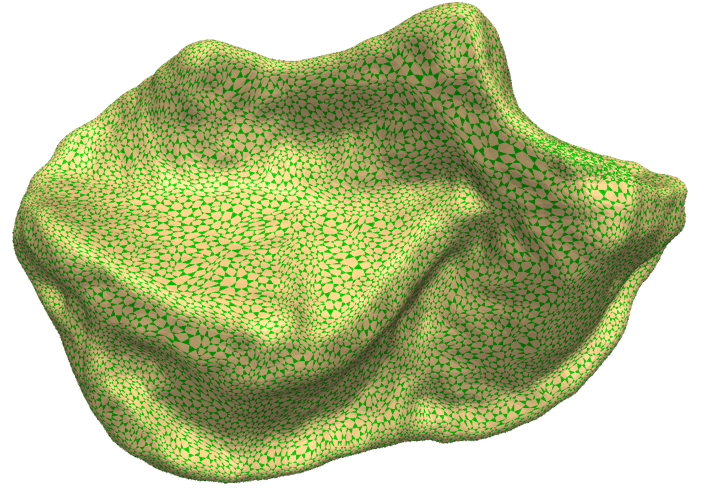} \\
Discrete mesh& Mid-edge mesh\\
\includegraphics[width=0.4\columnwidth]{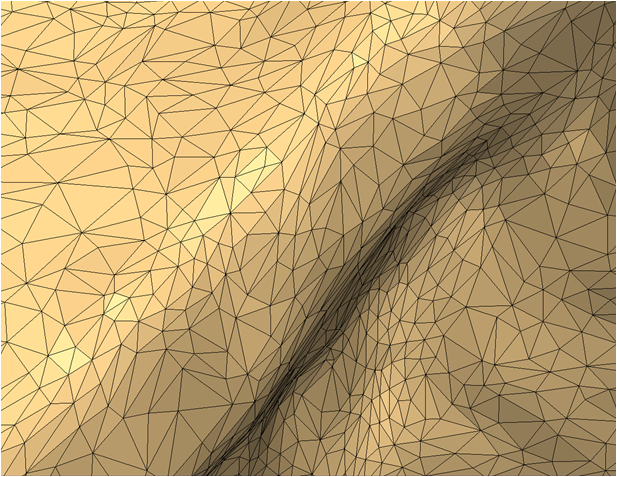}
&
\includegraphics[width=0.4\columnwidth]{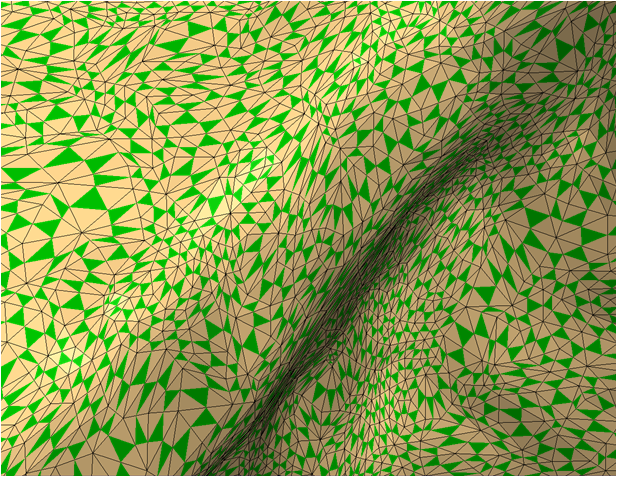} \\
Surface mesh zoom-in & Mid-edge mesh zoom-in
\end{tabular}
\caption{A mammalian tooth surface mesh, with the corresponding
mid-edge mesh. In the mid-edge mesh, the faces are the smaller green
triangles within the faces of the surface
mesh.}\label{f:discrete_type_1}
\end{figure}

This flattening procedure
maps the boundary of the mesh onto a region with a straight
horizontal slit (see Figure \ref{f:discrete_type_2}, where the
boundary points are marked in red) \cite{Lipman:2009:MVF}. 
We can assume, without loss of
generality, that this slit coincides with the interval $[-2, 2]
\subset \C$. Now applying the inverse of the holomorphic map $z=w+\frac{1}{w}$
maps $\mathbb{C}\setminus [-2,2]$ conformally to the disk $\D$,  
with the slit at $[-2,2]$
mapped to the boundary of the disk. It follows that when
this map is applied to our flattened mid-edge mesh, its
image is a mid-edge mesh in the unit disk, with the
boundary of the disk corresponding to the boundary of our
(disk-like) surface. (See Figure \ref{f:discrete_type_2}.) We shall
denote by $\Phi:\mV \rightarrow \C$ the composition of these
different conformal and discrete-conformal maps, from the original
mid-edge mesh to the corresponding mid-edge mesh in the unit disk.
\nomenclature{$\Phi$}{The discrete uniformization map taking
vertices of the mid-edge mesh $\mV \too \D$.}

\begin{figure}[ht]
\centering \setlength{\tabcolsep}{0.4cm}
\begin{tabular}{@{\hspace{0.0cm}}c@{\hspace{0.0cm}}c@{\hspace{0.0cm}}c@{\hspace{0.0cm}}}
\includegraphics[width=0.2\columnwidth]{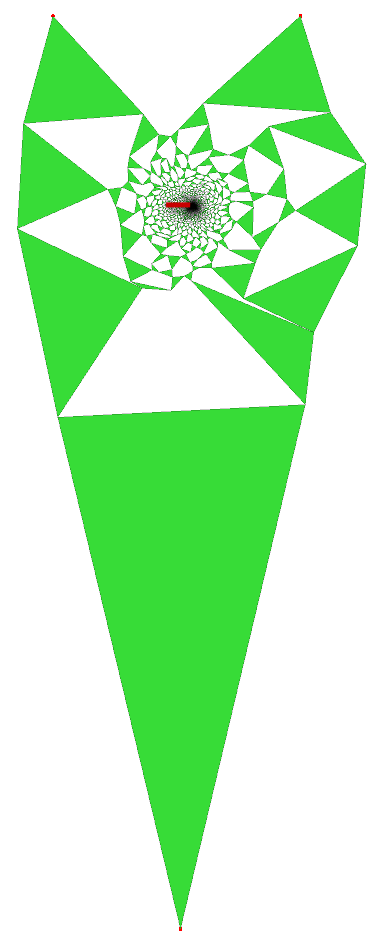} &
\includegraphics[width=0.4\columnwidth]{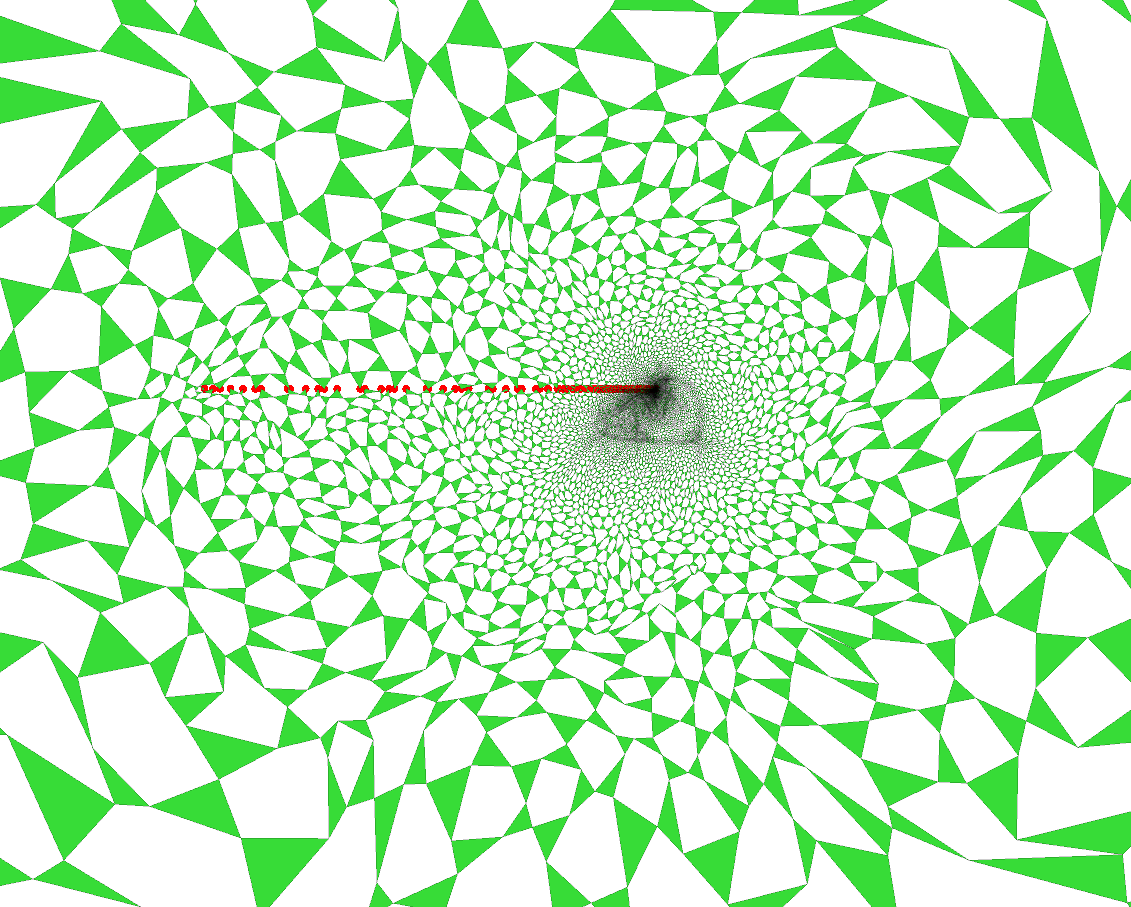}\\
Mid-edge uniformization & Uniformization Zoom-in \\
\includegraphics[width=0.4\columnwidth]{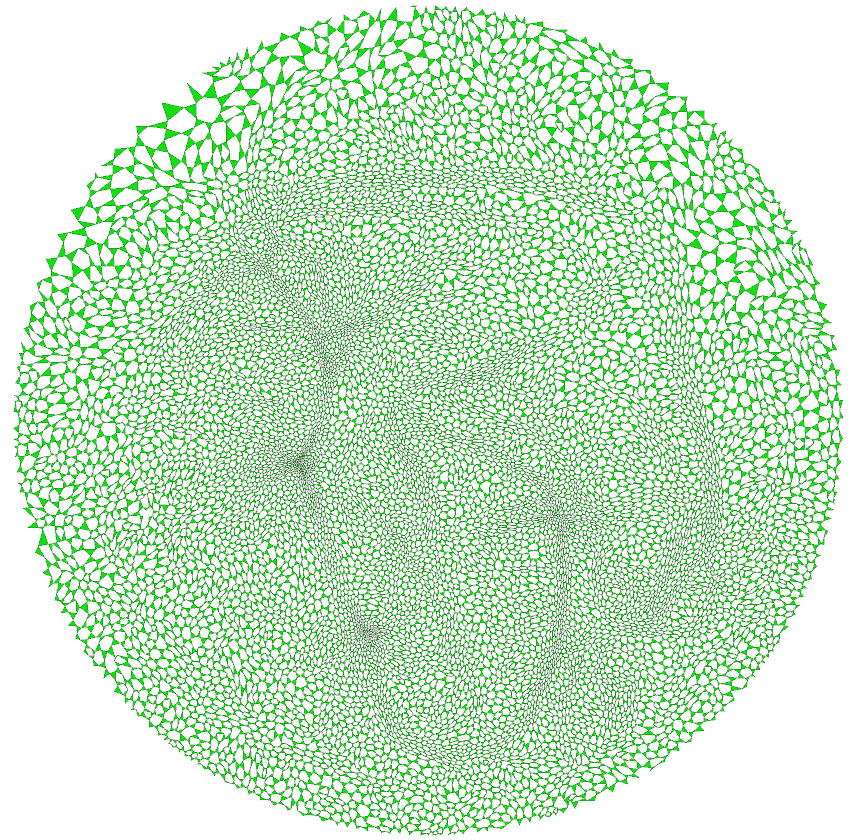}& \includegraphics[width=0.5\columnwidth]{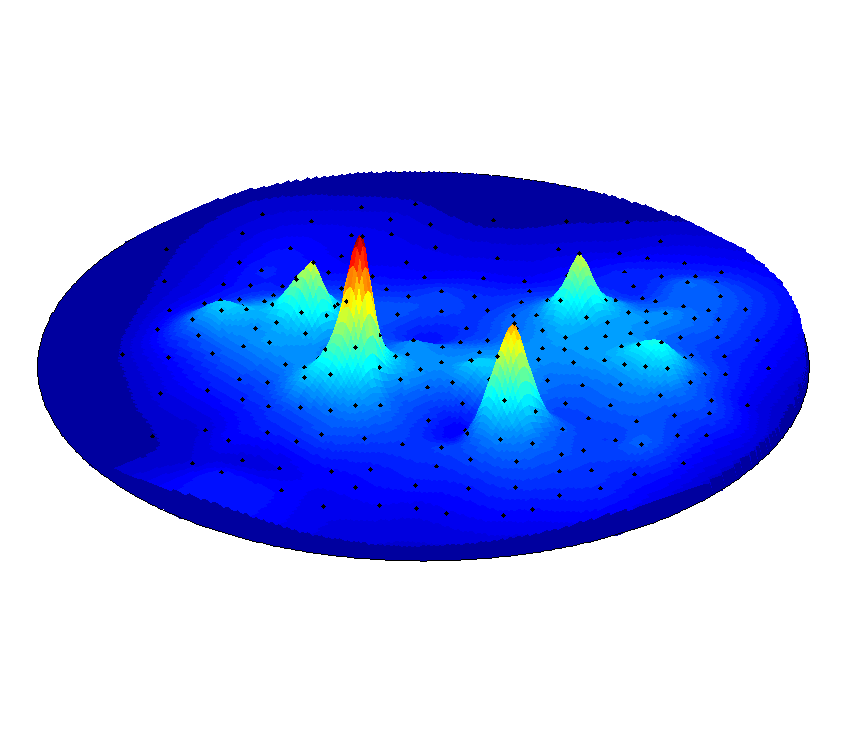}\\
After mapping to the disk & Interpolated conformal factor
\end{tabular}\\
\caption{The discrete conformal transform to the unit disk for the
surface of Figure \ref{f:discrete_type_1}, and the interpolation of
the corresponding discrete conformal factors (plotted with the JET
color map in Matlab). The red points in the top row's images show
the boundary points of the surface.} \label{f:discrete_type_2}
\end{figure}

Next, we define the Euclidean discrete conformal factors, defined as
the density, w.r.t. the Euclidean metric, of the mid-edge triangles
(faces), i.e.
\[
\mu^E_{\mf_{r,s,t}} =
\frac{\Vol_{\mathds{R}^3}(\mf_{r,s,t})}{\Vol(\Phi(\mf_{r,s,t}))}.
\]
Note that according to this definition, we have
\[
\int_{\Phi(\mf_{r,s,t})}\,\mu^E_{\mf_{r,s,t}}\,d\vol_E\, =
\,\frac{\Vol_{\mathds{R}^3}(\mf_{r,s,t})}{\Vol_E\left(\Phi(\mf_{r,s,t})\right)}\,
\Vol_E\left(\Phi(\mf_{r,s,t})\right)\,=\,\Vol_{\mathds{R}^3}(\mf_{r,s,t}),
\]
where $\Vol_E$ denotes the standard Lebesgue (Euclidean) volume
element in $\D$, and $\vol_{\R^3}(\mf)$ stands for the area of $\mf$
as induced by the standard Euclidean volume element in $\R^3$. The
discrete Euclidean conformal factor at a mid-edge vertex $\mv_r$ is
then defined as the average of the conformal factors for the two
faces $\mf_{r,s,t}$ and $\mf_{r,s',t'}$ that touch in $\mv_r$, i.e.
\[
\mu^E_{\mv_r} \,=\,
\frac{1}{2}\,\left(\mu^E_{\mf_{r,s,t}}\,+\,\mu^E_{\mf_{r,s',t'}}\right).
\]
Figure \ref{f:discrete_type_2} illustrates the values of the
Euclidean conformal factor for the mammalian tooth surface of
earlier figures. The discrete hyperbolic conformal factors are
defined according to the following equation, consistent with the
convention adopted in section \ref{s:intro_and_background},
\begin{equation}\label{e:discrete_hyperbolic_conformal_density}
    \mu^H_{\mv_r} \,= \,\mu^E_{\mv_r} \,\left(1 - |\Phi(\mv_r)|^2\right)^2.
\end{equation}
As before, we shall often drop the superscript $H$: unless otherwise
stated, $\mu=\mu^H$, and $\nu=\nu^H$.

The (approximately) conformal mapping of the original mesh to the
disk is completed by constructing a smooth interpolant $\Gamma_\mu
:\D \rightarrow \R$, that interpolates the discrete conformal factor
so far defined only at the vertices in $\Phi(\mV)$; $\Gamma_\nu$ is
constructed in the same way. \nomenclature{$\Gamma_\mu \
(\Gamma_\nu)$}{The smooth interpolant of the discrete conformal
factors.} In practice we use Thin-Plate Splines, i.e. functions of
the type
\begin{equation}\label{e:Gamma_mu_tps}
    \Gamma_\mu(z) = p_1(z)\, +\,\sum_i\, b_i \,\psi(|z-z_i|)\,,
\end{equation}
where $\psi(r)=r^2\log(r^2)$, $p_1(z)$ is a linear polynomial in
$x^1,x^2$, and  $b_i \in \C$; $p_1$ and the $b_i$ are determined by
the data that need to be interpolated. Similarly
$\Gamma_\nu(w)\,=\,q_1(w)\,+\,\sum_j \,c_j\, \psi(|w-w_j|)$ for some
constants $c_j \in \C$ and a linear polynomial $q_1(w)$ in
$y^1,y^2$. We use as interpolation centers two point sets
$Z=\set{z_i}_{i=1}^n,$ and $W=\set{w_j}_{j=1}^p$ that are uniformly
distributed over the surfaces $\M$ and $\N$ (resp.), they are
(relatively small) subsets of the mid-edge mesh vertex sets. In
practice we calculate these (sub) sample sets by starting from an
initial random seed on the surface (which will itself not be
included in the set), and take the geodesic furthest point in $\mV$
(we approximate geodesic distances with Dijkstra's algorithm) from
the seed as the initial point of the sample set. One then keeps
repeating this procedure, selecting at each iteration the point that
lies at the furthest geodesic distance from the set of points
already selected. This algorithm is known as the Farthest Point
Algorithm (FPS) \cite{eldar97farthest}. An example of the output of
this algorithm, using geodesic distances on a disk-type surface, is
shown in Figure \ref{f:discrete_type_3}. Further discussion of
practical aspects of Voronoi sampling of a surface can be found in
\cite{BBK2007non_rigid_book}.

\begin{figure}[h]
\centering \setlength{\tabcolsep}{0.4cm} \hspace*{.5 in}
\begin{minipage}{3 in}
\includegraphics[width=0.8\columnwidth]{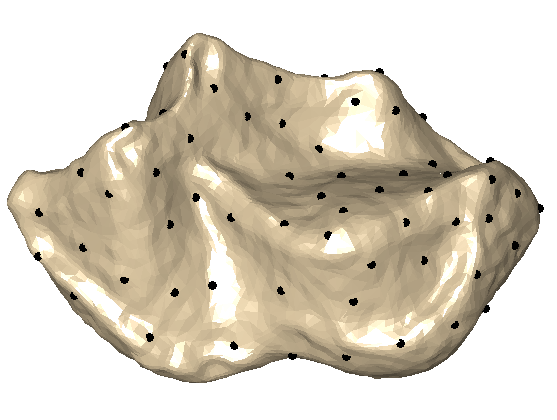}\hspace*{-.5 in}
\end{minipage}
\hspace*{-.5 in}
\begin{minipage}{3 in}
\caption{Sampling of the surface of  Figure \ref{f:discrete_type_1}
obtained by the Farthest Point  Algorithm.}
\label{f:discrete_type_3}
\end{minipage}
\end{figure}

Once the sample sets $Z$ and $W$ are determined we project them to
the uniformization space using $\Phi$. The bottom-right part of
Figure \ref{f:discrete_type_2} shows the result of the interpolation
based on the centers $Z,W$ (shown as black points).

To compute the explicit Thin-Plane-Splines (\ref{e:Gamma_mu_tps}),
we use a standard smoothing Thin-Plate Spline procedure:
$$\Gamma_\mu(z) = \mathop{\mathrm{argmin}}_{\gamma}
\set{ \lambda \sum_{r} \abs{\mu_{\mv_r} - \gamma(\Phi(\mv_r))}^2 +
(1-\lambda) \int_{\D} \parr{\frac{\partial^2 \gamma }{(\partial
x^1)^2}}^2 + \parr{\frac{\partial^2 \gamma }{\partial x^1
\partial x^2}}^2 +
\parr{\frac{\partial^2 \gamma }{(\partial x^2)^2}}^2  dx^1 \wedge
dx^2 },$$ where the minimization is over all $\gamma$ in the
appropriate Sobolev space and where we picked the values $0.95-0.99$
manually (it was fixed per dataset of surfaces) for the smoothing
factor $\lambda$ to avoid over-fitting the data. We noticed that
$\lambda$ does not have large effect on the results.

In our implementation we assumed we have a smooth representation of
the conformal factors $\mu(z)=\Gamma_\mu(z)$, $\nu(w)=\Gamma_\nu(w)$
and we simply use the notation $\mu,\nu$ for these approximations.

%%%%%%%%%%%%%%%%%%%%%%%%%%%%%%%%%%%%%%%%%%%%%%%%%%%%%%%%%%%%%%%%%%
%%%%%%%%%%%%%%%%%%%%%%%%%%%%%%%%%%%%%%%%%%%%%%%%%%%%%%%%%%%%%%%%%%
%%%%%%%%%%%%%%%%%%%%%%%%%%%%%%%%%%%%%%%%%%%%%%%%%%%%%%%%%%%%%%%%%%
%%%%%%%%%%%%%%%%%%%%%%%%%%%%%%%%%%%%%%%%%%%%%%%%%%%%%%%%%%%%%%%%%%

To conclude this whirlwind description of the algorithm and ideas we use
for discrete uniformization, we provide a short exposition on discrete and
conjugate discrete harmonic functions on triangular meshes as
in \cite{Dziuk88,Pinkall93,Polthier00,Polthier05}, and 
show how we use them to conformally
flatten disk-type ( or even just simply connected) triangular
meshes.

Discrete harmonic functions are defined using a variational
principle in the space of continuous piecewise linear functions
defined over the mesh $\PL_\M$ (\cite{Dziuk88}), as follows. Let us
denote by $\phi_i(z),$ $i=1,..,m,$ the scalar functions that satisfy
$\phi_j(v_i)=\delta_{i,j}$ and are affine on each triangle
$f_{i,j,k} \in \F$. Then, the (linear) space of continuous
piecewise-linear function on $M$ can be written in this basis:
$$\PL_M = \set{\sum_{i=1}^m u_i\phi_i(z) \ \mid \ (u_1,...,u_m)^T\in \Real^m }.$$
Next, the following quadratic form is defined over $\PL_M$:
\begin{equation}\label{e:discrete_dirichlet_energy}
    E_{Dir}(u) = \sum_{f\in \F} \int_{f} \ip{\nabla u,\nabla u} d\vol_{\Real^3},
\end{equation}
where $\ip{\cdot}=\ip{\cdot}_{\Real^3}$ denotes the
inner-product induced by the ambient Euclidean space, and
$d\vol_{\Real^3}$ is the induced volume element on $f$. This
quadratic functional, the {\em Dirichlet energy}, can be written
%in coordinates of the basis defined earlier 
as follows:
\begin{equation}\label{e:dirichelet_energy_in_coordinates}
    E_{Dir}\left(\sum_i u_i \phi_i \right) = \sum_{i,j=1}^m u_i u_j \brac{\sum_{f\in\F} \int_{f} \ip{\nabla \phi_i, \nabla \phi_j}}d\vol_{\Real^3} =
    \sum_{i,j=1}^m u_i u_j \int_{M} \ip{\nabla \phi_i, \nabla \phi_j}d\vol_{\Real^3}.
\end{equation}

The discrete harmonic functions are then defined as the functions $u\in PL_M$
that are critical for $E_{Dir}(u)$,
subject to some constraints on the boundary of $M$. The
linear equations for discrete harmonic function $u \in \PL_M$ are
derived by partial derivatives of $E_{Dir}$,
(\ref{e:dirichelet_energy_in_coordinates}) w.r.t. $u_i, i=1,..,m$:
\begin{equation}\label{e:partials_of_dirichlet}
    \frac{\partial E_{Dir}(u)}{\partial u_k} =
2\sum_{i=1}^m u_i \brac{\sum_{f\in\F} \int_{f} \ip{\nabla \phi_i,
\nabla \phi_k}}d\vol_{\Real^3} = 2 \int_M \ip{\nabla u , \nabla
\phi_k}d\vol_{\Real^3} = 2 \int_{R_k} \ip{\nabla u , \nabla
\phi_k}d\vol_{\Real^3},
\end{equation}
where $R_k \subset M$ is the 1-ring neighborhood of vertex $v_k$.
The last equality uses that $\phi_k$ is supported on
$R_k$.

Now, let $u=\sum_i u_i \phi_i$ be a discrete harmonic function.
Pinkall and Polthier observed that conjugating the
piecewise-constant gradient field $\nabla u$ (constant on each
triangle $f \in \F$), i.e. rotating the gradient $\nabla u$ in each
triangle $f$ by $\pi/2$ in the positive ( = counterclockwise) sense (we assume $M$ is
orientable),
results in a new vector field $*du = J du$ with the special property
that its integrals
along (closed) paths that cross edges only at their mid-points are
systematically zero
(see for
example \cite{Polthier05}). This means in particular that we can define a
piecewise linear function $\stu$ such that its gradient satisfies $d\stu
= *du$ and that is furthermore continuous through the mid-edges
$\mv \in \mV$. The space of piecewise-linear functions on meshes
that are continuous through the mid-edges is well-known in the
finite-element literature, where it is called $\ncPL_M$,
the space of non-conforming
finite elements \cite{Brenner:2008:MTF}. The Dirichlet
form (\ref{e:discrete_dirichlet_energy}) is defined over the
space of non-conforming elements $\ncPL_M$ as well; the non-conforming
discrete harmonic functions are defined
to be the functions $v\in\ncPL_M$ that are critical for $E_{Dir}$ and that
satisfy some
constraints on the mid-edges of the boundary of the mesh.
Polthier \cite{Polthier05} shows
that if $u\in \PL_M$ is a discrete harmonic function, then $\stu \in \ncPL_M$ is
also discrete harmonic, with the same Dirichlet energy, and
vice-versa. Solving for the discrete harmonic function after fixing
values at the boundaries amounts to solving a sparse linear system
which is explicitly given in \cite{Polthier05}.

This theory can be used to define discrete conformal mappings, and
used to flatten a mesh in a ``discrete conformal'' manner, as follows. The
flattening is done by constructing a pair of conjugate piecewise
linear functions $(u,\stu)$ where $u\in \PL_M$, $\stu \in \ncPL_M$,
and the flattening map $\Phi:\mM \rightarrow \C$ is given by
\begin{equation}\label{e:Phi_flattening_map}
    \Phi = u+\bfi \stu.
\end{equation}
Since $d\stu = J du$, $\Phi$ is a similarity transformation on each
triangle $f\in \F$. Furthermore, $\Phi$ is continuous through the
mid-edge points $\mv_r\in\mV$; $\Phi$ is thus well-defined on
the points in $\mV$ and maps them to the complex plane.

The function $u$ is defined by choosing an arbitrary triangle
$f_{out} \in F$, excising it from the mesh, setting the
values of $u$ at two of $f_{out}$'s vertices $u_{i_1},u_{i_2}$ to $0$ and $1$,
respectively,  and
then
solving for the discrete harmonic $u$ that satisfies these
constraints. See for example Figure \ref{f:discrete_type_2}
(top-left); the ``missing mid-edge face'' corresponding to the
excised face $f_{out}$ would have connected the three mid-edge vertices
that have a only one mid-edge face touching them. The conjugate
function $\stu$ is constructed by a simple conjugation (and
integration) process as described in \cite{Polthier05} and
\cite{Lipman:2009:MVF}.

A surprising property of the Discrete Uniformization $\Phi$ as it is
defined above, which nicely imitates the continuous theory (see
\cite{Springer57}) is that it takes the boundaries of $\M$ to
horizontal slits, see Figure \ref{f:discrete_type_2}, top row
(boundary vertices colored in red). This property allows us to
easily construct a closed form analytic map (with ``analytic''
in its standard complex analytic sense)
that will further bijectively map the entire complex plane $\C$
minus the slit to the open unit disk, completing our
Uniformization procedure.

The proof of this property is similar to that for
Proposition 35 in \cite{Polthier05}; see also
\cite{Lipman:2009:MVF}. More precisely:
\begin{thm} \label{t:midedge_boundary}
Let $\Phi:\mM \rightarrow \C$ be the flattening
map from the mid-edge mesh $\mM$ of a mesh $\M$ with boundary,
using a discrete harmonic and conjugate harmonic pair as described above.
Then, for each connected component of the boundary of $\M$,  the mid-edge
vertices of boundary edges
are all mapped onto one
line segment parallel to the real axis.
\end{thm}
\begin{proof}
Suppose $u=\sum_i u_i \phi_i(\cdot)$ is a discrete harmonic, piecewise linear and
continuous function, defined at each vertex $v_i \in\V$,  excluding the two
vertices of the excised triangle for which
values are prescribed; then we have, by
(\ref{e:partials_of_dirichlet}),
\begin{equation}\label{e:euler-lagrange_discrete_harmonic}
    \int_{R_i}\langle \nabla \phi_i, \nabla u\rangle d\vol_{\R^3} = 0,
\end{equation}
Next, consider a boundary vertex $v_j$ of the mesh $\M$. Denote by
$\mv_r,\mv_s$ the mid-edge vertices on the two boundary edges
touching vertex $v_j$. We will show that $\stu(\mv_r)=\stu(\mv_s)$;
this will imply the theorem, since $\stu$ gives the
imaginary coordinate for the images  of
the mid-edge vertices under the flattening map (see (\ref{e:Phi_flattening_map})) .
Observe that on the triangle $f_{i,j,k}$,
\begin{equation}\label{e:grad_phi_j}
    \nabla \phi_j = \frac{J(v_i-v_k)}{2\,\vol_{\Real^3}(f_{i,j,k})}.
\end{equation}
Recalling that $\nabla \stu= J \nabla u$, using (\ref{e:grad_phi_j}),
and $J^T=-J$, we obtain
\begin{align*}
\stu(\mv_r)-*u(\mv_s) & =\int_\gamma d \stu = \int_\gamma *du = \sum_{f_{i,j,k} \ni v_j} \ip{ J \nabla  u\mid_{f_{i,j,k}} , \frac{1}{2}(v_i-v_k)} \\
& = \sum_{f_{i,j,k}\ni v_j} \ip{  \nabla  u\mid_{f_{i,j,k}} , \frac{1}{2}J^T (v_i-v_k) } 
 =  \sum_{f_{i,j,k}\ni v_j} \ip{  \nabla  u\mid_{f_{i,j,k}} , - \nabla \phi_j \mid_f } \vol_{\Real^3}(f) \\
& = - \int_M \ip{ \nabla u , \nabla \phi_j } d\vol_{\Real^3} 
 = 0,
\end{align*}
where $\gamma$ is the piecewise linear path starting at $\mv_r$ and
passing through the mid-edge vertices of the 1-ring neighborhood of
$v_j$ ending at $\mv_s$. The last equality is due to
(\ref{e:partials_of_dirichlet}).
\end{proof}

A natural question, when dealing with any type of finite-element
approximation, concerns convergence as the mesh is refined:
convergence in what sense, and at what rate? For discrete harmonic
functions over meshes, this convergence is discussed in
\cite{Hildebrandt06,Polthier00}. These convergence results
are in the weak sense; this motivated our defining the discrete
conformal factors $\mu_{\mf}$ via integrated quantities (volumes) in
Section \ref{s:the_discrete_case_implementation}.

Finally, we note that the method presented here for Discrete
Uniformization is just one option among several; other
authors have suggested other techniques; for example \cite{Gu03}.
Typically, this
part of the complete algorithm described in this paper
could be viewed as a ``black box'': the remainder of the algorithm
would not change if one method of  Discrete
Uniformization is replaced by another.

\renewcommand{\thesection}{ \Alph{section}}
\section{}
\label{a:appendix C}
\renewcommand{\thesection}{\Alph{section}}
%\section*{Appendix C}
In this Appendix we prove a lemma used in the proof of Theorem
\ref{t:relaxation}.

{\bf Lemma} {\em The $N \times N$ matrices $\pi$ satisfying
\begin{equation}
     \left \{ \begin{array}{c}
             \sum_i \pi_{ij} \leq 1 \\
             \sum_j \pi_{ij} \leq 1 \\
             \pi_{ij} \geq 0 \\
             \sum_{i,j} \pi_{ij} = M < N
           \end{array}
   \right .
\label{appc:constraints}
\end{equation}
constitute a convex polytope $\P$ of which the  extremal points are
exactly those $\pi$ that satisfy all these constraints, and that
have all entries equal to either 0 or 1. }

{\em Remark}. Note that the matrices $\pi \in \P$ with all entries
in $\{0,1\}$ have exactly $M$ entries equal to 1, and all other
entries equal to zero; if one removes from these matrices all rows
and columns that consist of only zeros, what remains is a $M \times
M$ permutation matrix.

\begin{proof}
$\P$ can be considered as a subset of $\R^{N^2}$, with all entries
nonnegative, summing to $M$. The two inequalities in
(\ref{appc:constraints}) imply that the entries of any $\pi \in \P$
are bounded by 1. These inequalities can also be rewritten as the
constraint that every entry of $A \P - b \in \R^{2N}$ is non
positive, where $A$ is a $\R^{2N}\times \R^{N^2}$ matrix, and $b$ is
a  vector in $\R^{2N}$. It follows that $\P$ is a (bounded) convex
polytope in $\R^{N^2}$.

If $\pi \in \P\subset \R^{N^2}$ has entries equal to only 0 or 1,
then $\pi$ must be an extremal point of $\P$ by the following
argument. If $\pi_{\ell}=1$, and $\pi$ is a nontrivial convex
combination of $\pi^1$ and $\pi^2$ in $\P$, then
\[
\pi=\lambda\,\pi^1\,+\,(1-\lambda)\,\pi^2 \,\mbox{ with }\lambda \in
(0,1)\, \Longrightarrow 1=
\lambda\,\pi^1_{\ell}\,+\,(1-\lambda)\,\pi^2_{\ell} \,\mbox{ with }
\pi^1_{\ell}\,,\,\pi^2_{\ell}\geq 0 \Longrightarrow
\pi^1_{\ell}=\pi^2_{\ell}=1~.
\]
A similar argument can be applied for the entries of $\pi$ that are
0. It follows that we must have $\pi^1=\pi=\pi^2$, proving that
$\pi$ is extremal.

It remains thus to prove only that $\P$ has no other extremal
points. To achieve this, it suffices to prove that the extremal
points of $\P$ are all integer vectors, i.e. vectors all entries of
which are integers -- once this is established, the Lemma is proved,
since the only integer vectors in $\P$ are those with all entries in
$\{0,1\}$.

To prove that the extremal points of $\P$ are all integer vectors,
we invoke the Hoffman-Kruskal theorem (see \cite{Lovasz86}, Theorem
7C.1), which states that, given a $L\times K$ matrix $\MM$, with all
entries in $\{-1,0,1\}$, and a vector $b \in \R^L$ with integer
entries, the vertices of the polytope defined by $\{ x \in \R^K \,;
\, (\MM x)_\ell \le b_\ell \,\mbox{ for }\, \ell=1,\ldots,L\}$ are
all integer vectors in  $\R^K$ if and only if the matrix $\MM$ is
totally unimodular, i.e. if and only if  every square submatrix of
$\MM$ has determinant 1, 0 or $-1$.

We first note that (\ref{appc:constraints}) can indeed be written in
this special form. The equality $\sum_{i,j} \pi_{ij} = M$ can be
recast as the two inequalities $\sum_{i,j} \pi_{ij} \le M$ and
$-\,\sum_{i,j} \pi_{ij} \le - M$. The full system
(\ref{appc:constraints}) can then be written as $(\MM \pi)_\ell \le
b_\ell $ for $\ell=1,\ldots,L$, where $\MM$ is a $(2N+2+N^2)\times
N^2$ matrix constructed as follows. Its first $2N$ rows correspond
to the constraints on the sums over rows and columns; the entries of
the next row are all 1, and of the row after that, all $-1$ -- these
two rows correspond to the constraint $\sum_{i,j} \pi_{ij} = M$; the
final $N^2 \times N^2$ block is diagonal, with all its diagonal
entries equal to $-1$. The first $2N$ entries of $b$ are 1; the next
2 entries are $M$ and $-M$; its final $N^2$ entries are 0. By the
Hoffman-Kruskal theorem it suffices thus to show that $\MM$ is
totally unimodular.

Because the last $N^2$ rows,  the {\em bottom rows} of $\MM$, have
only one non-zero entry, which equals $-1$, we can disregard them.
Indeed, if we take a square submatrix of $\MM$ that includes (part
of) one of these bottom rows, then the determinant of the submatrix
is 0 if only zero entries of the bottom row ended up in the
submatrix; if the one -1 entry of the bottom row is an entry in the
submatrix, then the determinant is, possibly up to a sign change,
the same as if that row and the column of the $-1$ entry are
removed. By this argument, we can remove all the rows of the
submatrix partaking of the bottom rows of $\MM$.

We thus have to check unimodularity only for $\MM'$, the submatrix
of $\MM$ given by its first $2N+2$ rows. If any submatrix contains
(parts of) both the $(2N+1)$st and the $(2N+2)$nd row, then the
determinant is automatically zero, since the second of these two
rows equals the first one, multiplied by -1. This reduces the
problem to checking that $\MM''$, the submatrix of $\MM$ given by
its first $2N+1$ rows, is totally unimodular.

We now examine the top $2N$ rows of $\MM''$ more closely. A little
scrutiny reveals that it is, in fact,  the adjacency matrix $\GG$ of
the complete bipartite graph with $N$ vertices in each
part.\footnote{ The adjacency matrix $A$ for a graph $\mathcal{G}$
has as many columns as $\mathcal{G}$ has edges, and as many rows as
$\mathcal{G}$ has vertices; if we label the rows and columns of $A$
accordingly, then $A_{ve}=1$ if the vertex $v$ is an end point of
the edge $e$; otherwise $A_{ve}=0$. An adjacency matrix thus has
exactly two nonzero entries (both equal to 1) in each column. The
number of nonzero entries in the row with index $v$ is the degree of
$v$ in the graph. } It is well-known (see e.g. Theorem 8.3 in
\cite{Schrijver08}) that this adjacency matrix is  totally
unimodular, so any square submatrix of $\MM''$ that does not involve
the $(2N+1)$st row of $\MM''$ is already known to have determinant
0, 1 or $-1$. We thus have to check only submatrices that involve
the last row, i.e. matrices that consist of a $(n-1)\times n$
submatrix of $\GG$, with an added $n$th row with all entries equal
to 1. We'll denote such submatrices by $\GG'$.

We can then use a simple induction argument on $n$ to finish the proof. 
The case $n=2$ is trivial.
In proving the induction step for $n=m$,
we can assume that each of the top $m-1$ rows
of our $m \times m$ submatrix $\GG'$ 
contains at least two entries 
equal to 1, since
otherwise the determinant of $\GG'$ would automatically 
be 0, 1 or -1 by induction.

The 
first $m-1$ rows of $\GG'$ correspond to vertices in the
bipartite graph, and can thus be partitioned into two sets $S_1$ and $S_2$, 
based on which of the
two parts
of $N$ vertices in the graph they pertain to. Let us
call $S$ the larger of $S_1$ and $S_2$; $S$ consists of at least $\lceil
\frac{m-1}{2} \rceil$ rows.  Let us examine the $(\# S) \times m$
sub-matrix $\GG''$ constructed from exactly these rows.  We know that
each column of $\GG''$ has exactly one entry  $1$, since all the
rows of $\GG''$ correspond to the same group of vertices in the
bipartite graph. Therefore, summing all the rows of $\GG''$ gives a
vector $v$ of only zeros and ones; since each row in $\GG''$
contains at least two entries equal to 1, the sum of all entries in $v$
is at least $2\parr{\lceil
\frac{m-1}{2} \rceil} \geq m-1$. The vector 
$v$ has thus at least $m-1$ entries equal to $1$; the remaining $m$th
entry of this linear combination of the top $m-1$ rows
of $\GG'$ is either 1 or 0. In the first case, the determinant
of $\GG'$ vanishes, since its last row also consists of only ones.
In the second case, we can subtract $v$ from the last row of 
$\GG'$ without changing the value of the determinant; the resulting last 
row has all entries but one equal to 0, with a remaining entry equal
to 1. The determinant is then given by the minor
of this remaining entry, and is thus 0, 1 or -1 by the unimodularity
of $\GG$.
\end{proof}

\renewcommand{\thesection}{ \Alph{section}}
\section{}
\label{a:approximating_optimal_transport}
\renewcommand{\thesection}{\Alph{section}}
%\section*{Appendix D}
%% MY SYMBOLS-------------------------------------------------------
%\newcommand{\norm}[1]{\left\Vert#1\right\Vert}
%\newcommand{\abs}[1]{\left\vert#1\right\vert}
%\newcommand{\set}[1]{\left\{#1\right\}}
%\newcommand{\parr}[1]{\left (#1\right )}
%\newcommand{\brac}[1]{\left [#1\right ]}
%\newcommand{\Real}{\mathds{R}}
%\newcommand{\eps}{\varepsilon}
%\newcommand{\To}{\longrightarrow}
%\newcommand{\too}{\rightarrow}
%\newcommand{\BX}{\mathbf{B}(X)}
%\newcommand{\A}{\mathcal{A}}
%\newcommand{\B}{\mathcal{B}}
%\newcommand{\bbar}[1]{\overline{#1}}
%\def \X{\mathcal{X}}
%\def \Y{\mathcal{Y}}
%\def \T{\mathcal{T}}
%\def \N{\mathcal{N}} %ingrid's nibble-off operator
%\def \G{\mathcal{G}} %graph, for flow network graph
%\def \E{\mathcal{E}} %edges, for flow network graph
%\def \V{\mathcal{V}} %vertices, for flow network graph
%\def \s{\textsf{s}}   %source, for flow
%\def \t{\textsf{t}}   %target, for flow

In this Appendix we provide a constructive procedure and convergence
analysis for approximating the optimal transport cost between
\emph{general} separable complete compact metric spaces
$(\X,d_\X),(\Y,d_\Y)$ each equipped with a probability measure $\mu
\in P(\X),\nu \in P(\Y)$, where $P(\X)$ ($P(\Y)$) denotes the set of
probability measures on $\X$ ($\Y$). In the context of the algorithm
previously described $\X,\Y$ are the two given surfaces, $d_\X,d_\Y$
the corresponding geodesic distance metric functions, and $\mu,\nu$
the area measures of the surfaces induced from the metric tensors,
respectively. Since $R,\mu,\nu$ are kept fixed through this
discussion, we will denote, for brevity, $c(x,y) =
d^R_{\mu,\nu}(x,y)$.

The Kantorovich optimal transportation cost of the measures
$\mu,\nu$ is defined as
\begin{equation}\label{e:Kantorovich_optimal_transport_functional}
    \T_c(\mu,\nu) =
    \mathop{\inf}_{\pi \in \Pi(\mu,\nu)}\int_{X \times Y}c(x,y)d\pi(x,y),
\end{equation}
where $\Pi(\mu,\nu) \subset P(\X\times \Y)$ is the set of
probability measures on $\X \times \Y$ with marginals $\mu,\nu$,
that is, $\pi\in \Pi(\mu,\nu) \Rightarrow\pi(A \times \Y) = \mu(A)$
and $\pi(\X\times A')=\nu(A')$, for all Borel $A\subset \X,
A'\subset\Y$.

The main goal of this section is to present an \emph{approximation
result} for $\T_c(\mu,\nu)$ in this general framework. In
particular, this result will assure the convergence of our
algorithm.

To our knowledge, the only related result talks merely about
convergence of the optimal cost (e.g., \cite{Villani:2003}, Theorem
5.20). However, for practical applications it is important to
control the \emph{rate} of convergence, and therefore to be able to
compute error-bounded approximations.

In the specific case of $c(x,y)=\norm{x-y}^2$ there are good
approximation techniques that rely on the polar decomposition of
Brenier, for example the work of Haker and collaborators
\cite{Haker04optimalmass}. However, as far as we are aware no
approximation result is known in the general metric case as required
here.

We will show that solving discrete mass-transportation between two
sets of discrete measures $\mu_S,\nu_T$, based on Voronoi diagrams
of two collections of points $S=\set{s_i} \subset \X, T=\set{t_j}
\subset \Y$, achieve linear approximation order to the continuous
limit mass-transport cost $\T_c(\mu,\nu)$:
$$\Big | \T_c(\mu,\nu) - \T_c(\mu_S, \nu_T)  \Big | \leq \omega_c \parr{2h},$$
where $\omega_c(\alpha)$ is the modulus of continuity of $c$ defined
by
$$\omega_c(\alpha) = \sup_{d_\X(x,x')  + d_\Y(y,y') <
\alpha}\abs{c(x,y)-c(x',y')},$$ and $h=\max\{\eta(S),\eta(T)\}$,
where the \emph{fill distances} $\varphi_\X(S),\varphi_\Y(T)$ are
defined as before by
\begin{equation}\label{e:fill_distance}
    \varphi_\X(S)=\sup \Big\{ r\in\Real \ \Big | \
\exists x\in\X \ s.t. \ B_\X(x,r)\cap S = \emptyset    \Big \},
\end{equation}
where $B_\X(x,r) = \{q\in\X \ | \ d_\X(x,q)<r\}$ (and similarly for
$\varphi_\Y(T)$).

In particular, for Lipschitz cost function $c$ with Lipschitz
constant $\lambda$, we have the following bound for the error in the
approximation:
\begin{equation}\label{e:error_bound_in_optimal_approx}
    \Big | \T_c(\mu,\nu) - \T_c(\mu_S, \nu_T)  \Big | \leq 2\lambda h.
\end{equation}

In turn, this result suggests an algorithm for approximating
$\T_c(\mu,\nu)$: simply spread points $S\subset \X$ and $T \subset
\Y$ such that no big empty space is left uncovered, then compute
$\T_c(\mu_S,\nu_T)$ using linear-programming solver.

\subsection{Voronoi cells and discrete measures} Let
$(\X,d_\X,\mu),(\Y,d_\Y,\nu)$ be two compact, complete, separable
metric spaces with probability measures defined over the Borel sets.

The discrete measures will based on discrete sets of points
$S=\{s_i\}_{i=1}^m \subset \X$,$T=\{t_i\}_{i=1}^n \subset \Y$ and
the coarseness of the sets will be measured by means of the
so-called \emph{fill-distance} $\varphi_\X(S),\varphi_\Y(T)$ of the
sets $S,T$. The fill-distance of $S \subset \X$ is defined in
eq.~(\ref{e:fill_distance}). The sets get ``finer'' as
$\max\set{\varphi_\X(S),\varphi_\Y(T)}\too 0$.

\begin{defn}\label{def:voronoi_cells}
For set of points $S=\{s_i\}_{i=1}^m\subset \X$ we define
\begin{equation}\label{e:voronoi_cells_Oi}
    O^S_i = \Big\{ x\in \X \ \Big |
    \ d_\X(x,s_i) < \min_{j \ne i} d_\X(x,s_j)\Big\}.
\end{equation}
\begin{equation}\label{e:voronoi_cells_Ci}
    C^S_i = \Big\{ x\in \X \ \Big |
    \ d_\X(x,s_i) \leq \min_{j \ne i} d_\X(x,s_j)\Big\}.
\end{equation}
Then, {Voronoi cells} for the point set $S$ is any collection of
sets $\{V_i\}_{i=1}^m$ satisfying
\begin{enumerate}
  \item $\cup_{i=1}^m V_i = \X$.
  \item $V_i \cap V_j = \emptyset$ for $i \ne j$.
  \item $O^S_i \subset V_i \subset C^S_i$, for all $i=1,..,m$.
\end{enumerate}
\end{defn}

We will prove now a simple lemma, for later use, connecting the
fill-distance with the geometry of the Voronoi cells.
\begin{lem}\label{lem:voronoi_in_balls}
Let $S = \{s_i\}_{i=1}^m \subset \X$. If $\{V_i\}_{i=1}^m$ is a
collection of Voronoi cells corresponding to $S$, then for all $\eps
>0$,
$$V_i \subset B(s_i,\varphi_\X(S)+\eps), \ \ i=1,..,m$$
\end{lem}
\begin{proof}
Take $x \notin B(s_i,\varphi_\Y(S)+\eps)$.

By the definition of the fill-distance we have that
$$B\parr{x, \varphi_\X(S)+\frac{\eps}{2}}\cap S \ne \emptyset.$$
That is, there exists $s_k \in S$, $k\ne i$, such that
$$d_\X(x,s_k)<\varphi_\X(S)+\frac{\eps}{2}.$$ However, $$d_\X(x,s_i)\geq
\varphi_\X(S)+\eps > d_\X(x,s_k).$$ Hence from the definition of the
Voronoi cells
$$x \in \X \setminus C^S_i \subset \X \setminus V_i,$$
that is $x \notin V_i$.
\end{proof}

Given a set of points and a Voronoi cell collection
$S,\{V_i\}_{i=1}^m$ we define a discrete measure $\mu_S$ by
\begin{equation}\label{e:mu_S}
    \mu_S = \sum_{i=1}^m \mu(V_i)\delta_{s_i},
\end{equation}
where $\delta_{s_i}$ is the dirac measure centered at $s_i$. That is
$$\int_\X f \ d\mu_S = \sum_{i=1}^m \mu(V_i)f(s_i).$$
Similarly we define Voronoi cell collection $\{W_j\}_{j=1}^n$ for
point set $T \subset \Y$, and corresponding discrete measure
$\nu_T$.

Let us prove that the discrete measures $\mu_S,\nu_T$ converge in
the weak sense to $\mu,\nu$. By weak convergence $\mu_S \too \mu$ of
measure we mean (see for example \cite{Billingsley68}) that for
every continuous bounded function $f:\X \too \Real$ there exists
$$\int_\X f\ d\mu_S \To \int_\X f\ d\mu\ ,  \ \ \mathrm{as }\
\varphi_\X(S)\too 0.$$

\begin{thm}
$\lim_{\varphi_\X(S)\too 0}\int_\X f\ d\mu_S = \int_\X f\ d\mu$, for
all $f:\X\too \Real$ bounded and continuous.
\end{thm}
\begin{proof}
$X$ is compact therefore $f$ is uniformly continuous. Take arbitrary
$\eps>0$, let $\delta(\eps)>0$ be such that
$$x,x'\in\X \ , \ d_\X(x,x') < \delta(\eps) \ \Rightarrow \ \abs{f(x)-f(x')}<\eps.$$
For $S$ with $\varphi_\X(S)<\delta(\eps)$ we have
$$\abs{\int_\X f\ d\mu - \int_\X f\ d\mu_S } = \abs{\sum_{i=1}^m \int_{V_i} \parr{f(x)-f(s_i)} \ d\mu} \leq $$
$$\leq \sum_{i=1}^m \int_{V_i}\abs{f(x)-f(s_i)}\ d\mu \leq \sum_{i=1}^m \eps \mu(V_i) = \eps,$$
where in the second to last transition we used the fact that $V_i
\subset B(s_i,\delta(\eps))$ which we know is the case from Lemma
\ref{lem:voronoi_in_balls}.
\end{proof}

Lastly, we will denote by $\Pi(\mu_S,\nu_T)\subset P(S\times T)$ the
subset of probability measures on the discrete product space
$S\times T$ with marginals $\mu_S,\nu_T$.

\subsection{Approximation of optimal cost} In this subsection we prove
the main result of this appendix. Namely, that the optimal transport
cost $\T_c(\mu_S,\nu_T)$ of the discrete measures $\mu_S,\nu_T$ is
an $\eps-$approximation to the optimal transport cost
$\T_c(\mu,\nu)$ of $\mu,\nu$ if the fill-distance
$h=\max\{\varphi_\X(S),\varphi_\Y(T)\} < \frac{1}{2}\delta(\eps)$,
where $\delta(\eps)$ is the uniform continuity constant of $c$.

Actually, we will prove a slightly stronger result: denote the sets
\begin{equation}\label{e:def_of_A_and_B}
    \A = \Big\{\int_{\X\times\Y}c\ d\pi \ \Big | \ \pi\in\Pi(\mu,\nu)  \Big\} \ \ , \ \ \B=\Big\{\int_{\X\times\Y}c\ d\pi \ \Big | \ \pi \in\Pi(\mu_S,\nu_T)  \Big\}.
\end{equation}
We will show that the Hausdorff distance $d_H(\A,\B) \rightarrow 0$
as $h\rightarrow 0$. Where by Hausdorff distance of two sets
$\A,\B\subset \mathds{R}$ we mean $$d_H(\A,\B) = \inf \Big\{r \Big |
\ \B \subset U(\A,r)\ , \ \A \subset U(\B,r) \Big\},$$ where
$U(\A,r)=\cup_{a\in \A}B(a,r)$ and $B(a,r)$ is the open ball of
radius $r$ centered at $a$. Moreover, we will provide a linear (in
$h$) bound controlling the convergence rate, for Lipschitz cost
function $c$.

\begin{thm}\label{thm:main_theorem}
If $c:\X\times\Y \To \Real_+$ is a continuous function, $\X,\Y$
compact complete separable metric spaces. Let
$h=\max\set{\varphi_\X(S),\varphi_\Y(T)}$ then,
$$d_H(\A,\B) \leq \omega_c(2h),$$ where $\A,\B$ are defined in (\ref{e:def_of_A_and_B}). In particular,
$$\Big | \T_c(\mu,\nu) - \T_c(\mu_S, \nu_T)  \Big | \leq \omega_c (2h).$$
\end{thm}
\begin{proof}

Take arbitrary $S=\{s_i\}_{i=1}^m\subset \X$,
$T=\{t_i\}_{j=1}^n\subset\Y$. Choose collections of Voronoi cells
$\{V_i\}_{i=1}^m$, $\{W_j\}_{j=1}^n$ for $S,T$ respectively.

Now take $\pi\in\Pi(\mu,\nu)$.

Set $$\Lambda=\sum_{i=1}^m \sum_{j=1}^n \pi(V_i \times W_j)
\delta_{s_i,t_j}.$$ Then,
$$\Lambda(A\times \Y) = \sum_{i=1}^m \sum_{j=1}^n \pi(V_i \times W_j) \delta_{s_i}(A) = $$ $$\sum_{i=1}^m \pi(V_i \times \Y) \delta_{s_i}(A) = \sum_{i=1}^m \mu(V_i) \delta_{s_i}(A) = \mu_S(A).$$

Similarly $$\Lambda(\X \times A') = \nu_T(A').$$ Therefore $$\Lambda
\in \Pi(\mu_S,\nu_T).$$

Moreover,
$$\Big|\int_{\X \times \Y} c\ d\pi - \int_{\X \times \Y}c\ d\Lambda\Big|=
\Big|\sum_{i=1}^m\sum_{j=1}^n \int_{V_i \times W_j}
\big[c(x,y)-c(s_i,t_j)\big] \ d\pi \Big|\leq$$
$$\sum_{i=1}^m\sum_{j=1}^n \int_{V_i \times W_j} \big|c(x,y)-c(s_i,t_j)\big | \ d\pi.$$
For $(x,y)\in V_i \times W_j$, we have $d_\X(x,s_i) \leq
\varphi_\X(S)$, $d_\Y(y,t_j) \leq \varphi_\Y(T)$ which implies
\begin{equation}\label{e:main_thm_using_uniform_continuity}
\abs{c(x,y)-c(s_i,t_j)} \leq
\omega_c\parr{\varphi_\X(S)+\varphi_\Y(T)} \leq \omega_c(2h).
\end{equation}
So we have
$$\abs{\int_{\X \times \Y}c\ d\pi - \int_{\X \times \Y}c\ d\Lambda} \leq \omega_c(2h) \sum_{i=1}^n \sum_{j=1}^n \pi(V_i \times W_j) = \omega_c(2h),$$
since $\set{V_i\times W_j}_{i,j=1,1}^{n,m}$ form a partition of
$X\times Y$. So we proved that
$$\A \subset \bbar{U(\B,\omega_c(2h))}.$$

We now prove the other direction. Take $\Lambda \in
\Pi(\mu_S,\nu_T)$.

Denote by $\mu\times\nu \big |_{V_i\times W_j}$ the product measure
$\mu\times\nu$ restricted to $V_i\times W_j \subset \X\times \Y$.

That is, $$\mu\times\nu \big |_{V_i\times W_j}(A\times A') =
\mu\times\nu \left((V_i\times W_j)\cap(A\times A') \right) =
\mu(V_i\cap A)\nu(W_j\cap A').$$

Now pick
$$\pi = \sum_{i=1}^m\sum_{j=1}^n \frac{\Lambda(V_i\times W_j)}{\mu(V_i)\nu(W_j)} \mu\times\nu\big | _{V_i\times W_j}.$$

Then,
$$\pi(A\times \Y) =
\sum_{i=1}^m\sum_{j=1}^n \frac{\Lambda(V_i\times
W_j)}{\mu(V_i)\nu(W_j)}\mu(V_i\cap A)\nu(W_j\cap \Y)=$$
$$\sum_{i=1}^m\sum_{j=1}^n \frac{\Lambda(V_i\times W_j)}{\mu(V_i)}\mu\parr{V_i\cap A}=$$
$$\sum_{i=1}^m \frac{\mu(V_i\cap A)}{\mu(V_i)}\brac{\sum_{j=1}^n \Lambda\parr{V_i\times W_j}}=$$
$$\sum_{i=1}^m \frac{\mu(V_i\cap A)}{\mu(V_i)}\Lambda\parr{V_i\times \Y}=$$
$$\sum_{i=1}^m \frac{\mu(V_i\cap A)}{\mu(V_i)}\mu(V_i)=$$
$$\sum_{i=1}^m \mu(V_i\cap A)=\mu(\X\cap A)=\mu(A).$$
Similarly
$$\pi(\X \times A') = \nu(A').$$
Therefore,
$$\pi \in \Pi(\mu,\nu).$$
Now,
$$\abs{\int_{\X \times \Y} c \ d\pi - \int_{\X \times \Y} c \ d\Lambda}=$$
$$\abs{\sum_{i=1}^m\sum_{j=1}^n \frac{\Lambda\parr{V_i\times W_j}}{\mu(V_i)\nu(W_j)}\int_{V_i\times W_j}c\ d \parr{\mu\times\nu} -
\sum_{i=1}^m\sum_{j=1}^n c(s_i,t_j) \Lambda(V_i\times W_j)}=$$
$$\abs{\sum_{i=1}^m\sum_{j=1}^n \brac{\frac{\Lambda\parr{V_i\times W_j}}{\mu(V_i)\nu(W_j)}\int_{V_i\times W_j}c\ d \parr{\mu\times\nu} -
\frac{\Lambda\parr{V_i\times W_j}}{\mu(V_i)\nu(W_j)}\int_{V_i\times
W_j} c(s_i,t_j) \ d\parr{\mu\times\nu} } }\leq$$
$$\sum_{i=1}^m\sum_{j=1}^n \frac{\Lambda\parr{V_i\times W_j}}{\mu(V_i)\nu(W_j)}\int_{V_i\times W_j}\abs{c(x,y) - c(s_i,t_j)}\ d \parr{\mu\times\nu} .$$
As before, for $(x,y)\in V_i \times W_j$, we have $d_\X(x,s_i) \leq
\varphi_\Y(S)$, $d_\Y(y,t_j) \leq \varphi_\Y(T)$ therefore
$$\abs{c(x,y)-c(s_i,t_j)} \leq \omega_c\parr{\varphi_\X(S) + \varphi_\Y(T)} \leq \omega_c(2h).$$
And therefore
$$\abs{\int_{\X \times \Y} c \ d\pi - \int_{\X \times \Y} c \ d\Lambda} \leq \omega_c(2h)\sum_{i=1}^m\sum_{j=1}^n \frac{\Lambda\parr{V_i\times W_j}}{\mu(V_i)\nu(W_j)}\mu(V_i)\nu(W_j)= $$
$$\omega_c(2h)\sum_{i=1}^m\sum_{j=1}^n \Lambda\parr{V_i\times W_j} = \omega_c(2h).$$ So we proved
$$\B \subset \bbar{U(\A,\omega_c(2h))},$$
and
$$ d_H(\A,\B) \leq \omega_c(2h).$$
In particular this means that
$$\Big | \T_c(\mu,\nu) - \T_c(\mu_S, \nu_T)  \Big | = \Big | \inf(\A) - \inf(\B) \Big | \leq  \omega_c(2h). $$

\end{proof}

We will call $c:\X\times\Y \To \Real_+$ Lipschitz continuous with a
constant $\lambda$ if $$\abs{c(x,y)-c(x',y')}\leq \lambda
\parr{d_\X(x,x')+d_\Y(y,y')}, \ \forall x,x'\in\X \ , \ y,y'\in\Y.$$
 {\bf Theorem \ref{t:convergence_optimal_cost}}
{\em Suppose $c:\X\times\Y \To \Real_+$ is a continuous function,
with $\X,\Y$ compact complete separable metric spaces, $S$ and $T$
are sample sets in $\X,\Y$ (resp.), $\mu,\nu$ are probability
measures on $\X,\Y$.
\\ (A) if $c$ is uniformly
continuous then $$\T_c(\mu_S, \nu_T) \too \T_c(\mu,\nu), \ \ \ as \
h \too 0,$$\\ (B) if $c$ is Lipschitz continuous with a constant
$\lambda$, then
$$\abs{\T_c(\mu,\nu) - \T_c(\mu_S,\nu_T) } < 2\lambda h,$$
where, $h=\max \set{\varphi_\X(S),\varphi_\Y(T)}$, and $\mu_S,\nu_T$
are as defined in (\ref{e:mu_S}).}
%\begin{thm}\label{t:convergence_optimal_cost}
%if $c$ is uniformly continuous then $$\T_c(\mu_S, \nu_T) \too
%\T_c(\mu,\nu),$$ as $h \too 0$, and if $c$ is Lipschitz continuous
%with a constant $\lambda$, then
%$$\abs{\T_c(\mu,\nu) - \T_c(\mu_S,\nu_T) } < 2\lambda h,$$
%where, as before, $h=\max \set{\varphi_\X(S),\varphi_\Y(T)}$.
%\end{thm}
\begin{proof}
The result follow from the fact that for uniformly continuous $c$,
$\omega_c(t) \too 0$ as $t \too 0$, and that for Lipschitz $c$ with
constant $\lambda$, $\omega_c(t)\leq \lambda t$.
\end{proof}

Therefore, a simple generalization to the non-compact case, that is
when $\X,\Y$ are complete separable metric spaces, can be achieved
by requiring that $c:\X\times\Y \too \Real_+$ is uniformly
continuous on $\X\times\Y$.
\begin{cor}\label{cor:convergence_optimal_cost}
For $\X,\Y$ complete, separable metric spaces, and $c:\X\times\Y
\too \Real_+$ uniformly continuous $$\T_c(\mu_{S},\nu_T)\To
\T_c(\mu,\nu).$$ Moreover, for Lipschitz $c$ with Lipschitz constant
$\lambda$
$$\abs{\T_c(\mu,\nu) - \T_c(\mu_S,\nu_T) } < 2\lambda h,$$
where, as before, $h=\max \set{\varphi_\X(S),\varphi_\Y(T)}$.
\end{cor}

Note that our argument used only one specific property of Voronoi
cells, namely that each is contained in an $h$-size closed ball.
Many other ways to partition $\X$ and $\Y$ can be considered. The
next Lemma uses a property of the Voronoi cells to show that the
proposed discretization with Voronoi cells is in some sense optimal.
%
%We further want to comment that one could have think about different
%ways to discretize the metric spaces and measures: we only used the
%fact that each Voronoi cell is contained in an $h$-sized closed
%ball, and one can think of different ways to subdivide the space. In
%the next Lemma we use a well known property of the Voronoi cells to
%show that given point sets $S=\set{s_i}_{i=1}^m\subset \X$,
%$T=\set{t_j}_{j=1}^n$, and Lipschitz cost function $c:\X\times\Y
%\too \Real_+$ with constant $\lambda$, the suggested discretization
%with Voronoi cells is optimal in our context in the following sense.

\begin{lem}
Let $c:\X\times\Y\too \Real_+$ be Lipschitz with constant $\lambda$,
and let $S=\set{s_i}_{i=1}^m\subset \X$, $T=\set{t_j}_{j=1}^n
\subset \Y$ be given point sets. Then, among all the choices of
subdividing $\X$ and $\Y$, $\X=\cup_{i=1}^m Q_i$, and
$\Y=\cup_{j=1}^n R_j$, the Voronoi cells $Q_i=V_i, R_j=W_j$ minimize
a bound on the error term:
$$\abs{\int_{\X\times\Y}c\ d\pi - \int_{\X\times\Y}c\ d\Lambda}.$$
\end{lem}
\begin{proof}
$$\abs{\int_{\X\times\Y}c\ d\pi - \int_{\X\times\Y}c\ d\Lambda}
\leq \sum_{i=1}^m \sum_{j=1}^n \int_{Q_i\times R_j}
\abs{c(x,y)-c(s_i,t_j)} d\pi \leq $$
$$
\sum_{i=1}^m \sum_{j=1}^n \int_{Q_i\times R_j} \lambda
\brac{d_\X(x,s_i)+d_\Y(y,t_j)} d\pi = $$ $$ \lambda \sum_{i=1}^m
\int_{Q_i \times \Y} d_\X(x,s_i) d\pi + \lambda \sum_{j=1}^n
\int_{\X\times R_j} d_\X(y,t_j) d\pi=
$$
$$
\lambda \sum_{i=1}^m  \int_{Q_i} d_\X(x,s_i) d\nu + \lambda
\sum_{j=1}^n \int_{R_j} d_\X(y,t_j) d\mu,$$ and it is not hard to
see that the choices of $R_i,Q_j$ that minimizes this last error
bound are the Voronoi cells $R_i=V_i$, and $Q_j=W_j$, where
$\set{V_i}, \set{W_j}$ are the Voronoi cells corresponding to $S,T$,
respectively.
\end{proof}

\end{document}